\documentclass{amsart}
\usepackage{fancyhdr}
\usepackage[T1]{fontenc}
\usepackage[utf8]{inputenc}
\usepackage{lmodern}
\usepackage[english]{babel}
\usepackage{csquotes}
\usepackage{amsmath}
\usepackage{amssymb}
\usepackage{mathtools}
\usepackage[dvipsnames]{xcolor}
\usepackage{graphicx}
\usepackage[hypertexnames=false,colorlinks=true,allcolors=black]{hyperref} 
\usepackage{enumitem}
\usepackage{MnSymbol}

\numberwithin{equation}{section}

\usepackage[noabbrev,capitalize,english]{cleveref}

\theoremstyle{plain}
 \newtheorem{theoremINTERNAL}{Theorem}[section]
 \newtheorem{lemmaINTERNAL}[theoremINTERNAL]{Lemma}
 \newtheorem{propositionINTERNAL}[theoremINTERNAL]{Proposition}
 \newtheorem{corollaryINTERNAL}[theoremINTERNAL]{Corollary}

\theoremstyle{definition}
 \newtheorem{definitionINTERNAL}[theoremINTERNAL]{Definition}
 \newtheorem{exampleINTERNAL}[theoremINTERNAL]{Example}
 \newtheorem{remarkINTERNAL}[theoremINTERNAL]{Remark}
\crefname{theorem}{Theorem}{Theorems}
\crefname{lemma}{Lemma}{Lemmas}
\crefname{proposition}{Proposition}{Propositions}
\crefname{corollary}{Corollary}{Corollaries}
\crefname{definition}{Definition}{Definitions}
\crefname{example}{Example}{Examples}
\crefname{remark}{Remark}{Remarks}
\crefformat{theorem}{Thm.~#2#1#3}
\crefformat{lemma}{Lem.~#2#1#3}
\crefformat{proposition}{Prop.~#2#1#3}
\crefformat{corollary}{Cor.~#2#1#3}
\crefformat{definition}{Def.~#2#1#3}
\crefformat{example}{Ex.~#2#1#3}
\crefformat{remark}{Rem.~#2#1#3}
\crefformat{equation}{(#2#1#3)}

\NewDocumentEnvironment{theorem}{O{} o}
{
 \newcommand*{\TEMPtheorem}{#2}
 \ifblank{#1}
 {\theoremINTERNAL}
 {\theoremINTERNAL[#1]}
 \IfNoValueTF{#2}

 {\label[theorem]{\TEMPtheorem}}
}
{
 \endtheoremINTERNAL
}

\NewDocumentEnvironment{lemma}{O{} o}
{
 \newcommand*{\TEMPlemma}{#2}
 \ifblank{#1}
 {\lemmaINTERNAL}
 {\lemmaINTERNAL[#1]}
 \IfNoValueTF{#2}

 {\label[lemma]{\TEMPlemma}}
}
{
 \endlemmaINTERNAL
}

\NewDocumentEnvironment{proposition}{O{} o}
{
 \newcommand*{\TEMPproposition}{#2}
 \ifblank{#1}
 {\propositionINTERNAL}
 {\propositionINTERNAL[#1]}
 \IfNoValueTF{#2}

 {\label[proposition]{\TEMPproposition}}
}
{
 \endpropositionINTERNAL
}

\NewDocumentEnvironment{corollary}{O{} o}
{
 \newcommand*{\TEMPcorollary}{#2}
 \ifblank{#1}
 {\corollaryINTERNAL}
 {\corollaryINTERNAL[#1]}
 \IfNoValueTF{#2}

 {\label[corollary]{\TEMPcorollary}}
}
{
 \endcorollaryINTERNAL
}

\NewDocumentEnvironment{definition}{O{} o}
{
 \newcommand*{\TEMPdefinition}{#2}
 \ifblank{#1}
 {\definitionINTERNAL}
 {\definitionINTERNAL[#1]}
 \IfNoValueTF{#2}

 {\label[definition]{\TEMPdefinition}}
}
{
 \enddefinitionINTERNAL
}

\NewDocumentEnvironment{example}{O{} o}
{
 \newcommand*{\TEMPexample}{#2}
 \ifblank{#1}
 {\exampleINTERNAL}
 {\exampleINTERNAL[#1]}
 \IfNoValueTF{#2}

 {\label[example]{\TEMPexample}}
}
{
 \endexampleINTERNAL
}

\NewDocumentEnvironment{remark}{O{} o}
{
 \newcommand*{\TEMPremark}{#2}
 \ifblank{#1}
 {\remarkINTERNAL}
 {\remarkINTERNAL[#1]}
 \IfNoValueTF{#2}

 {\label[remark]{\TEMPremark}}
}
{
 \endremarkINTERNAL
}

\DeclarePairedDelimiter{\abs}{\lvert}{\rvert} \DeclarePairedDelimiter{\dpair}{\langle}{\rangle}
\DeclarePairedDelimiter{\graffe}{\{}{\}}
\DeclarePairedDelimiter{\norm}{\|}{\|}

\DeclarePairedDelimiter{\set}{\{}{\}}
\DeclarePairedDelimiter{\tonde}{(}{)}
\usepackage{pict2e}
\newcommand{\xmathpalette}[2]{\mathchoice
 {#1\displaystyle\textfont{#2}}
 {#1\textstyle\textfont{#2}}
 {#1\scriptstyle\scriptfont{#2}}
 {#1\scriptscriptstyle\scriptscriptfont{#2}}
}

\makeatletter

\newcommand{\MRESTRICTION@thickness}[1]{\dimexpr1.5\fontdimen8 #13\relax}

\newcommand{\MRESTRICTIONA}{\mspace{3mu}{\xmathpalette\MRESTRICTIONA@\relax}\mspace{3mu}}
\newcommand{\MRESTRICTIONA@}[3]{\begingroup
 \setlength\unitlength{\dimexpr\fontcharht#21`A-0.5\MRESTRICTION@thickness{#2}}\raisebox{0.5\dimexpr\MRESTRICTION@thickness{#2}}{\begin{picture}(1,1)
 \roundcap\roundjoin
 \linethickness{\MRESTRICTION@thickness{#2}}\polyline(0,1)(0,0)(1,0)
 \end{picture}}\endgroup
}

\newcommand{\MRESTRICTIONB}{\mspace{5mu}{\xmathpalette\MRESTRICTIONB@\relax}\mspace{5mu}}
\newcommand{\MRESTRICTIONB@}[3]{\begingroup
 \setlength\unitlength{\dimexpr0.8\fontcharht#21`A-0.5\MRESTRICTION@thickness{#2}}\raisebox{0.5\dimexpr\MRESTRICTION@thickness{#2}}{\begin{picture}(0.5,1)
 \roundcap\roundjoin
 \linethickness{\MRESTRICTION@thickness{#2}}\polyline(0,1)(0,0)(0.5,0)
 \end{picture}}\endgroup
}

\newcommand{\MRESTRICTIONC}{\mspace{3mu}{\xmathpalette\MRESTRICTIONC@\relax}\mspace{3mu}}
\newcommand{\MRESTRICTIONC@}[3]{\begingroup
 \setlength\unitlength{\dimexpr0.8\fontcharht#21`A-0.5\MRESTRICTION@thickness{#2}}\raisebox{0.5\dimexpr\MRESTRICTION@thickness{#2}}{\begin{picture}(1,1)
 \roundcap\roundjoin
 \linethickness{\MRESTRICTION@thickness{#2}}\polyline(0,1)(0,0)(1,0)
 \end{picture}}\endgroup
}
\makeatother

\NewDocumentCommand{\cardinality}{o}{\IfNoValueT{#1}{\char"0023}
\IfNoValueF{#1}{\char"0023#1}
}
\newcommand*{\littletaller}{\mathchoice{\vphantom{\big|}}{}{}{}}
\NewDocumentCommand{\mres}{}{\MRESTRICTIONC}
\newcommand*{\restr}[2]{\left.\kern-\nulldelimiterspace #1 \littletaller \right|_{#2}}
\NewDocumentCommand{\sommatoria}{o o m}{\IfNoValueT{#1}{\IfNoValueT{#2}{\sum#3}}
\IfNoValueT{#1}{\IfNoValueF{#2}{\sum_{#1}#3}}
\IfNoValueF{#1}{\IfNoValueT{#2}{\sum_{#1}#3}}
\IfNoValueF{#1}{\IfNoValueF{#2}{\sum^{#2}_{#1}#3}}
}

\begin{document}
 \title{Generalized BMO-type seminorms and vector-valued Sobolev functions}
	
	\author[K.~Bessas]{Konstantinos Bessas}
	\address[K.~Bessas]{Dipartimento di Matematica,
 Università di Pavia, Via Adolfo Ferrata 5, 27100 Pavia, Italy}
	\email{konstantinos.bessas@unipv.it}
	
	\author[S.~Guarino Lo Bianco]{Serena Guarino Lo Bianco}
	\address[S.~Guarino Lo Bianco]{Dipartimento di Scienze FIM, Università degli studi di Modena e Reggio Emilia, Via Campi 213/b, 41125 Modena, Italy}
	\email{serena.guarinolobianco@unimore.it}

	\author[R.~Schiattarella]{Roberta Schiattarella}
	\address[R.~Schiattarella]{Dipartimento di Matematica e Applicazioni ”R. Caccioppoli”, Università di Napoli “Federico II”, via Cintia, 80126 Napoli, Italy}
	\email{roberta.schiattarella@unina.it}

    \dedicatory{Dedicated to G.Moscariello on the occasion of her 70\textsuperscript{th} birthday}
	
	\date{\today}
	
	\keywords{BMO-type seminorms, Sobolev spaces, Sobolev functions, Poincaré inequality, Korn inequality, Poincaré-Korn inequality. }
	
	\subjclass[2020]{46E35, 46E40, 26D10, 	30H35}
	
	\thanks{\textit{Acknowledgements}. 
	K.B. has been supported by Fondazione Cariplo, grant n° 2023-0873.\\R.S. has been partially funded by PRIN Project
2022XZSAFN.\\
    The authors are members of “GNAMPA” of Istituto Nazionale di Alta Matematica (INdAM). 
	K.B. is 
    partially supported
	by the INdAM - GNAMPA 2026 Project "Metodi non locali classici e distribuzionali per problemi variazionali",
	codice CUP E53C25002010001.}

	\begin{abstract}
		We establish a pointwise limit theorem for a broad class of pa\-ra\-me\-ter-\-de\-pen\-dent BMO-type seminorms as the parameter tends to zero. By introducing novel BMO-type seminorms, we provide a unified framework that extends several existing results and yields non-distributional characterizations of Sobolev-type spaces, both in the scalar and in the vector-valued setting.
More precisely, for any open set $\Omega\subset \mathbb{R}^n$ and any $p\in (1, \infty)$, we provide a characterization of the Sobolev space $W^{1,p}(\Omega; \mathbb{R}^m)$. In addition, we characterize the space $E^{1,p}(\Omega;\mathbb{R}^n)$ of $L^p$ maps with $p$-integrable distributional symmetric gradient.\\
Finally, for all $p\in [1, \infty)$, we show that these seminorms converge to integral functionals with convex, $p$-homogeneous integrands associated with the distributional gradient and the symmetric gradient.
	\end{abstract}

	\maketitle

	\tableofcontents

\section{Introduction}\label{sec:introduction}
Let $\Omega \subset \mathbb{R}^n$ be an open set. The characterization of Sobolev spaces through integral functionals that do not explicitly involve weak derivatives has been a central theme in analysis and the calculus of variations. Given $u\in L^p(\Omega)$, $p\in [1,+\infty)$, one seeks structural conditions under which
the boundedness of 
suitable non-distributional one-parameter energies
characterizes membership in $W^{1,p}(\Omega)$ and yields, in the limit, the $L^p(\Omega)$ norm of the gradient.\\
A foundational result in this direction is established in
\cite{BBM01-MR3586796} where the classical Sobolev seminorm is encoded in the asymptotical behavior of purely nonlocal difference-quotient energies. For $s\in(0,1)$ define the fractional seminorm
\[
[u]_{W^{s,p}(\Omega)}^p
=
\int_{\Omega}\int_{\Omega}
\frac{|u(x)-u(y)|^p}{|x-y|^{n+sp}}\,\mathrm{d}x\,\mathrm{d}y.
\]
If $1<p<\infty$ and $\Omega$ is bounded with Lipschitz boundary,
then for every $u\in W^{1,p}(\Omega)$
\[
\lim_{s\uparrow 1}
(1-s)[u]_{W^{s,p}(\Omega)}^p
=
K_{n,p}\int_{\Omega} |\nabla u|^p\,\mathrm{d}x,
\]
where $K_{n,p}>0$ depends only on $n$ and $p$.
Conversely, if
\[
\liminf_{s\uparrow 1} (1-s)[u]_{W^{s,p}(\Omega)}^p < +\infty,
\]
then $u\in W^{1,p}(\Omega)$. \\
An alternative formulation, introduced in \cite[Theorem 2]{BBM01-MR3586796}, replaces fractional powers by radial kernels $\rho_\varepsilon$
  concentrating at the origin. In this case one obtains
\begin{equation}\label{fda98y}
    \lim_{\varepsilon \to 0}  \int_\Omega\int_\Omega\frac{|u(y)-u(x)|^p}{|y-x|^p}\rho_\varepsilon(|y-x|)\, \mathrm{d}y\, \mathrm{d}x = K_{p,n}\int_\Omega |\nabla u|^p\, \mathrm{d}x,
\end{equation}
with the convention that the RHS in \eqref{fda98y} is equal to $+\infty$ if $u\notin W^{1,p}(\Omega)$. These results provide a derivative-free characterization of Sobolev spaces based purely on nonlocal quantities.

A conceptually different approach was introduced in \cite{FusMosSbo18-MR3816417}, where $p$-mean oscillation-type functionals inspired by BMO seminorms were considered.  Given a function $u\in L^p_\text{loc}(\Omega)$ (here $\Omega$ is an arbitrary open set) and any $\varepsilon >0$, they define  
$$
\kappa_\varepsilon(u;p;\Omega)\coloneqq\varepsilon^{n-p}\sup_{\mathcal G_\varepsilon} 
 \sum_{Q'\in\mathcal G_\varepsilon}\strokedint_{ Q' } \left| u - \strokedint_{ Q' } u \right|^p\thinspace\mathrm{d} x, 
$$
where the supremum is taken over all families $\mathcal G_\varepsilon$ of disjoint $\varepsilon$-cubes $Q'$ of side length $\varepsilon$ contained in $\Omega$ and arbitrary orientation. They proved that for every $u\in L^p_\text{loc}(\Omega)$,
$$
\nabla u\in L^p(\Omega;\mathbb{R}^n)
\quad\Longleftrightarrow\quad \liminf_{\varepsilon \to 0  }\kappa_\varepsilon(u;p;\Omega)<+\infty.$$
Moreover, they identified the limit
\begin{equation*}
\lim_{\varepsilon\to 0^+} \kappa_\varepsilon(u;p;\Omega) = \gamma(n,p) \| \nabla u\|_{L^p(\Omega)}^p,
\end{equation*}
where $\gamma(n,p)$ is a positive constant that depends only on $n$ and $p$.

Variants of this result were subsequently developed in \cite{FarFusGuaSch20-MR4062330}, \cite{FarGuaSch20-MR4109099, DiFFio20-MR4093787}. In these works two main generalizations are introduced. First, the cubes $Q^\prime$ are replaced by $\varepsilon-$dilation of a bounded connected open set $D$ (called the reference cell). Second, the families of sets $\mathcal G_\varepsilon$ can either be made of sets with arbitrary orientation (isotropic case) or with fixed orientation  (anisotropic case).
When rotations are not allowed, it is  proved in \cite{FarGuaSch20-MR4109099} that by considering  
\[
H_\varepsilon^D(u,p,\Omega)\coloneqq\varepsilon^{n-p}\sup_{\mathcal H_\varepsilon} 
 \sum_{D'\in\mathcal H_\varepsilon}\strokedint_{ D' } \left| u - \strokedint_{ D' } u \right|^p\thinspace\mathrm{d} x, 
\]
where $\mathcal H_\varepsilon$ is any pairwise disjoint family of translations $D^\prime$ of $\varepsilon D$ contained in $\Omega$, there exists a locally Lipschitz continuous $p$-homogeneous function $\psi_p^D: \mathbb R^n\to [0,+\infty)$ such that
\[
\lim_{\varepsilon\to 0^+} H_\varepsilon^D(u,p,\Omega) = \int_{\mathbb R^n} \psi_p^D(\nabla u(x))\, \mathrm{d}x.
\]

In the vector-valued setting, the situation becomes richer.
Given $\Omega\subset \mathbb{R}^n$ open bounded with sufficiently smooth boundary, for vector fields $u: \Omega \to \mathbb R^n$ a non-local limiting integral formula, avoiding the direct manipulation of distributional derivatives, was obtained in the spirit of a representation of the BBM difference quotient for Sobolev functions. 
In particular, in the seminal work \cite{Men12-MR2965673}, the author introduced a nonlocal functional of the form
\[
\Phi_{\varepsilon}^p[u] = \int_\Omega\int_\Omega \frac{|(u(y)-u(x))\cdot ( y-x)|^p}{|y-x|^{2p}}\rho_\varepsilon(y-x)\, \mathrm{d}y\,\mathrm{d}x,
\]
where $\rho_\varepsilon$ are standard radial mollifiers.

A variant of this operator appears in continuum mechanics within the framework of peridynamics, where equations are formulated using integral operators rather than differential ones  (see \cite{MenSpe-MR3299181,DiFSol25-MR4840265, MenDu-MR3424902} and the references therein).
It was proved in \cite[Theorem 2.2]{Men12-MR2965673},  that for $p\in (1,\infty)$ and $u\in L^p(\Omega, \mathbb R^n)$, 
\[
u\in W^{1,p}(\Omega,\mathbb R^n)
\quad\Longleftrightarrow\quad \liminf_{\varepsilon \to 0^+}  \Phi_{\varepsilon}^p[u]<+\infty.
\]
More precisely, 
\[
\lim_{\varepsilon\to 0^+} \Phi_{\varepsilon}^p[u] = \int_\Omega \mathbf{Q}_p(\mathcal Eu(x))\, \mathrm{d}x,
\]
where $$\mathcal Eu:= \frac{\nabla u + \nabla u ^T}{2}
$$ is the symmetric gradient and 
$\mathbf{Q}_p(A)$ is a norm in the space of symmetric matrices defined by
\[
\mathbf{Q}_p(A)= \left((2p+n)\int_{ B_1(0) } |\langle Aw, w\rangle|^p\, \mathrm{d}w\right)^{1/p},
\]
where $B_1(0)\subset\mathbb{R}^n$
is the Euclidean ball centered at zero of radius $1$. The nonlocal energy therefore controls only the symmetric part $\mathcal Eu$ of the gradient and the finiteness of the limiting functional implies $\mathcal Eu\in L^p(\Omega)$. For bounded Lipschitz domains, Korn's inequality yields
$$
W^{1,p}(\Omega; \mathbb R^n)= \left\{ u\in L^p(\Omega; \mathbb R^{n}): \mathcal Eu\in L^p(\Omega; \mathbb R^{n\times n}) \right\},
$$
together with the estimate
\begin{equation}\label{korn-poincare}
\inf_{B\in \mathbb{R}^{n\times n}_{\mathrm{skew}}, b\in \mathbb R^n} \| u-(Bx+b)\|_{W^{1,p}(\Omega; \mathbb R^n)} \leq C(\Omega) \|\mathcal Eu\|_{L^p(\Omega; \mathbb{R}^{n\times n})},
\end{equation} where $\mathbb{R}^{n\times n}_{\mathrm{skew}}$ is the space of real valued skew-symmetric $ n \times n $ matrices (see \cite{DurMus04-MR2108898, Men12-MR2965673}).

This identifies a structural principle: a nonlocal energy may encode only a prescribed linear component of the gradient, while the passage to full Sobolev regularity is enforced by a Korn-type inequality. 
This strategy is possible due to the regularity assumptions on $\Omega$. 

In our work, we do not require any regularity assumptions on $\Omega$ considering an arbitrary open set (in the spirit of \cite{FusMosSbo18-MR3816417,FarGuaSch20-MR4109099,DiFFio20-MR4093787}) and we introduce a broad class of one-parameter families of BMO-type seminorms originated by suitable functionals, which we call \emph{$p$-core functionals} (see \cref{def:core_function}). 
Through our definition, we can recover the classical scalar $p$-mean oscillation functionals available in the literature and introduce novel ones, defined also for vector-valued functions.

The goal of this paper is to present a general and unifying theorem ensuring the existence of
the pointwise limit of these one-parameter families as the parameter goes to zero.
We will also prove that, when this limit functional is evaluated at a generic function $u$, it admits an integral representation determined by the value that the limit functional itself assumes on linear maps.

More precisely, given a bounded open set $D\subset\mathbb{R}^n$ and a group $\Gamma$ of affine transformations on $\mathbb{R}^n$, a $p$-core functional $\alpha_p$ 
is a positive functional defined on pairs of the form
$(u,D')$, where $u\in L^p_\text{loc}(\mathbb{R}^n;\mathbb{R}^m)$ and $D'\subset\mathbb{R}^n$ is a transformation of $D$ according to $\Gamma$, satisfying certain axioms. These include $p$-homogeneity and convexity in the first variable and the validity of the following crucial inequality:
\begin{equation}\label{poinc}
\exists c_{1} >0
 \text{ s.t. }
 \alpha_{ p}\tonde*{ u , D } 
 \leq
  c_{1} \int_{ D } \abs*{ \nabla u}^{ p }_{2} \thinspace\mathrm{d} x 
 \text{ }
 \forall u \in W^{1,p}\tonde*{D;\mathbb{R}^{m}}.
\end{equation}
We observe that in the $p$-mean oscillation case, assuming also $D$ Lipschitz and connected, \eqref{poinc} corresponds to the Poincaré inequality.

In defining BMO-type seminorms originated by $p$-core functionals,
we introduce a further level of generality compared to \cite{FarFusGuaSch20-MR4062330,FarGuaSch20-MR4109099, DiFFio20-MR4093787} by allowing the presence of an \enquote{intermediate rotation} group $ \mathcal{G} $ between $\set*{ \operatorname{id}_{\mathbb{R}^{n}}}$ and $ \mathrm{SO}\tonde*{n} $.
More precisely, given an open set $ \Omega \subseteq \mathbb{R}^{n} $ (which we call the \emph{ambient space}) and a subgroup $ \mathcal{G} \subset \mathrm{SO}\tonde*{n}$, we denote by $ \mathcal{G}^{ D}_{\varepsilon}\tonde*{\Omega} $  the class of all the subsets of $\Omega$ obtained by translating, $\varepsilon$-dilating and rotating, according to the rotations in $\mathcal{G}$, the reference set $D$, and by $\mathcal G^D\tonde*{\Omega}=\bigcup_{ \varepsilon>0} \mathcal{G}^{ D}_{\varepsilon}\tonde*{\Omega}$ (see \cref{def:Gamma_G}). For every $ u \in L^{p}_{\mathrm{loc}}\tonde*{\mathbb{R}^{n};\mathbb{R}^{m}} $ we   define the following general \emph{BMO-type seminorm}
 \begin{equation*}
 G^{\alpha_{p}}_{ \varepsilon }\tonde*{ u,\Omega} 
 \coloneqq
 \varepsilon^{ n-p } \sup\limits_{\mathcal{G}_{\varepsilon}} \sum_{ D'\in\mathcal{G}_{\varepsilon}} \alpha_{ p}\tonde*{ u , D' } ,
 \end{equation*}
 where each $ \mathcal{G}_{\varepsilon} $ is a family made of pairwise disjoint elements of $ \mathcal{G}^{ D}_{\varepsilon}\tonde*{\Omega} $ 
 and the supremum is taken with respect to all of these families $ \mathcal{G}_{\varepsilon} $.
Our main result is the following.

\begin{theorem}[][thm:Main_BMO]
 Let $ \Omega \subseteq \mathbb{R}^{n} $ be an open set, $ m \geq1$ a natural number, $ p \in \left[1,\infty\right) $,
 $ \mathcal{G} $ a subgroup of $ \mathrm{SO}\tonde*{n} $ and $ D \subseteq \mathbb{R}^{n} $ be a bounded open set.
Let $ \alpha_{ p} $ be a $p$-core functional
 with domain
 $L^{p}_{\mathrm{loc}}\tonde*{\mathbb{R}^{n};\mathbb{R}^{m}} \times \mathcal{G}^{ D}\tonde*{\mathbb{R}^{n}} $ and $ u \in L^{p}_{\mathrm{loc}}\tonde*{\mathbb{R}^{n};\mathbb{R}^{m}} $.

 Then, if $ \nabla \restr{ u }{ \Omega } \in L^{p}\tonde*{\Omega;\mathbb{R}^{m\times n}} $,
 \begin{equation}\label{eq:tesi_Main_BMO}
 \lim\limits_{\varepsilon\to 0^{+}} G^{ \alpha_{ p}}_{ \varepsilon }\tonde*{ u,\Omega } = \int_{ \Omega } \psi_{\alpha_{p}}\tonde*{ \nabla u } \thinspace\mathrm{d} x ,
 \end{equation}
 where $\psi_{\alpha_{p}}$ is a convex and $p$-homogeneous function defined by
 \begin{equation*}
 \psi_{\alpha_{p}}\tonde*{A} \coloneqq
 \lim\limits_{\varepsilon\to 0^+}G^{ \alpha_{ p}}_{\varepsilon}\tonde*{ l^{ A} , \tonde*{-\frac{1}{2},\frac{1}{2}}^n } 
 \end{equation*} 
for every $ A \in \mathbb{R}^{m\times n}$ and
  $ l^{ A}\tonde*{ x } \coloneqq A x $ for every $ x \in \mathbb{R}^{n} $.

 Moreover, if $\restr{ u }{ \Omega }\in W^{1,p}_{\mathrm{loc}}\tonde*{\Omega;\mathbb{R}^{m}} $ and $ \psi_{\alpha_{p}} \not\equiv0$,
 \eqref{eq:tesi_Main_BMO} still holds.
\end{theorem}

In order to prove Theorem \ref{thm:Main_BMO}  we follow 
the strategy introduced in \cite[Theorem 4.1]{FarFusGuaSch20-MR4062330} and \cite[Theorem 1.2]{FarGuaSch20-MR4109099}.
We highlight that in the cases considered in the literature, the condition $ \psi_{\alpha_{p}} \not\equiv0$ is always satisfied (actually, $\psi_{\alpha_{p}}>0$ on $\mathbb{R}^{m\times n}\setminus\set*{0}$ ).
However, we remark that in general when $\restr{ u }{ \Omega }\in W^{1,p}_{\mathrm{loc}}\tonde*{\Omega;\mathbb{R}^{m}}$ the condition $\psi_{\alpha_p}\not\equiv 0$ is necessary.
Indeed,
 if $\restr{ u }{ \Omega }\in W^{1,p}_{\mathrm{loc}}\tonde*{\Omega;\mathbb{R}^{m}} $ with $ \nabla \restr{ u }{ \Omega } \notin L^{p}\tonde*{\Omega;\mathbb{R}^{m\times n}} $
 and $ \psi_{\alpha_{p}} \equiv0$ then \eqref{eq:tesi_Main_BMO} may or may not hold (see \cref{ex:trivial} and \cref{ex:counterex_psi_null}).

In \cref{thm:Main_BMO}, under suitable assumptions on the $p$-core functional, $\psi_{\alpha_{p}}$ can be written more explicitly as in \cref{res:formula_psi}. In \cref{sec:applications} we provide several examples for which  $\psi_{\alpha_p}$ can be computed explicitly. 

A further structural property of the energy arises by considering the decomposition of the matrix space  $\mathbb{R}^{m\times n}=\mathcal{P} \oplus \mathcal{N}_{\psi_{\alpha_{p}}} $, where $\mathcal N_{\psi_{\alpha_p}}=\{A\in \mathbb{R}^{m\times n} \,:\, \psi_{\alpha_p}(A)=0\}$ and $\mathcal{P}$ is a linear subspace of $\mathbb{R}^{m\times n}$
 (see \cref{res:psi_alphap_Pspace}).
Let $\pi_P:\mathbb{R}^{m\times n} \to \mathcal{P}$ denote the projection onto $\mathcal{P}$ associated to the direct sum decomposition. Then, for every  $ A \in \mathbb{R}^{m\times n} $, it holds
 \begin{equation*}
  \psi_{\alpha_{p}}\tonde*{ \pi_{\mathcal{P}}\tonde*{ A }}=\psi_{\alpha_{p}}\tonde*{ A }.
 \end{equation*}
 Moreover, the norm induced by  $\psi_{\alpha_{p}}$ is equivalent to the $L^p$-norm of  $\pi_{\mathcal P}(\nabla u)$. More precisely, there exist constants $c, C>0$ depending on $\alpha_p$ such that
 \[
 c\int_\Omega |\pi_{\mathcal P}(\nabla u)|_2^p\, \mathrm{d}x \leq \int_\Omega \psi_{\alpha_p}(\pi_{\mathcal P}(\nabla u))\, \mathrm{d}x\leq C\int_\Omega |\pi_{\mathcal P}(\nabla u)|_2^p\, \mathrm{d}x.
 \]
 If $\psi_{\alpha_{p}}>0$ on $ \mathbb{R}^{m\times n} \setminus\set*{0}$, then $\mathcal{N}_{\psi_{\alpha_p}}=\{0\}$ and the above inequality yields the equivalence
 \begin{equation*}
 \int_\Omega \psi_{\alpha_p}(\pi_{\mathcal P}(\nabla u))\, \mathrm{d}x \approx \int_\Omega |\nabla u|_2^p\, \mathrm{d}x.
 \end{equation*}
 In this case, assuming also a lower semicontinuity property on $\alpha_p$ we obtain the characterization
  \begin{equation}\label{intro:charat}
  \liminf_{\varepsilon\to 0^+} G^{ \alpha_{ p}}_\varepsilon \tonde*{ u,\Omega }<+\infty \quad\Longleftrightarrow\quad \nabla \restr{ u }{ \Omega } \in L^{p}\tonde*{\Omega;\mathbb{R}^{m\times n}}
 \end{equation}
 (see \cref{res:cor_Main_BMO2}). Examples of this situation are discussed in  \cref{ex:fun_meno_media} and \cref{ex:integ_doppio} (see also \cref{res:W1p_char}), where the functional $\alpha_p$ corresponds respectively to the $p$-mean oscillation functional 
 \[
 \alpha_p(u,D^\prime)= \strokedint_{D^\prime} \left|u(x)-\strokedint_{ D' } u\right|_2^p\, \mathrm{d}x,
 \]
 and to the functional
$$\alpha_p(u,D^\prime)=\strokedint_{D^\prime}\strokedint_{D^\prime} \left| u(x)-u(y)\right|_2^p\, \mathrm{d}x\, \mathrm{d}y.
 $$

 However, in general, the equivalence in \cref{intro:charat} does not hold. Nevertheless, thanks to \cref{res:cor_Main_BMO}, assuming a lower semicontinuity property on $\alpha_p$ and denoting by $u_\sigma=\rho_\sigma*u$ the standard mollification of $u$, it is still possible to obtain the estimate 
\begin{equation}\label{diseq}
 \begin{split}
  c\limsup\limits_{\sigma\to 0^{+}} \int_{ \tilde{\Omega}} \abs*{ \pi_{\mathcal{P}}\tonde*{ \nabla u_{\sigma}}}_{2}^{p} \thinspace\mathrm{d} x 
  &\leq
   \lim_{\varepsilon \to 0^+}G_{\varepsilon}^{ \alpha_{ p}}\tonde*{ u_\sigma,\tilde\Omega } 
 \\&\leq
 \liminf_{\varepsilon\to 0^+} G^{ \alpha_{ p}}_{\varepsilon}\tonde*{ u,\Omega } 
  \leq 
  C \int_\Omega |\nabla u|^p_2\, \mathrm{d}x,
  \end{split}
 \end{equation}
 where $ \tilde{\Omega}\subset\mathbb{R}^n$ is an open set such that $ \tilde{\Omega} \Subset \Omega $.
 Assuming $\liminf_{\varepsilon\to 0^+}G^{ \alpha_{ p}}_\varepsilon\tonde*{ u,\Omega }<+\infty$, since $ u_{\sigma} \to u $ in $ L^{p}_{\mathrm{loc}}\tonde*{\mathbb{R}^{n};\mathbb{R}^{m}} $ as $\sigma\to 0^{+} $, one can only conclude that 
 $\pi_{\mathcal{P}}\tonde*{\nabla \restr{ u_{\sigma}}{\Omega}} \to D_{\pi_{\mathcal{P}}}\tonde*{\restr{u}{\Omega }}$  as $\sigma\to 0^{+}$ in the sense of distributions (here $D_{\pi_{\mathcal{P}}}\tonde*{\restr{u}{\Omega}}$ is the distribution associated to the projection of the gradient according to $\pi_\mathcal{P}$, see  \cref{sec:notation} for the precise definition).
 This leads to the implication
  \begin{equation}\label{unaimpl}
  \liminf_{\varepsilon\to 0^+}G^{ \alpha_{ p}}_\varepsilon\tonde*{ u,\Omega }<+\infty
 \quad\Longrightarrow\quad
  D_{ \pi_{\mathcal{P}}}\tonde*{\restr{ u }{ \Omega }} \in L^{p}\tonde*{\Omega;\mathbb{R}^{m\times n}}. 
  \end{equation}
  The converse implication does not hold in general, as shown in \cref{ex:counterex_psi_null},
  \[
  \alpha_p(u,D^\prime)=  \inf\limits_{ A \in \mathbb{R}^{m\times n}} 
  \strokedint_{ D'} \abs*{ u\tonde*{ x }-A \tonde*{ x-\operatorname{bar}\tonde*{ D' }}-\strokedint_{ D' } u
  }^{ p }_{2} \thinspace\mathrm{d} x ,
  \]
  where $\operatorname{bar}\tonde*{ D' }$ is the barycenter of $D'$. 
  Indeed, in this case $\lim_{\varepsilon\to 0^+}G^{ \alpha_{ p}}_\varepsilon\tonde*{ u,\Omega }=+\infty$ but $\psi_{\alpha_p}\equiv 0$ implies automatically $D_{\pi_\mathcal{P}}(u)=0\in L^p(\Omega,\mathbb{R}^{m\times n})$.

 To recover the reverse implication in \eqref{unaimpl}, we introduce the concept of \emph{strong} $p$-core functional, i.e. a $p$-core functional satisfying a  stronger version of \eqref{poinc}. More precisely, we require that 
  \[
  \exists  c_{2} >0
 \text{ s.t. }
 \alpha_{ p}\tonde*{ u , D } 
 \leq
 c_{2} \int_{ D }\abs*{D_{\pi_{\mathcal{P}}} u }^{ p }_{2} \thinspace\mathrm{d} x 
 \text{ }
 \forall u \in W^{\pi_{\mathcal{P}},p}\tonde*{D;\mathbb{R}^{m}},
  \]
where $W^{\pi_{\mathcal{P}},p}\tonde*{D;\mathbb{R}^{m}}=\set*{ u \in L^{p}\tonde*{D;\mathbb{R}^{m}} : D_{\pi_{\mathcal{P}}} u \in L^{p}\tonde*{D;\mathbb{R}^{m\times n}}}$.
  Under this assumption, the upper bound in the previous estimate \eqref{diseq} improves to 
  \[
  \liminf_{\varepsilon\to 0^+}G^{ \alpha_{ p}}_{\varepsilon}\tonde*{ u,\Omega } 
  \leq 
  c_2 \int_\Omega |D_{\pi_\mathcal{P}} u|^p_2\, \mathrm{d}x,
  \]
  which implies that 
  \begin{equation}\label{dueimpl}
      D_{ \pi_{\mathcal{P}}}\tonde*{\restr{ u }{ \Omega }} \in L^{p}\tonde*{\Omega;\mathbb{R}^{m\times n}} \qquad \Longrightarrow\qquad \liminf_{\varepsilon\to 0^+}G^{ \alpha_{ p}}_\varepsilon\tonde*{ u,\Omega }<+\infty
  \end{equation}
  (see \cref{res:cor_Main_BMO2_WL}).

  In this framework we obtain the following analogue of \cref{thm:Main_BMO}. 
 
\begin{theorem}[][thm:Main_BMO_WL]
 Let $ \Omega \subseteq \mathbb{R}^{n} $ be an open set, $ m \geq1$ a natural number, $ p \in \left[1,\infty\right) $,
 $ \mathcal{G} =\set*{ \operatorname{id}_{\mathbb{R}^{n}}}$ and $ D \subseteq \mathbb{R}^{n} $ be a bounded open set.
 Let $ \alpha_{ p} $ be a strong $p$-core functional
 with domain
 $ L^{p}_{\mathrm{loc}}\tonde*{\mathbb{R}^{n};\mathbb{R}^{m}} \times \mathcal{G}^{ D}\tonde*{\mathbb{R}^{n}} $ and $ u \in L^{p}_{\mathrm{loc}}\tonde*{\mathbb{R}^{n};\mathbb{R}^{m}} $.

 Then, if $ D_{ \pi_{\mathcal{P}}}\tonde*{\restr{ u }{ \Omega }} \in L^{p}\tonde*{\Omega;\mathbb{R}^{m\times n}} $,
 \begin{equation}\label{eq:tesi_Main_BMO_WL}
 \lim\limits_{\varepsilon\to 0^{+}} G^{ \alpha_{ p}}_{ \varepsilon }\tonde*{ u,\Omega } = \int_{ \Omega } \psi_{\alpha_{p}}\tonde*{ D_{ \pi_{\mathcal{P}}}\tonde*{\restr{ u }{ \Omega }}} \thinspace\mathrm{d} x .
 \end{equation}

 Moreover, if $\restr{ u }{ \Omega }\in W^{\pi_{\mathcal{P}},p}_{\mathrm{loc}}\tonde*{\Omega;\mathbb{R}^{m}} $ and $ \psi_{\alpha_{p}} \not\equiv0$,
 \eqref{eq:tesi_Main_BMO_WL} still holds.
\end{theorem}

We remark that in \cref{thm:Main_BMO_WL} we assume that $ \mathcal{G} =\set*{ \operatorname{id}_{\mathbb{R}^{n}}}$ because, a priori, the differential operator $ D_{ \pi_{\mathcal{P}}}$ imposes more rigidity conditions than the canonical distributional gradient
$\nabla$.

In \cref{sec:applications} we present several examples of $p$-core and strong $p$-core functionals.   We observe that when $\mathcal{N}_{\psi_{\alpha_p}}=\{0\}$ the two notions coincide, while if $\mathcal{N}_{\psi_{\alpha_p}}=\mathbb{R}^{m\times n}$ no strong $p$-core functional exists (except the trivial case of $\alpha_p\equiv0$ in \cref{ex:trivial}). Intermediate situations are illustrated in \cref{ex:inf_antisym}, where
\[
\alpha_{ p}\tonde*{ u , D' } 
  \coloneqq
  \inf\limits_{ A \in \mathbb{R}^{n\times n}_{\mathrm{skew}}} 
  \strokedint_{ D'} \abs*{ u\tonde*{ x }-A \tonde*{ x-\operatorname{bar}\tonde*{ D' } }-\strokedint_{ D' } u}^{ p }_{2} \thinspace\mathrm{d} x,
\]
and one finds that
\[
\mathcal{N}_{\psi_{\alpha_p}}=\ \mathbb{R}^{n\times n}_{\mathrm{skew}}, \qquad D_{\pi_{\mathcal{P}}}( \restr{ u }{ \Omega })=\mathcal{E}u.
\]
As a consequence in this case thanks to \eqref{unaimpl} and \eqref{dueimpl} we are able to obtain the following  (cfr. \cref{res:WL_antisym}):
if $ p \in \left(1,\infty\right) $, for every $ u \in L^{p}\tonde*{\Omega;\mathbb{R}^{n}} $,
 \begin{equation*} 
 \liminf\limits_{\varepsilon\to 0^{+}} G^{\alpha_{p}}_{ \varepsilon }\tonde*{ u,\Omega } <+\infty\iff u \in E^{1,p}\tonde*{\Omega;\mathbb{R}^{n}},
 \end{equation*}
where the space $E^{1,p}(\Omega;\mathbb{R}^n)\coloneqq\set*{ u \in L^{p}\tonde*{\Omega;\mathbb{R}^{n}} : \mathcal{E} u \in L^{p}\tonde*{\Omega;\mathbb{R}^{n\times n}}}$.  

Finally, as a consequence of \cref{thm:Main_BMO} and \cref{thm:Main_BMO_WL}, the detour in the characterization of Sobolev spaces leads us to obtain a version of the so-called "Constancy Theorem" in the spirit of  \cite{Bre02-MR1942116}. This class of results is presented as open problem in \cite{Bre02-MR1942116} and it is instrumental in related topics  such as the degree theory for classes of discontinuous maps. 

In our abstract setting for a connected open set $\Omega$, we obtain for $p\in(1,\infty)$ (see \cref{res:constancy}) that assuming a lower semicontinuity property on a strong $p$-core functional $\alpha_p$ then
\[
\liminf_{\varepsilon\to 0^+}G_\varepsilon^{ \alpha_{ p}}\tonde*{ u,\Omega } = 0 \quad\Longrightarrow\quad D_{\pi_\mathcal{P}}(\restr{ u }{ \Omega })=0.
\]
Thus, if $\psi_{\alpha_p}>0$ in $\mathbb{R}^{m\times n}$, then $u$ is constant a.e. on $\Omega$. On the other hand if we consider the $p$-core functional {} in \cref{ex:inf_antisym} for $ p \in \left(1,\infty\right) $ we obtain that if $ u \in L^{p}_{\mathrm{loc}}\tonde*{\mathbb{R}^{n};\mathbb{R}^{n}} $ satisfies $\liminf_{\varepsilon\to 0^+} G^{ \alpha_{ p}}_{\varepsilon}\tonde*{ u,\Omega }=0$,
then $\restr{ u }{ \Omega }\in E^{1,p}\tonde*{\Omega;\mathbb{R}^{n}}$ and $ \mathcal{E}\tonde*{\restr{ u }{ \Omega }}=0$. Then, by the Korn-Poincaré inequality \eqref{korn-poincare}
we find that $u(x)$ identify a rigid motion, i.e. $ u\tonde*{ x }=A x+h $ for every $ x \in \mathbb{R}^{n} $, where $ A \in \mathbb{R}^{n\times n}_{\mathrm{skew}} $ and $ h \in \mathbb{R}^{n} $.

\subsection{Plan of the paper}
The plan of the paper is the following.

In \cref{sec:notation} we set the notation and we collect preliminary results.

In \cref{sec:genaralBMO} 
we introduce the notion of $p$-core functional (cfr. \cref{def:core_function}) and
define the associated 
one-parameter family of generalized BMO-type seminorms
(see \cref{def:tipo_BMO}).
After establishing the main properties of these 
generalized BMO-type seminorms,
for any $p$-core functional $\alpha_p$
we define the $p$-homogeneous and convex integrand
$\psi_{\alpha_p}$
(see \cref{def:psi_alphap}) and we study its main properties.
Then, we prove \cref{thm:Main_BMO} and,
after introducing the concept of strong $p$-core functional (cfr. \cref{def:strong_core_function}), \cref{thm:Main_BMO_WL}.

In \cref{sec:characterizations},
we discuss some consequences of \cref{thm:Main_BMO} and \cref{thm:Main_BMO_WL}.
We first prove \cref{res:cor_Main_BMO2}, involving $p$-core functionals, highlighting that the implication
in its thesis cannot be reversed in general
(cfr. \cref{res:optimality}).
Therefore, after reinforcing its hypotheses,
we then prove \cref{res:cor_Main_BMO2_WL}, the characterization result involving strong $p$-core functionals.
We finally deduce the constancy theorem \cref{res:constancy}.

In \cref{sec:applications},
we present several examples 
of (strong) $p$-core functionals
to which we apply the results proved in \cref{sec:genaralBMO} and \cref{sec:characterizations}.
In particular, we show how these examples can be used to characterize through a non-distributional approach the spaces $W^{1,p}(\Omega;\mathbb{R}^m)$ and 
$E^{1,p}(\Omega;\mathbb{R}^n)$ (see \cref{res:W1p_char} and \cref{res:WL_antisym}).

\section{Notation and preliminary results}\label{sec:notation}

We fix some notation. We denote by $\cardinality[ S ]$ the cardinality of a set $ S $. We denote by $ \abs*{\cdot}_{2} $ the Euclidean norm of $ \mathbb{R}^{n} $ and if $ x , y \in \mathbb{R}^{n} $ their standard scalar product
is denoted by $ x \cdot y $. Given $ x \in \mathbb{R}^{n} $ and $ r >0$, we let $ B_{r}\tonde*{x} $ be the Euclidean ball of radius $ r $ centered at $x$,  $ Q \coloneqq \left(-\frac{1}{2},\frac{1}{2}\right) ^ n $ and $ Q\tonde*{x;r} \coloneqq r Q + x $.
If $ \Omega \subseteq \mathbb{R}^{n} $ is an open set, we denote by $ \mathcal{A}_{\Omega} $
the family of all the open subsets of $ \Omega $.

Let $ \mathrm{Aff}\tonde*{\mathbb{R}^{n}} $ be the affine group of the Euclidean space $ \mathbb{R}^{n} $, i.e. the
set of all invertible affine transformations from $ \mathbb{R}^{n} $ into itself.

We denote by $ \mathbb{R}^{m\times n} $ the space of real valued $ m \times n $ matrices, by $ \mathbb{R}^{n\times n}_{\mathrm{skew}} $,  the space of real valued skew-symmetric $ n \times n $ matrices, by $ \mathrm{O}\tonde*{n} \subset \mathbb{R}^{n\times n} $ the group of all orthogonal matrices
and by $ \mathrm{SO}\tonde*{n} \subset \mathrm{O}\tonde*{n} $ the group of all orthogonal matrices with determinant equal to $1$.

If $ \mathcal{V}_{1} , \mathcal{V}_{2} $ are linear subspaces of $ \mathbb{R}^{m\times n} $, we write $ \mathbb{R}^{m\times n} = \mathcal{V}_{1} \oplus \mathcal{V}_{2} $
if $ \mathbb{R}^{m\times n} = \mathcal{V}_{1} + \mathcal{V}_{2} $ and $ \mathcal{V}_{1} \cap \mathcal{V}_{2} =\set*{0}$.

For $ A \in \mathbb{R}^{m\times n} $ we put,
\begin{equation*}
 \abs*{ A }_{2} 
 \coloneqq
 \sup\limits_{ } \set*{ \abs*{ A x }_{2} : x \in \mathbb{R}^{n} , \abs*{ x }_{2} \leq1}.
\end{equation*}
Given $ A \in \mathbb{R}^{m\times n} $, we denote the linear function associated to $ A $
by $ l^{ A}\tonde*{ x } \coloneqq A x $ for every $ x \in \mathbb{R}^{n} $.

It easy to verify that for $ A , B \in \mathbb{R}^{m\times n} $ and $ p \in \left[1,\infty\right) $ it holds,
\begin{equation}\label{eq:differenze_p}
 \abs*{ \abs*{ A }^{ p }_{2} - \abs*{ B }^{ p }_{2}}\leq p \tonde*{ \abs*{ A }_{2} + \abs*{ B }_{2}}^{ p -1} \abs*{ A - B }_{2} ,
\end{equation}
and for $ A , B \in \mathbb{R}^{m\times n} \setminus\set*{0}$ it holds,
\begin{equation}\label{eq:diff_normalizzati}
 \abs*{ \frac{ A }{ \abs*{ A }_{2}}-\frac{ B }{ \abs*{ B }_{2}}}_{2} 
 \leq2
 \frac{
 \abs*{ A - B }_{2} 
 }{
 \abs*{ A }_{2} 
 }.
\end{equation}

We denote by $ \mathcal{L}^{n} $ the Lebesgue measure on $ \mathbb{R}^{n} $
and for a (Lebesgue) measurable set $ S \subseteq \mathbb{R}^{n} $
we put $ \abs*{ S } \coloneqq \mathcal{L}^{n}\tonde*{ S } $.
For a measurable set $ S \subseteq \mathbb{R}^{n} $ of finite and strictly positive measure
and an integrable function $ u : S \to \mathbb{R}^{d} $ we put,
\begin{equation*}
 u_{ S } 
 =
 \strokedint_{ S } u \thinspace\mathrm{d} x 
 \coloneqq
 \frac{1}{ \abs*{ S }} \int_{ S } u \thinspace\mathrm{d} x .
\end{equation*}

If $ S $ is a measurable set of strictly positive and finite measure,
we define its \emph{barycenter} by,
\begin{equation*}
 \operatorname{bar}\tonde*{ S } \coloneqq \strokedint_{ S } x \thinspace\mathrm{d} x .
\end{equation*}

Throughout the paper we will always denote by $ \graffe*{ \rho _{ \sigma }}_{ \sigma >0} \subset C^{\infty}_{c}\tonde*{\mathbb{R}^{n}} $
a family of standard radial mollifiers, i.e.,
\begin{equation*}
 \rho _{ \sigma }\tonde*{ x } \coloneqq \sigma ^{- n } \rho \tonde*{ x / \sigma } \text{ }\forall x \in \mathbb{R}^{n} ,
\end{equation*}
where $ \rho \in C^{\infty}_{c}\tonde*{\mathbb{R}^{n}} $ is a standard radial mollifier with $ \operatorname{supp}\tonde*{ \rho } \subset B_{1}\tonde*{0} $.

The following lemma is a useful Jensen-type inequality.
\begin{lemma}[][res:jensen_alpha_p]
 Let $ p \in \left[1,\infty\right) $, $ \sigma >0$, $ u \in L^{p}\tonde*{\mathbb{R}^{n};\mathbb{R}^{m}} $ and $ f : L^{p}\tonde*{\mathbb{R}^{n};\mathbb{R}^{m}} \to \left[0,+\infty\right) $
 be a proper convex and lower semicontinuous function.
 Then,
 \begin{equation}\label{eq:jensen_thesis}
 f\tonde*{ u_{ \sigma }} \leq \int_{ B_{1}\tonde*{0}} f\tonde*{ u\tonde*{\cdot- \sigma y }} \rho \tonde*{ y } \thinspace\mathrm{d} y ,
 \end{equation}
 where $ u_{ \sigma } \coloneqq \rho _{ \sigma } \ast u $.
\end{lemma}
\begin{proof}
 Let $ \sigma >0$ and $ u_{ \sigma } \coloneqq \rho _{ \sigma } \ast u \in L^{p}\tonde*{\mathbb{R}^{n};\mathbb{R}^{m}} $.
 We consider the probability space $\tonde*{ \mathbb{R}^{n} , \mathcal{A} , \mu }$,
 where $ \mu \coloneqq \rho \mathcal{L}^{n} = \rho \mathcal{L}^{n} \mres B_{1}\tonde*{0} $ and $ \mathcal{A} $ is the $ \sigma $-algebra
 of Lebesgue measurable sets of $ \mathbb{R}^{n} $.

 We also define $ X \coloneqq L^{p}\tonde*{\mathbb{R}^{n};\mathbb{R}^{m}} $ and consider the measurable space 
 $\tonde*{ X , \mathcal{B}}$, where $ \mathcal{B} $ is the 
 $ \sigma $-algebra of Borel sets of $ X $.
 We recall that $ X $, being a Banach space, is \emph{evenly convex}
 (see the definition for instance in \cite[Sec. 2]{Ves17-MR3742489}).
 We define,
 \begin{align*}
 G : \mathbb{R}^{n} &\longrightarrow X 
 \\
 y &\mapsto u\tonde*{.- \sigma y } .
 \end{align*}
 Moreover, thanks to the $ L^{p} $-continuity of translations, $ G $ is continuous.
 This then implies that $ G :\tonde*{ \mathbb{R}^{n} , \mathcal{A}}\to\tonde*{ X , \mathcal{B}}$
 is measurable.
 We consider the following \emph{Pettis integral}
 (see the definition for instance in \cite[Sec. 2]{Ves17-MR3742489}),
 \begin{equation}\label{pettis}
 \int_{ \mathbb{R}^{n}} G \thinspace\mathrm{d} \mu \in X .
 \end{equation}
 In particular \eqref{pettis} is well-defined, since for every $ \phi \in X^{\prime} \simeq L^{ p^{\prime}}\tonde*{\mathbb{R}^{n};\mathbb{R}^{m}} $
 the map $ \phi \circ G :\tonde*{ \mathbb{R}^{n} , \mathcal{A}}\to \mathbb{R} $ is Lebesgue $\mu$-integrable
 and there exists $ v \in X $ such that
 \begin{equation*}
 \int_{ \mathbb{R}^{n}} \phi \circ G \thinspace\mathrm{d} \mu = \phi \tonde*{ v } \text{ }\forall\phi\in X^{\prime}.
 \end{equation*}
 Indeed, fixed $ \phi \in X^{\prime} $ (and letting $ \tilde{ \phi } \in L^{ p^{\prime}}\tonde*{\mathbb{R}^{n};\mathbb{R}^{m}} $ be the function representing $ \phi $)
 for every $ y \in \mathbb{R}^{n} $,
 \begin{align*}
 \tonde*{ \phi \circ G }\tonde*{ y }
 =
 \int_{ \mathbb{R}^{n}} \tilde{ \phi } \cdot G\tonde*{ y } \thinspace\mathrm{d} \mathcal{L}^{n} 
 =
 \int_{ \mathbb{R}^{n}} \tilde{ \phi }\tonde*{ x } \cdot u\tonde*{ x - \sigma y } \thinspace\mathrm{d} x .
 \end{align*}
 Moreover, applying Tonelli's theorem
 \begin{align*}
 \int_{ \mathbb{R}^{n}} \tonde*{ \phi \circ G }\tonde*{ y } \thinspace\mathrm{d} \mu\tonde*{ y } 
 =
 \int_{ \mathbb{R}^{n}} \int_{ \mathbb{R}^{n}} \tilde{ \phi }\tonde*{ x } \cdot u\tonde*{ x - \sigma y } \rho \tonde*{ y } \thinspace\mathrm{d} x \mathrm{d} y 
 =
 \int_{ \mathbb{R}^{n}} \tilde{ \phi }\tonde*{ x } \cdot u_{ \sigma }\tonde*{ x } \thinspace\mathrm{d} x .
 \end{align*}
 This implies that 
 \begin{equation}\label{pettis2}
 \int_{ \mathbb{R}^{n}} G \thinspace\mathrm{d} \mu = u_{ \sigma } \in X .
 \end{equation}

 By Jensen's inequality (\cite[Thm. 3]{Ves17-MR3742489}), we get that
 \begin{equation}\label{jensen}
 f\tonde*{ \int_{ \mathbb{R}^{n}} G \thinspace\mathrm{d} \mu } \leq \int_{ \mathbb{R}^{n}} f \circ G \thinspace\mathrm{d} \mu .
 \end{equation}
 Combining \eqref{pettis2} and \eqref{jensen} we obtain \eqref{eq:jensen_thesis}.
\end{proof}

\begin{lemma}[][res:null_cx_p_hom]
Let $ X $ be a real vector space, $ p\geq1$ and $ f : X \to \left[0,+\infty\right) $
be a convex and $ p $-homogeneous function:
\begin{equation}\label{dsa980h}
 f\tonde*{ t u } =\abs{ t }^{ p } f\tonde*{ u } 
 \text{ }
 \forall u \in X ,\forall t \in \mathbb{R} .
 \end{equation} 
Then, the set
 \begin{equation*}
 \mathcal{N}_{f} \coloneqq\set*{ u \in X : f\tonde*{ u } =0},
 \end{equation*}
 is a linear subspace of $ X $. 
Moreover,
 \begin{equation}\label{f_tran_null}
 f\tonde*{ u } = f\tonde*{ u + v } ,
 \end{equation}
 for every $ u \in X $ and $ v \in \mathcal{N}_{f} $. 
 
 Finally, $f^{1/p}$ is a seminorm on $X$.
 Thus, for every $W$ linear subspace of $X$, 
 the restriction of $f^{1/p}$ to $W$ is a norm
 whenever $W\cap\mathcal{N}_{f}=\set*{0}$.
\end{lemma}
\begin{proof}
 We observe that \cref{dsa980h} implies that $ f $ is an even function and that $f(0)=0$.
 We prove that $ \mathcal{N}_{f} $ is a linear subspace of $ X $.
Let $ t , s \in \mathbb{R} $ and $ u , v \in \mathcal{N}_{f} $. Then, by the $ p $-homogeneity and the convexity of $ f $
we deduce,
\begin{align*}
 f\tonde*{ t u + s v } 
 &=
 \tonde*{\abs*{ t }+\abs*{ s }}^ p f\tonde*{\frac{\abs*{ t }}{\abs*{ t }+\abs*{ s }} \operatorname{sign}\tonde*{ t } u +\frac{\abs*{ s }}{\abs*{ t }+\abs*{ s }} \operatorname{sign}\tonde*{ s } v } 
 \\
 &\leq
 \tonde*{\abs*{ t }+\abs*{ s }}^ p 
 \tonde*{
 \frac{\abs*{ t }}{\abs*{ t }+\abs*{ s }} f\tonde*{ u } 
 +
 \frac{\abs*{ s }}{\abs*{ t }+\abs*{ s }} f\tonde*{ v } 
 }=0,
\end{align*}
if $ s \neq0$.
If $ s =0$, $ f\tonde*{ t u + s v } =\abs*{ t }^ p f\tonde*{ u } =0$.

Finally, we prove \eqref{f_tran_null}.
Let $ u \in X $ and $ v \in \mathcal{N}_{f} $ and $ t \in \left(0,1\right) $.
\begin{align*}
 f\tonde*{ u } 
 &=
 f\tonde*{ t \frac{ u + v }{ t }+(1- t )\frac{- v }{(1- t )}} 
 \\
 &\leq
 t f\tonde*{\frac{ u + v }{ t }} 
 +
 (1- t ) f\tonde*{\frac{- v }{(1- t )}} 
 =
 t ^{1- p } f\tonde*{ u + v } .
\end{align*}
Letting $ t \to 1$ we deduce that 
\begin{equation}\label{fds0}
 f\tonde*{ u } \leq f\tonde*{ u + v } 
 \text{ }
 \forall u \in X \forall v \in \mathcal{N}_{f} .
\end{equation}

Then, we let $ u \in X $ and $ v \in \mathcal{N}_{f} $
and applying $\eqref{fds0}$, we observe that $ f\tonde*{ u + v } \leq f\tonde*{ u + v +\tonde*{- v }} = f\tonde*{ u } $,
since $- v \in \mathcal{N}_{f} $. This together with \eqref{fds0} proves \eqref{f_tran_null}.

Finally, we define $C\coloneqq\set*{x\in X: f(x)<1}$ and observe that $tC=\set*{x\in X: f(x)<t^p}$ for every $t>0$, by the $p$-homogeneity of $f$.
Taking into account this information, we compute the \emph{Minkowski functional of $C$} (cfr. \cite{Bre11-MR2759829}) and obtain that it coincides with $f^{1/p}$. Therefore, by \cite[Lem 1.2]{Bre11-MR2759829} we deduce that $f^{1/p}$ is a seminorm. 

\end{proof}

The following lemma is a crucial tool in the proof of our main theorem (cfr. \cite[Lemma 3.1]{FarGuaSch20-MR4109099}).
\begin{lemma}[][res:cover_cubes]
 Let $ \Omega \subset \mathbb{R}^{n} $ be a bounded open set, $ u \in C^{1}\tonde*{\overline{\Omega};\mathbb{R}^{m}} $
 and for every $ t >0$ let $ U_{t} \coloneqq\set*{ x \in \Omega : \abs*{ \nabla u\tonde*{ x }}_{2} > t }$.

 Let $ t >0$ and $ \sigma >0$.
 Then, there exist $ r > 0$ and
 a ﬁnite family of pairwise disjoint open cubes $ \graffe*{ Q\tonde*{x_{i};r}}_{1\leq i \leq l } $ 
 contained
 in $ U_{t} $, such that,
 \begin{enumerate}[label=\alph*)]
 \item\label{item:cover_cubes_a} $ \abs*{ U_{t} \setminus \bigcup^{ l }_{ i =1} Q\tonde*{x_{i};r}} < \sigma $,
 \item\label{item:cover_cubes_b} $ \abs*{ \nabla u\tonde*{ x } - \nabla u\tonde*{ y }}_{2} < \sigma \qquad \forall\, x , y \in Q\tonde*{x_{i};r} \,\forall1\leq i \leq l $.
 \end{enumerate}
 Moreover, if $ \abs*{ \partial U_{t}} =0$ there exists $ W_{t, \sigma } \subset \Omega $ open s.t.
 $ \Omega \cap \tonde*{ \overline{ U_{t}} \setminus \bigcup^{ l }_{ i =1} Q\tonde*{x_{i};r}}\subset W_{t, \sigma } $
 and $ \abs*{ W_{t, \sigma }} < \sigma $.
\end{lemma}
\begin{proof}
 Let $ t >0,\, \sigma >0$.
 Since $ \nabla u $ is uniformly continuous in $ \overline{ \Omega } $,
 there exists $ \delta >0$ s.t. $ \abs*{ \nabla u\tonde*{ x } - \nabla u\tonde*{ y }}_{2} < \sigma $
 for every cube $ Q^{ \delta } \subset \overline{\Omega} $ of side length $ \delta $
 and for every $ x , y \in Q^{ \delta } $. 
 Since $ U_{t} $ is open, there exists a sequence
 of pairwise disjoint open dyadic cubes $ \graffe*{Q_{k}}_{k\in\mathbb{N}} $ with side length less than $ \delta $, 
 such that $ U_{t} = \bigcup_{ k \in \mathbb{N}} Q_{ k } $ up to a set of zero Lebesgue measure.
 Moreover, by the inner regularity of the Lebesgue measure,
 we can find a compact set $ K_{t, \sigma } \subset U_{t} $
 s.t. $ \abs*{ U_{t} \setminus K_{t, \sigma }} < \sigma /2$.
 Thanks to the compactness of $ K_{t, \sigma } $ from $ \graffe*{Q_{k}}_{k\in\mathbb{N}} $ we can extract a finite subfamily 
 $ \graffe*{ Q_{ k }}_{1\leq k \leq l_{0}} $ such that
 $ K_{t, \sigma } \subset \bigcup^{ l_{0}}_{ k =1} Q_{ k } $.
 We define $ r >0$ as the minimum side length
 among all the cubes in $ \graffe*{ Q_{ k }}_{1\leq k \leq l_{0}} $.
 For every $1\leq k \leq l_{0} $ we subdivide in a dyadic way $ Q_{ k } $
 with respect to the length $ r $:
 we observe that $ Q_{ k } = \bigcup_{ i=1 }^{j_k} Q^{i}_{k} $ (up to a set of zero Lebesgue measure),
 where $ \graffe*{ Q^{i}_{k}}_{1\leq i \leq j_{k}} $ is a finite family of pairwise disjoint open dyadic cubes of side length $ r $.
 $ \graffe*{ Q^{i}_{k}}_{1\leq k \leq l_{0} ,1\leq i \leq j_{k}} $ is (up to relabeling the indexes)
 the desired family $ \graffe*{ Q\tonde*{x_{i};r}}_{1\leq i \leq l } $, which guarantees
 that \ref{item:cover_cubes_a} and \ref{item:cover_cubes_b} hold.

 Finally, if $ \abs*{ \partial U_{t}} =0$, then
 $ \abs*{ \overline{ U_{t}} \setminus \bigcup_{1\leq i \leq l } Q\tonde*{x_{i};r}} = \abs*{ U_{t} \setminus \bigcup_{1\leq i \leq l } Q\tonde*{x_{i};r}} < \sigma /2$.
 Therefore, by the outer regularity of the Lebesgue measure,
 we can find an open set $ O_{t, \sigma } $ s.t. $ \overline{ U_{t}} \setminus \bigcup_{1\leq i \leq l } Q\tonde*{x_{i};r} \subset O_{t, \sigma } $
 and $ \abs*{ O_{t, \sigma }} < \sigma $.
 
 We can then define 
 $ W_{t, \sigma } \coloneqq \Omega 
 \cap 
 O_{t, \sigma } 
 $ to conclude the proof.
\end{proof}

For $ \Omega \subseteq \mathbb{R}^{n} $ open, $ p \in \left[1,\infty\right] $ and $ u \in W^{1,p}_{\mathrm{loc}}\tonde*{\Omega;\mathbb{R}^{n}} $ we denote by $ \nabla u $ 
the weak Jacobian matrix of $ u $ and by $ \mathcal{E}u $ the symmetric part of $ \nabla u $, also referred
to as the \emph{symmetric gradient} of $ u $:
\begin{equation*}
  \mathcal{E}u \coloneqq\frac{ \nabla u + \nabla u ^{T}}{2}.
\end{equation*}

Throughout this work we will need the following generalization of the notion of weak gradient.
\begin{definition}[][def:DL]
 Let $ \Omega \subseteq \mathbb{R}^{n} $ be an open set, $ u \in L^{1}_{\mathrm{loc}}\tonde*{\Omega;\mathbb{R}^{m}} $
 and $ \mathfrak{L} : \mathbb{R}^{m\times n} \to \mathbb{R}^{m\times n} $ be a linear map:
 \begin{equation*}
 \mathfrak{L}\tonde*{ A } _{ i j } \coloneqq\sommatoria[1\leq k \leq m ,1\leq l \leq n ]{ \mathfrak{L} _{ i j k l } A_{ k l }},
 \end{equation*}
 for every $ A \in \mathbb{R}^{m\times n} $ and for every $1\leq i \leq m $ and $1\leq j \leq n $.
 We define the $ m \times n $ matrix of distributions $ D_{\mathfrak{L}} u $ as,
 \begin{equation*}
 \dpair*{ \tonde*{ D_{\mathfrak{L}} u }_{ i j } , \phi } 
 \coloneqq
 -
 \int_{ \Omega } 
 \sommatoria[1\leq k \leq m ,1\leq l \leq n ]{ \mathfrak{L} _{ i j k l } u _{ k } \frac{\partial \phi }{\partial x_{ l }}}
 \thinspace\mathrm{d} x 
 =
 -
 \int_{ \Omega } 
 \mathfrak{L}\tonde*{ u \otimes \nabla \phi } _{ i j } 
 \thinspace\mathrm{d} x 
 ,
 \end{equation*}
 for every $1\leq i \leq m $, $1\leq j \leq n $ and for every $ \phi \in C^{\infty}_{c}\tonde*{\Omega} $. 
\end{definition}

We will also make use of the following \emph{Sobolev-type spaces}.
\begin{definition}[][def:WL]
 Let $ \Omega \subseteq \mathbb{R}^{n} $ be an open set, $ p \in \left[1,\infty\right) $ and $ \mathfrak{L} : \mathbb{R}^{m\times n} \to \mathbb{R}^{m\times n} $ be a linear map.
 Given $ u \in L^{p}_{\mathrm{loc}}\tonde*{\Omega;\mathbb{R}^{m}} $,
 whenever the matrix of distributions
 $ D_{\mathfrak{L}} u $ is represented by a map in $ L^{p}_{\mathrm{loc}}\tonde*{\Omega;\mathbb{R}^{m\times n}} $,
 by a slight abuse of notation, we will denote that function
 using the same symbol $ D_{\mathfrak{L}} u $.
 We define,
 \begin{equation*}
 W^{\mathfrak{L},p}_{\mathrm{loc}}\tonde*{\Omega;\mathbb{R}^{m}} \coloneqq\set*{ u \in L^{p}_{\mathrm{loc}}\tonde*{\Omega;\mathbb{R}^{m}} : D_{\mathfrak{L}} u \in L^{p}_{\mathrm{loc}}\tonde*{\Omega;\mathbb{R}^{m\times n}}},
 \end{equation*}
 \begin{equation*}
 W^{\mathfrak{L},p}\tonde*{\Omega;\mathbb{R}^{m}} \coloneqq\set*{ u \in L^{p}\tonde*{\Omega;\mathbb{R}^{m}} : D_{\mathfrak{L}} u \in L^{p}\tonde*{\Omega;\mathbb{R}^{m\times n}}}.
 \end{equation*}
\end{definition}
For a more general version of the spaces in \cref{def:WL} and their properties we
refer for instance to \cite{DieGme25-MR4937883,GmeRai19-MR4019087}.

We observe that $ W^{\mathfrak{L},p}_{\mathrm{loc}}\tonde*{\Omega;\mathbb{R}^{m}} = L^{p}_{\mathrm{loc}}\tonde*{\Omega;\mathbb{R}^{m}} $ and $ W^{\mathfrak{L},p}\tonde*{\Omega;\mathbb{R}^{m}} = L^{p}\tonde*{\Omega;\mathbb{R}^{m}} $ if $ \mathfrak{L} =0$,
while $ W^{\mathfrak{L},p}_{\mathrm{loc}}\tonde*{\Omega;\mathbb{R}^{m}} = W^{1,p}_{\mathrm{loc}}\tonde*{\Omega;\mathbb{R}^{m}} $ and $ W^{\mathfrak{L},p}\tonde*{\Omega;\mathbb{R}^{m}} = W^{1,p}\tonde*{\Omega;\mathbb{R}^{m}} $ if $ \mathfrak{L} = \operatorname{id}_{\mathbb{R}^{m\times n}} $.

Moreover, if $ u \in W^{1,p}_{\mathrm{loc}}\tonde*{\Omega;\mathbb{R}^{m}} $, then for every $ \mathfrak{L} : \mathbb{R}^{m\times n} \to \mathbb{R}^{m\times n} $ linear map $ D_{\mathfrak{L}} u = \mathfrak{L}\tonde*{ \nabla u } \in L^{p}_{\mathrm{loc}}\tonde*{\Omega;\mathbb{R}^{m\times n}} $, being
\begin{equation*}
 \dpair*{ \tonde*{ D_{\mathfrak{L}} u }_{ i j } , \phi } 
 =
 \int_{ \Omega } 
 \mathfrak{L}\tonde*{ \nabla u } _{ i j } 
 \phi 
 \thinspace\mathrm{d} x ,
\end{equation*}
for every $ \phi \in C^{\infty}_{c}\tonde*{\Omega} $.

Therefore, 
\begin{equation}\label{oij3298y}
 W^{1,p}_{\mathrm{loc}}\tonde*{\Omega;\mathbb{R}^{m}} \subset W^{\mathfrak{L},p}_{\mathrm{loc}}\tonde*{\Omega;\mathbb{R}^{m}} ,\text{ } W^{1,p}\tonde*{\Omega;\mathbb{R}^{m}} \subset W^{\mathfrak{L},p}\tonde*{\Omega;\mathbb{R}^{m}} ,
\end{equation}
and the set inclusions in \eqref{oij3298y} become equalities when
$ \mathfrak{L} $ coincides with the identity operator.

We also observe that if 
\begin{equation}\label{130h}
 \mathfrak{L}\tonde*{ A } \coloneqq\frac{ A + A ^{T}}{2},
\end{equation}
for every $ A \in \mathbb{R}^{n\times n} $,
then $ D_{\mathfrak{L}} u = \mathcal{E}u $ for every $ u \in W^{1,p}_{\mathrm{loc}}\tonde*{\Omega;\mathbb{R}^{n}} $. 

\begin{definition}[][def:E1p]
 Let $ \Omega \subseteq \mathbb{R}^{n} $ be an open set, $ p \in \left[1,\infty\right) $ and $ \mathfrak{L} $
 as in \cref{130h}.
 We define,
\begin{equation*}
 E^{1,p}\tonde*{\Omega;\mathbb{R}^{n}} \coloneqq W^{\mathfrak{L},p}\tonde*{\Omega;\mathbb{R}^{n}} ,\text{ } E^{1,p}_{\mathrm{loc}}\tonde*{\Omega;\mathbb{R}^{n}} \coloneqq W^{\mathfrak{L},p}_{\mathrm{loc}}\tonde*{\Omega;\mathbb{R}^{n}} .
\end{equation*}
\end{definition}

We remark that for $ p \in \left(1,\infty\right) $
the space $ E^{1,p}\tonde*{\Omega;\mathbb{R}^{n}} $ coincides with $ W^{1,p}\tonde*{\Omega;\mathbb{R}^{n}} $
if $ \Omega $ is sufficiently regular (cfr. \cite[Thm. 2]{DiFSol25-MR4840265})
while in general the two spaces are different (cfr. \cite{AcoDurLop13-MR2988724}).
Therefore, the set inclusions in \eqref{oij3298y} can be equalities even for operators $ \mathfrak{L} $
different from the identity.
Instead, for $ p =1$ the spaces $ W^{1,1}\tonde*{\Omega;\mathbb{R}^{n}} $ and $ E^{1,1}\tonde*{\Omega;\mathbb{R}^{n}} $ are different
(cfr. \cite{DiFSol25-MR4840265}).
For a discussion related to this subject we refer to
\cite{Van13-MR3085095} and the references therein.

\begin{lemma}[][res:DL_e_conv]
 Let $ p \in \left[1,\infty\right) $ and $ \mathfrak{L} : \mathbb{R}^{m\times n} \to \mathbb{R}^{m\times n} $
 be a linear operator.
 Let $ u \in W^{\mathfrak{L},p}\tonde*{\mathbb{R}^{n};\mathbb{R}^{m}} $.
 Then,
 \begin{equation*}
 \mathfrak{L}\tonde*{ \nabla \tonde*{ \rho _{ \sigma } \ast u }} 
 =
 D_{ \mathfrak{L}}\tonde*{ \rho _{ \sigma } \ast u } 
 =
 \rho _{ \sigma } \ast D_{\mathfrak{L}} u 
 \in L^{p}\tonde*{\mathbb{R}^{n};\mathbb{R}^{m\times n}} 
 \end{equation*}
\end{lemma}
\begin{proof}
 The strategy is the same one used in the proof of
 \cite[Lem 8.4]{Bre11-MR2759829}
\end{proof}
\begin{lemma}[][res:DL_Leibniz]
 Let $ \Omega \subseteq \mathbb{R}^{n} $ be an open set, $ p \in \left[1,\infty\right) $ and $ \mathfrak{L} : \mathbb{R}^{m\times n} \to \mathbb{R}^{m\times n} $
 be a linear operator.
 Let $ u \in W^{\mathfrak{L},p}_{\mathrm{loc}}\tonde*{\Omega;\mathbb{R}^{m}} $ and $ v \in C^{1}_{c}\tonde*{\Omega} $.
 Then, $\overline{ u v }\in W^{\mathfrak{L},p}\tonde*{\mathbb{R}^{n};\mathbb{R}^{m}} $ and,
 \begin{equation*}
 D_{ \mathfrak{L}}\overline{ u v } =
 \overline{ v D_{\mathfrak{L}} u + u D_{ \mathfrak{L}} v }
 \in L^{p}\tonde*{\mathbb{R}^{n};\mathbb{R}^{m\times n}} ,
 \end{equation*}
 where $\overline{f}$ denotes the null extensions of a function $f$ defined on $ \Omega $ to the whole space $ \mathbb{R}^{n} $.
\end{lemma}
\begin{proof}
 The strategy is the same one used in the proof of
 \cite[Ch. 9, Rem. 4, part (ii)]{Bre11-MR2759829}.
\end{proof}

\begin{lemma}[][res:DL_density]
 Let $ \Omega \subseteq \mathbb{R}^{n} $ be an open set, $ p \in \left[1,\infty\right) $ and $ \mathfrak{L} : \mathbb{R}^{m\times n} \to \mathbb{R}^{m\times n} $
 be a linear operator.
 Let $ u \in W^{\mathfrak{L},p}_{\mathrm{loc}}\tonde*{\Omega;\mathbb{R}^{m}} $.
 Then, 
 there exists $ \graffe*{u_{ \sigma }}_{ \sigma >0} \subset C^{\infty}_{c}\tonde*{\mathbb{R}^{n};\mathbb{R}^{m}} $ s.t. $ u_{ \sigma } \to u $ in $ L^{p}_{\mathrm{loc}}\tonde*{\Omega;\mathbb{R}^{m}} $ as $ \sigma \to 0^{+} $ and,
 \begin{equation*}
 \lim\limits_{ \sigma \to 0^{+}} \norm*{ D_{\mathfrak{L}} u - D_{ \mathfrak{L}} u_{ \sigma }}_{L^{p}\tonde*{\tilde{\Omega};\mathbb{R}^{m\times n}}} =0,
 \end{equation*}
 for every $ \tilde{\Omega} \in \mathcal{A}_{\Omega} $ s.t. $ \tilde{\Omega} \Subset \Omega $.

 Moreover, if $ u \in W^{\mathfrak{L},p}\tonde*{\Omega;\mathbb{R}^{m}} $, 
 then $ u_{ \sigma } \to u $ in $ L^{p}\tonde*{\Omega;\mathbb{R}^{m}} $ 
 and $ u_{ \sigma } \to u $ in $ W^{\mathfrak{L},p}\tonde*{\tilde{\Omega};\mathbb{R}^{m}} $ for every 
 $ \tilde{\Omega} \in \mathcal{A}_{\Omega} $ s.t. $ \tilde{\Omega} \Subset \Omega $ as $ \sigma \to 0^{+} $.
\end{lemma}

\begin{proof}
 The strategy is the same one used in the proof of \cite[Thm. 9.2]{Bre11-MR2759829}.
 
 For every $ \sigma >0$ we define $ u_{ \sigma } \coloneqq\zeta_ \sigma \tonde*{ \rho _{ \sigma } \ast \overline{ u }}$
 (where $\overline{ u }$ is the null extension of $ u $ to $ \mathbb{R}^{n} $)
 and $\zeta\in C^{\infty}_{c}\tonde*{\mathbb{R}^{n};[0,1]} $ is a cut-off function
 null in $ \mathbb{R}^{n} \setminus B_{2}\tonde*{0} $ and identically equal to $1$ in $ \overline{ B_{1}\tonde*{0}} $
 and $\zeta_ \sigma ( x )\coloneqq\zeta( \sigma x )$ for every $ \sigma >0$ and $x\in\mathbb{R}^n$.

 We fix $ \tilde\Omega, \Omega' \in \mathcal{A}_{\Omega} $ s.t. $ \tilde{\Omega} \Subset \Omega' \Subset \Omega $ and a function $ v \in C^{\infty}_{c}\tonde*{\mathbb{R}^{n};[0,1]} $ with $ \operatorname{supp}\tonde*{ v } \subset \Omega' $
 s.t $ v \equiv1$ on a neighborhood of $ \tilde{\Omega} $.
 Then let $\overline{ u v }$ be the null extension of $ u v $ to $ \mathbb{R}^{n} $.
 We then observe that $\restr{\tonde*{ \rho _{ \sigma } \ast \overline{ u v }}}{ \tilde{\Omega} 
 }=\restr{\tonde*{ \rho _{ \sigma } \ast \overline{ u }}}{ \tilde{\Omega}}$ for every $ \sigma >0$.

 By \cref{res:DL_e_conv} and \cref{res:DL_Leibniz} we obtain
 that $\graffe*{u_{ \sigma }}_{ \sigma >0}$ is the desired sequence.
\end{proof} 
\section{\texorpdfstring{$p$-core}{p-core} functionals and generalized BMO-type seminorms}\label{sec:genaralBMO}

Given two open sets $ \Omega , D \subseteq \mathbb{R}^{n} $ and $ \Gamma \subseteqq \mathrm{Aff}\tonde*{\mathbb{R}^{n}} $, we denote by, 
\begin{equation*}
 \mathcal{E}^{ D}_{ \Gamma }\tonde*{\Omega} \coloneqq\set*{ F\tonde*{ D } : F \in \Gamma \text{ and } F\tonde*{ D } \subseteq \Omega }, 
\end{equation*}
the class of all subsets of $ \Omega $ obtained by transforming $ D $ 
according to the elements of $ \Gamma $.

We start stating  the main definition of our work.
\begin{definition}[$p$-core functional][def:core_function]
 Let $ D \subseteq \mathbb{R}^{n} $ be a bounded open set, $ m \geq1$ a natural number, $ p \in \left[1,\infty\right) $
 and $ \Gamma $ a subgroup of $ \mathrm{Aff}\tonde*{\mathbb{R}^{n}} $.
 
 We define $ \operatorname{dom}\tonde*{\alpha_{p,D, \Gamma ,m}} \coloneqq L^{p}_{\mathrm{loc}}\tonde*{\mathbb{R}^{n};\mathbb{R}^{m}} \times \mathcal{E}^{ D}_{ \Gamma }\tonde*{\mathbb{R}^{n}} $
 and we let,
 \begin{equation*}
 \alpha_{ p, D, \Gamma ,m} : \operatorname{dom}\tonde*{\alpha_{p,D, \Gamma ,m}} \to \left[0,+\infty\right) 
 \end{equation*}
 be a positive function.

 To avoid heavy notation, when the set $ D $,
 the subgroup $ \Gamma $ and $ m $ are clear from the context
 we shall drop some subscripts in $ \alpha_{ p, D, \Gamma ,m} $.
 For instance, below we will write $ \alpha_{ p} $ instead of $ \alpha_{ p, D, \Gamma ,m} $.

 We say that $ \alpha_{ p} $ is a \emph{ $p$-core functional } (subordinated to the \emph{reference set} $ D $,
 the subgroup $ \Gamma $ and $ m $)
 if it satisfies the following properties.

 Given $ D' \in \mathcal{E}^{ D}_{ \Gamma }\tonde*{\mathbb{R}^{n}} $
 and a certain function in $ L^{p}\tonde*{D';\mathbb{R}^{m}} $,
 $ \alpha_{ p} $ is \emph{extension independent} in the following sense:
 \begin{equation}\label{alpha_p:ext_ind}
 \tag{EI}
 \alpha_{ p}\tonde*{ u , D' } = \alpha_{ p}\tonde*{ v , D' } 
 \text{ }
 \forall D' \in \mathcal{E}^{ D}_{ \Gamma }\tonde*{\mathbb{R}^{n}} 
 \forall u , v \in L^{p}_{\mathrm{loc}}\tonde*{\mathbb{R}^{n};\mathbb{R}^{m}} \text{ s.t. }\restr{ u }{ D' }=\restr{ v }{ D' }.
 \end{equation}
 $ \alpha_{ p} $ is \emph{convex} in its first entry:
 \begin{equation}\label{alpha_p:conv}
 \tag{CX}
 \alpha_{ p}\tonde*{\cdot, D' } \text{ is convex }
 \forall D' \in \mathcal{E}^{ D}_{ \Gamma }\tonde*{\mathbb{R}^{n}} .
 \end{equation}
 $ \alpha_{ p} $ is \emph{$ p $-homogeneous} in its first entry:
 \begin{equation}\label{alpha_p:p-homogeneous}
 \tag{$ p $H}
 \alpha_{ p}\tonde*{ t u , D' } =\abs{ t }^{ p } \alpha_{ p}\tonde*{ u , D' } 
 \text{ }
 \forall\tonde*{ u , D' }\in \operatorname{dom}\tonde*{\alpha_{p}} \,\forall t \in \mathbb{R} .
 \end{equation}
 $ \alpha_{ p} $ is \emph{translation invariant} in its first entry:
 \begin{equation}\label{alpha_p:outer_trans_inv}
 \tag{T1}
 \alpha_{ p}\tonde*{ u + h , D' } = \alpha_{ p}\tonde*{ u , D' } 
 \text{ }
 \forall\tonde*{ u , D' }\in \operatorname{dom}\tonde*{\alpha_{p}} \,\forall h \in \mathbb{R}^{m} .
 \end{equation}
 $ \alpha_{ p} $ satisfies the following \emph{change of variable} property:
 \begin{equation}\label{alpha_p:change_of_variable}
 \tag{CV}
 \alpha_{ p}\tonde*{ u \circ F , D' } = \alpha_{ p}\tonde*{ u , F\tonde*{ D' }} 
 \text{ }
 \forall\tonde*{ u , D' }\in \operatorname{dom}\tonde*{\alpha_{p}} \forall F \in \Gamma .
 \end{equation}
 When $ \alpha_{ p}\tonde*{\cdot, D } $ is tested on a Sobolev function it is
 \emph{bounded} from above by the $ L^{p}\tonde*{D;\mathbb{R}^{m\times n}} $ norm of its \emph{gradient}
 (up to a multiplicative constant):
 \begin{equation}\label{alpha_p:poincare_type_bound}
 \tag{GB}
 \begin{split}
 &\exists c_{1} >0
 \text{ s.t. }
 \alpha_{ p}\tonde*{ u , D } 
 \leq
 c_{1} \int_{ D } \abs*{ \nabla u }^{ p }_{2} \thinspace\mathrm{d} x 
 \text{ }
 \\
 &\forall u \in L^{p}_{\mathrm{loc}}\tonde*{\mathbb{R}^{n};\mathbb{R}^{m}} \text{ s.t. }\restr{ u }{ D }\in W^{1,p}\tonde*{D;\mathbb{R}^{m}} .
 \end{split}
 \end{equation}
 
\end{definition}

We observe that if $ \alpha_{ p} $ is a $p$-core functional subordinated to a certain
$ \Gamma $, subgroup of $ \mathrm{Aff}\tonde*{\mathbb{R}^{n}} $, then it can be naturally seen also as a $p$-core functional subordinated to every 
subgroup of $ \Gamma $.

Let us collect some properties of a $p$-core functional $ \alpha_{ p} $ as in \cref{def:core_function}.
 Let $ D' \in \mathcal{E}^{ D}_{ \Gamma }\tonde*{\mathbb{R}^{n}} $. 
 Then, for every $ u , v \in L^{p}_{\mathrm{loc}}\tonde*{\mathbb{R}^{n};\mathbb{R}^{m}} $,
 \begin{equation}\label{alpha_p:almost_triang}
 \alpha_{ p}\tonde*{ u + v , D' } 
 \leq
 2^{ p -1}\tonde*{ \alpha_{ p}\tonde*{ u , D' } + \alpha_{ p}\tonde*{ v , D' }},
 \end{equation}
 thanks to \eqref{alpha_p:conv} and \eqref{alpha_p:p-homogeneous}.
 In particular, if $ p =1$, \eqref{alpha_p:almost_triang} amounts to a triangular inequality
 for $ u \mapsto \alpha_{ 1}\tonde*{ u , D' } $. This then allows to conclude (using also the symmetry inherited
 from \eqref{alpha_p:p-homogeneous}), that
 \begin{equation*}
 \abs*{ \alpha_{ 1}\tonde*{ u , D' } - \alpha_{ 1}\tonde*{ v , D' }}
 \leq
 \alpha_{ 1}\tonde*{ u - v , D' } , 
 \end{equation*}
 for every $ u , v \in L^{1}_{\mathrm{loc}}\tonde*{\mathbb{R}^{n};\mathbb{R}^{m}} $.

 The following inequalities
 hold for any $ p \in \left[1,\infty\right) $ and are another consequence
 of \eqref{alpha_p:conv} and \eqref{alpha_p:p-homogeneous}.
 \begin{equation*}
 \alpha_{ p}\tonde*{ u + v , D' } 
 \leq
 \tonde*{1+ \delta }^ p \alpha_{ p}\tonde*{ u , D' } 
 +
 \tonde*{\frac{1+ \delta }{ \delta }}^ p \alpha_{ p}\tonde*{ v , D' } 
 \end{equation*}
 \begin{equation}\label{alpha_p:utile_2}
 \alpha_{ p}\tonde*{ u - v , D' } 
 \geq
 \frac{1}{\tonde*{1+ \delta }^ p } \alpha_{ p}\tonde*{ u , D' } 
 -
 \frac{1}{ \delta ^ p } \alpha_{ p}\tonde*{ v , D' } 
 \end{equation}
 for every $ u , v \in L^{p}_{\mathrm{loc}}\tonde*{\mathbb{R}^{n};\mathbb{R}^{m}} $ and $ \delta \in \left(0,1\right) $.

 Furthermore,
 $ \alpha_{ p} $ is zero when evaluated on constant functions,
 \begin{equation}\label{alpha_p:zero_constant}
 \tag{ZC}
 \alpha_{ p}\tonde*{ h , D' } =0
 \text{ }
 \forall h \in \mathbb{R}^{m} ,
 \end{equation}
 thanks to \eqref{alpha_p:outer_trans_inv} and \eqref{alpha_p:p-homogeneous}.

 In addition to that, \eqref{alpha_p:change_of_variable}, \eqref{alpha_p:ext_ind} and \eqref{alpha_p:poincare_type_bound}
 imply
 \begin{equation}\label{alpha_p:better_poinc}
 \begin{split}
 & \alpha_{ p}\tonde*{ u , F\tonde*{ D }} 
 \leq
 c_{1} \frac{ \abs*{ \nabla F }^{ p }_{2}}{\abs*{\det\tonde*{ \nabla F }}} \int_{ F\tonde*{ D }} \abs*{ \nabla \tonde*{\restr{ u }{ F\tonde*{ D }}}}^{ p }_{2} \thinspace\mathrm{d} x 
 \\
 &\forall F \in \Gamma \,\forall u \in L^{p}_{\mathrm{loc}}\tonde*{\mathbb{R}^{n};\mathbb{R}^{m}} \text{ s.t. } \nabla \tonde*{\restr{ u }{ F\tonde*{ D }}} \in L^{p}\tonde*{F\tonde*{D};\mathbb{R}^{m\times n}} . 
 \end{split}
 \end{equation}

In the next definition we will introduce the subgroups of $\mathrm{Aff}\tonde*{\mathbb{R}^{n}}$ which will
be fundamental in defining our class of generalized BMO-type seminorms.
\begin{definition}[][def:Gamma_G]
 Let $ \Omega \subseteq \mathbb{R}^{n} $ be an open set,
 $ D \subseteq \mathbb{R}^{n} $ a bounded open set and
 $ \mathcal{G} $ a subgroup of $ \mathrm{SO}\tonde*{n} $.
 For every $ \varepsilon >0$ we define,
 \begin{align*}
 & \Gamma ^{\mathcal{G}}_{\varepsilon} 
 \coloneqq
 \set*{ \mathbb{R}^{n} \ni x \mapsto \varepsilon \mathcal{R}\tonde*{ x } + h : \mathcal{R} \in \mathcal{G} , h \in \mathbb{R}^{n}}\subset \mathrm{Aff}\tonde*{\mathbb{R}^{n}} ,
 \text{ }
 \Gamma ^{\mathcal{G}} 
 \coloneqq
 \bigcup_{ \varepsilon >0} \Gamma ^{\mathcal{G}}_{\varepsilon} \subset \mathrm{Aff}\tonde*{\mathbb{R}^{n}} ,
 \\
 & \mathcal{G}^{ D}_{\varepsilon}\tonde*{\Omega} \coloneqq \mathcal{E}^{ D }_{ \Gamma ^{\mathcal{G}}_{\varepsilon}}\tonde*{ \Omega } ,
 \text{ }
 \mathcal{G}^{ D}\tonde*{\Omega} \coloneqq \mathcal{E}^{ D }_{ \Gamma ^{\mathcal{G}}}\tonde*{ \Omega } = \bigcup_{ \varepsilon >0} \mathcal{G}^{ D}_{\varepsilon}\tonde*{\Omega} .
 \end{align*}
 We observe that $ \mathcal{G}^{ D}_{\varepsilon}\tonde*{\Omega} \neq\emptyset$ if $ \varepsilon >0$ is small enough.
\end{definition}

Let $ \mathcal{G} $ be a subgroup of $ \mathrm{SO}\tonde*{n} $. Let us collect some properties of a $p$-core functional $ \alpha_{ p} $ be  as in \cref{def:core_function}
 subordinated to $ \Gamma ^{\mathcal{G}} $ as in \cref{def:Gamma_G}. 
 
 \eqref{alpha_p:change_of_variable} implies that
 $ \alpha_{ p} $ satisfies the following property connected to \emph{translations},
 \begin{equation}\label{alpha_p:inner_trans_inv}
 \tag{T2}
 \alpha_{ p}\tonde*{ u , D' } = \alpha_{ p}\tonde*{ u\tonde*{\cdot- h }, D' + h } 
 \text{ }
 \forall\tonde*{ u , D' }\in \operatorname{dom}\tonde*{\alpha_{p}} \,\forall h \in \mathbb{R}^{n} ,
 \end{equation}
 and that
 $ \alpha_{ p} $ satisfies the following \emph{scaling} property:
 \begin{equation}\label{alpha_p:scaling}
 \tag{SC}
 \alpha_{ p}\tonde*{ u\tonde*{\frac{\cdot}{ t }}, D' } = \alpha_{ p}\tonde*{ u ,\frac{ D' }{ t }} 
 \text{ }
 \forall\tonde*{ u , D' }\in \operatorname{dom}\tonde*{\alpha_{p}} \,\forall t >0.
 \end{equation}
 Finally, from \eqref{alpha_p:better_poinc} we can infer that,
 \begin{equation}\label{alpha_p:poincare_+}
 \tag{GB+}
 \begin{split}
 & \alpha_{ p}\tonde*{ u , D' } 
 \leq
 c_{2} \operatorname{diam}\tonde*{ D' } ^{ p - n } \int_{ D' } \abs*{ \nabla \tonde*{\restr{ u }{ D' }}}^{ p }_{2} \thinspace\mathrm{d} x 
 \text{ }
 \\
 &\forall D' \in \mathcal{G}^{ D}\tonde*{\mathbb{R}^{n}} \forall u \in L^{p}_{\mathrm{loc}}\tonde*{\mathbb{R}^{n};\mathbb{R}^{m}} \text{ s.t. } \nabla \tonde*{\restr{ u }{ D' }} \in L^{p}\tonde*{D';\mathbb{R}^{m\times n}} ,
 \end{split}
 \end{equation}
 where $ c_{2} = c_{1} \operatorname{diam}\tonde*{ D } ^{ n - p }$.

\begin{proposition}[][res:Amat_loc_lip]
 Let $ \alpha_{ p} $ be a $p$-core functional as in \cref{def:core_function}.
 Then,
 \begin{equation}\label{alpha_p:convex_su_mat}
 \mathbb{R}^{m\times n} \ni A \mapsto \alpha_{ p}\tonde*{ l^{ A}, D' } \text{ is convex }\forall D' \in \mathcal{E}^{ D}_{ \Gamma }\tonde*{\mathbb{R}^{n}} .
 \end{equation}
 Moreover,
 \begin{equation}\label{alpha_p:locLip_su_mat}
 \mathbb{R}^{m\times n} \ni A \mapsto \alpha_{ p}\tonde*{ l^{ A}, D' } \text{ is locally Lipschitz continuous }\forall D' \in \mathcal{E}^{ D}_{ \Gamma }\tonde*{\mathbb{R}^{n}} .
 \end{equation}
 Assume now that $ \alpha_{ p} $ is subordinated to
 $ \Gamma = \Gamma ^{\mathcal{G}} $ where $ \mathcal{G} $ is a subgroup of $ \mathrm{SO}\tonde*{n} $.
 Then, 
 \begin{equation}\label{eq:loclip_keyfinal}
 \abs*{ \alpha_{ p}\tonde*{ l^{ A}, D' } - \alpha_{ p}\tonde*{ l^{ B}, D' }}\leq \varepsilon ^ p C_{ \Sigma } \abs*{ A - B }_{2} 
 \text{ }
 \forall A , B \in \Sigma \forall \varepsilon >0\forall D' \in \mathcal{G}^{ D}_{\varepsilon}\tonde*{\mathbb{R}^{n}} ,
 \end{equation}
 where
 \begin{equation}\label{def:Sigma}
 \Sigma \coloneqq\set*{ A \in \mathbb{R}^{m\times n} : \abs*{ A }_{2} \leq1},
 \end{equation}
 and $ C_{ \Sigma } >0$ is a positive constant (depending on $ \Sigma $ and $ \alpha_{ p} $).
\end{proposition}
\begin{proof}
 Thanks to \eqref{alpha_p:conv} and the linearity of the map $ \mathbb{R}^{m\times n} \ni A \mapsto l^{ A} \in L^{p}_{\mathrm{loc}}\tonde*{\mathbb{R}^{n};\mathbb{R}^{m}} $,
 we observe that \eqref{alpha_p:convex_su_mat} holds.
 Since $ \mathbb{R}^{m\times n} $ is a finite-dimensional normed vector space, \eqref{alpha_p:convex_su_mat} implies
 \eqref{alpha_p:locLip_su_mat}.
 
 Assume now that $ \Gamma = \Gamma ^{\mathcal{G}} $ where $ \mathcal{G} $ is a subgroup of $ \mathrm{SO}\tonde*{n} $.
 Thanks to \eqref{alpha_p:locLip_su_mat} we can find a constant $ C_{ \Sigma } >0$
 (depending on $ \Sigma $ and $ \alpha_{ p} $) such that
 \begin{equation}\label{eq:loclip_key1}
 \abs*{ \alpha_{ p}\tonde*{ l^{ A}, D } - \alpha_{ p}\tonde*{ l^{ B}, D }}\leq C_{ \Sigma } \abs*{ A - B }_{2} 
 \text{ }
 \forall A , B \in \Sigma .
 \end{equation}
 Let $ \varepsilon >0$ and $ D' \in \mathcal{G}^{ D}_{\varepsilon}\tonde*{\mathbb{R}^{n}} $.
 Then, $ D' = \varepsilon \mathcal{R} D + h $ for some $ \mathcal{R} \in \mathcal{G} $ and $ h \in \mathbb{R}^{n} $.
 By \eqref{alpha_p:change_of_variable}, \eqref{alpha_p:outer_trans_inv} and
 \eqref{alpha_p:p-homogeneous} we find,
 \begin{equation}\label{eq:loclip_key2}
 \alpha_{ p}\tonde*{ l^{ A}, D' } 
 =
 \alpha_{ p}\tonde*{ l^{ A}\tonde*{ \varepsilon \mathcal{R}\tonde*{\cdot} + h }, D } 
 =
 \varepsilon ^ p \alpha_{ p}\tonde*{ l^{ A \mathcal{R}}, D } \text{ }\forall A \in \mathbb{R}^{m\times n} .
 \end{equation}
 Combining \eqref{eq:loclip_key1} and \eqref{eq:loclip_key2} we obtain
 \eqref{eq:loclip_keyfinal},
 noting that $ A \mathcal{R} \in \Sigma $ for every $ \mathcal{R} \in \mathrm{SO}\tonde*{n} $ and $ A \in \Sigma $.
\end{proof}
We are now ready to define the generalized BMO-type seminorms
which will be the focus of our work.
\begin{definition}[BMO-type seminorm functional][def:tipo_BMO] 
 Let $ \Omega \subseteq \mathbb{R}^{n} $ be an open set,
 which we will refer to as the \emph{ambient space},
 and $ \mathcal{G} $ be a subgroup of $ \mathrm{SO}\tonde*{n} $.
 Recalling \cref{def:Gamma_G}, let $ \alpha_{ p} = \alpha_{ p , D, \Gamma ^{\mathcal{G}} ,m } $ be a $p$-core functional as in \cref{def:core_function}
 with 
 $ \operatorname{dom}\tonde*{\alpha_{p}} = L^{p}_{\mathrm{loc}}\tonde*{\mathbb{R}^{n};\mathbb{R}^{m}} \times \mathcal{G}^{ D}\tonde*{\mathbb{R}^{n}} $.
 For every $ u \in L^{p}_{\mathrm{loc}}\tonde*{\mathbb{R}^{n};\mathbb{R}^{m}} $ we define,
 \begin{equation}\label{12d1}
 G^{\alpha_{p}}_{ \varepsilon }\tonde*{ u,\Omega} 
 \coloneqq
 \varepsilon ^{ n - p } \sup\limits_{\mathcal{G}_{\varepsilon}} \sum_{ D'\in\mathcal{G}_{\varepsilon}} \alpha_{ p}\tonde*{ u , D' } ,
 \end{equation}
 where each $ \mathcal{G}_{\varepsilon} $ is a family made of pairwise disjoint elements of $ \mathcal{G}^{ D}_{\varepsilon}\tonde*{\Omega} $ 
 and the supremum is taken with respect to all of these families $ \mathcal{G}_{\varepsilon} $.
 We observe that the cardinality of every family $ \mathcal{G}_{\varepsilon} $ 
 is at most countable and it is bounded from above by $ \abs*{ \Omega } /\tonde*{ \varepsilon ^ n \abs*{ D }}$.
 For convenience, if $ \mathcal{G}_{\varepsilon} =\emptyset$, we let the supremum in \eqref{12d1} be equal to $0$.
 We define,
 \begin{align*}
 & G^{ \alpha_{ p}}_{ +}\tonde*{ u,\Omega } \coloneqq \limsup\limits_{\varepsilon\to 0^{+}} G^{\alpha_{p}}_{ \varepsilon }\tonde*{ u,\Omega } ,
 \\
 & G^{ \alpha_{ p}}_{ -}\tonde*{ u,\Omega } \coloneqq \liminf\limits_{\varepsilon\to 0^{+}} G^{\alpha_{p}}_{ \varepsilon }\tonde*{ u,\Omega } ,
 \end{align*}
 for every $ u \in L^{p}_{\mathrm{loc}}\tonde*{\mathbb{R}^{n};\mathbb{R}^{m}} $.
 We also define,
 \begin{equation*}
 G^{ \alpha_{ p}}\tonde*{ u,\Omega } 
 \coloneqq
 \lim\limits_{\varepsilon\to 0^{+}} G^{\alpha_{p}}_{ \varepsilon }\tonde*{ u,\Omega } ,
 \end{equation*}
 for every $ u \in L^{p}_{\mathrm{loc}}\tonde*{\mathbb{R}^{n};\mathbb{R}^{m}} $ such that
 $ G^{ \alpha_{ p}}_{ -}\tonde*{ u,\Omega } = G^{ \alpha_{ p}}_{ +}\tonde*{ u,\Omega } $.
 
 We finally observe that for every $ \varepsilon >0$, thanks to \eqref{alpha_p:ext_ind},
 \begin{align}\label{tipo_BMO:ext_ind}
 & G^{\alpha_{p}}_{ \varepsilon }\tonde*{ u , \Omega } = G^{\alpha_{p}}_{ \varepsilon }\tonde*{ v , \Omega } ,
 & G^{ \alpha_{ p}}_{\pm}\tonde*{ u, \Omega } = G^{ \alpha_{ p}}_{\pm}\tonde*{ v , \Omega } ,
 \end{align}
 for all $ u , v \in L^{p}_{\mathrm{loc}}\tonde*{\mathbb{R}^{n};\mathbb{R}^{m}} $ such that $\restr{ u }{ \Omega }=\restr{ v }{ \Omega }$.
\end{definition}

\begin{remark}[][res:tip_BMO_bound]
 Under the assumptions of \cref{def:tipo_BMO},
 if $ u \in L^{p}_{\mathrm{loc}}\tonde*{\mathbb{R}^{n};\mathbb{R}^{m}} $ and $ \nabla \restr{ u }{ \Omega } \in L^{p}\tonde*{\Omega;\mathbb{R}^{m\times n}} $
 thanks to \eqref{alpha_p:poincare_+}
 we obtain that
 \begin{equation*}
 \begin{split}
 G^{\alpha_{p}}_{ \varepsilon }\tonde*{ u,\Omega } 
 &\leq
 \varepsilon ^{ n - p } \sup\limits_{\mathcal{G}_{\varepsilon}} \sum_{ D'\in\mathcal{G}_{\varepsilon}} c_{2} \operatorname{diam}\tonde*{ D' } ^{ p - n } \int_{ D' } \abs*{ \nabla u }^{ p }_{2} \thinspace\mathrm{d} x 
 \\
 &\leq
 c_{1} \int_{ \Omega } \abs*{ \nabla u }^{ p }_{2} \thinspace\mathrm{d} x ,
 \end{split}
 \end{equation*}
 for every $ \varepsilon >0$.
 Therefore,
 \begin{equation*}
 G^{ \alpha_{ p}}_{\pm}\tonde*{ u,\Omega } \leq c_{1} \int_{ \Omega } \abs*{ \nabla u }^{ p }_{2} \thinspace\mathrm{d} x .
 \end{equation*}
\end{remark}
We list some properties of $ G^{\alpha_{p}}_{ \varepsilon } $
and $ G^{ \alpha_{ p}}_{\pm} $, omitting the elementary
proofs.

\begin{proposition}
 Let $ \Omega \subseteq \mathbb{R}^{n} $ be an open set, $ m \geq1$ a natural number, $ p \in \left[1,\infty\right) $,
 $ \mathcal{G} $ a subgroup of $ \mathrm{SO}\tonde*{n} $ and $ D \subseteq \mathbb{R}^{n} $ be a bounded open set.
 Let $ \alpha_{ p} = \alpha_{ p , D, \Gamma ^{\mathcal{G}} ,m } $ 
 be a $p$-core functional as in \cref{def:core_function}
 subordinated to $ \Gamma ^{\mathcal{G}} $ as in \cref{def:Gamma_G}.
 Let $ \varepsilon >0$ and $ u \in L^{p}_{\mathrm{loc}}\tonde*{\mathbb{R}^{n};\mathbb{R}^{m}} $ and $ A \in \mathbb{R}^{m\times n} $.
 Recalling \cref{def:tipo_BMO},
 for every $ \lambda >0$ it holds,
 \begin{align*}
 & \lambda ^{ n - p } G^{ \alpha_{ p , \lambda D , \Gamma ^{\mathcal{G}}}}_{ \varepsilon }\tonde*{ u,\Omega } 
 =
 G^{ \alpha_{ p}}_{ \varepsilon \lambda }\tonde*{ u,\Omega } ,
 & \lambda ^{ n - p } G^{ \alpha_{ p , \lambda D , \Gamma ^{\mathcal{G}}}}_{\pm}\tonde*{ u,\Omega } 
 =
 G^{ \alpha_{ p}}_{\pm}\tonde*{ u,\Omega } ,
 \end{align*}
 where 
 $ \alpha_{ p , \lambda D , \Gamma ^{\mathcal{G}}} $ is the $p$-core functional
 defined by,
 \begin{equation*}
 \alpha_{ p , \lambda D , \Gamma ^{\mathcal{G}}}\tonde*{ u , D' } 
 \coloneqq
 \alpha_{ p}\tonde*{ u ,\frac{1}{ \lambda } D' } ,
 \end{equation*}
 for every $ D' \in \mathcal{G}^{ \lambda D }\tonde*{ \mathbb{R}^{n}} = \mathcal{G}^{ D}\tonde*{\mathbb{R}^{n}} $.

 For every $ t >0$, 
 \begin{align}
 &\notag
 t ^{ n - p } G^{ \alpha_{ p}}_{ \varepsilon }\tonde*{ u,\Omega } 
 =
 G^{ \alpha_{ p}}_{ t \varepsilon }\tonde*{ u\tonde*{\cdot/ t } , t \Omega } ,
 &&
 t ^{ n - p } G^{ \alpha_{ p}}_{\pm}\tonde*{ u,\Omega } 
 =
 G^{ \alpha_{ p}}_{\pm}\tonde*{ u\tonde*{\cdot/ t } , t \Omega } ,
 \\ & \label{eq:G_lA-scaling} 
 t ^{ n } G^{ \alpha_{ p}}_{ \varepsilon }\tonde*{ l^{ A} , \Omega } 
 =
 G^{ \alpha_{ p}}_{ t \varepsilon }\tonde*{ l^{ A} , t \Omega } ,
 &&
 t ^{ n } G^{ \alpha_{ p}}_{\pm}\tonde*{ l^{ A} , \Omega } 
 =
 G^{ \alpha_{ p}}_{\pm}\tonde*{ l^{ A} , t \Omega } .
 \end{align}

 For every $ h \in \mathbb{R}^{n} $, 
 \begin{align}
 & \notag
 G^{ \alpha_{ p}}_{ \varepsilon }\tonde*{ u,\Omega } 
 =
 G^{ \alpha_{ p}}_{ \varepsilon }\tonde*{ u\tonde*{\cdot- h } , \Omega + h } ,
 &&
 G^{ \alpha_{ p}}_{\pm}\tonde*{ u,\Omega } 
 =
 G^{ \alpha_{ p}}_{\pm}\tonde*{ u\tonde*{\cdot- h } , \Omega + h } ,
 \\
 & \label{eq:G_lA-trans}
 G^{ \alpha_{ p}}_{ \varepsilon }\tonde*{ l^{ A} , \Omega } 
 =
 G^{ \alpha_{ p}}_{ \varepsilon }\tonde*{ l^{ A} , \Omega + h } ,
 &&
 G^{ \alpha_{ p}}_{\pm}\tonde*{ l^{ A} , \Omega } 
 =
 G^{ \alpha_{ p}}_{\pm}\tonde*{ l^{ A} , \Omega + h } .
 \end{align}
\end{proposition}

\begin{proposition}[][res:prop_funz_insieme]
 Let $ \Omega \subseteq \mathbb{R}^{n} $ be an open set, $ m \geq1$ a natural number, $ p \in \left[1,\infty\right) $,
 $ \mathcal{G} $ a subgroup of $ \mathrm{SO}\tonde*{n} $ and $ D \subseteq \mathbb{R}^{n} $ be a bounded open set.
 Let $ \alpha_{ p} = \alpha_{ p , D, \Gamma ^{\mathcal{G}} ,m } $ 
 be a $p$-core functional as in \cref{def:core_function}
 subordinated to $ \Gamma ^{\mathcal{G}} $ as in \cref{def:Gamma_G}.
 Let $ \varepsilon >0$ and $ u \in L^{p}_{\mathrm{loc}}\tonde*{\mathbb{R}^{n};\mathbb{R}^{m}} $.

 Recalling \cref{def:tipo_BMO}, we define the functionals acting on open subsets of $\Omega$,
 \begin{align*}
 & \mathcal{A}_{\Omega} \ni \tilde{\Omega} \mapsto g_{\varepsilon}\tonde*{ \tilde{\Omega}} \coloneqq G^{ \alpha_{ p}}_{ \varepsilon }\tonde*{ u , \tilde{\Omega}} ,
 & \mathcal{A}_{\Omega} \ni \tilde{\Omega} \mapsto g_{\pm}\tonde*{ \tilde{\Omega}} \coloneqq G^{ \alpha_{ p}}_{\pm}\tonde*{ u , \tilde{\Omega}} .
 \end{align*}
 Then,
 \begin{enumerate}
 \item\label{it:decreasing} $ g_{\varepsilon} , g_{\pm} $ are monotone non-decreasing with respect to set inclusion,
 \item\label{it:superadditive} $ g_{\varepsilon} , g_{-} $ are superadditive:
 \begin{align*}
 & g_{\varepsilon}\tonde*{ \Omega_{1} \cup \Omega_{2}} \geq g_{\varepsilon}\tonde*{ \Omega_{1}} + g_{\varepsilon}\tonde*{ \Omega_{2}} ,
 & g_{-}\tonde*{ \Omega_{1} \cup \Omega_{2}} \geq g_{-}\tonde*{ \Omega_{1}} + g_{-}\tonde*{ \Omega_{2}} ,
 \end{align*}
 for every $ \Omega_{1} , \Omega_{2} \in \mathcal{A}_{\Omega} $ s.t. $ \Omega_{1} \cap \Omega_{2} =\emptyset$.
 \item if $ \nabla \restr{ u }{ \Omega } \in L^{p}\tonde*{\Omega;\mathbb{R}^{m\times n}} $, then,
 \begin{enumerate}[label=\textnormal{\textit{(\roman*)}}]
 \item\label{it:supremum} $ g_{+} $ is inner-regular in the following sense:
 \begin{align}
 g_{+}\tonde*{ \tilde{\Omega}} &= \sup\limits_{ } \set*{ g_{+}\tonde*{ \Omega' } : \Omega' \Subset \tilde{\Omega} , \Omega' \in \mathcal{A}_{\Omega}}\text{ for all } \tilde{\Omega} \in \mathcal{A}_{\Omega} .\label{eq:sup_gpiu}
 \end{align}
 \item\label{it:subadditive} $ g_{+} $ is $ \sigma $-subadditive:
 \begin{equation*}
 g_{+}\tonde*{ \bigcup_{ k \in \mathbb{N}} \Omega_{ k }} \leq\sommatoria[ k \in \mathbb{N}]{ g_{+}\tonde*{ \Omega_{ k }}},
 \end{equation*}
 for all sequences $ \graffe*{\Omega_{k}}_{k\in\mathbb{N}} \subset \mathcal{A}_{\Omega} $.
 \end{enumerate}
 \end{enumerate}
\end{proposition}
\begin{proof}
 It is elementary to obtain \eqref{it:decreasing} and \eqref{it:superadditive}.
 For the proof of \ref{it:supremum} and \ref{it:subadditive} we refer
 to the proof of \cite[Prop. 3.1]{FarFusGuaSch20-MR4062330}
 taking into account \cref{res:tip_BMO_bound}.
\end{proof}
We observe that, for a given linear map $ l^{ A}\tonde*{ x } = A x $,
the functionals $ G^{ \alpha_{ p}}_{\pm}\tonde*{ l^{ A} ,\cdot } $
coincide if they act on any unitary cube centered in the origin.
The following proposition is obtained by retracing \cite[Prop. 3.2]{FarFusGuaSch20-MR4062330}.
\begin{proposition}[][res:well-def_psi]
 Let $ \Omega \subseteq \mathbb{R}^{n} $ be an open set, $ m \geq1$ a natural number, $ p \in \left[1,\infty\right) $,
 $ \mathcal{G} $ a subgroup of $ \mathrm{SO}\tonde*{n} $ and $ D \subseteq \mathbb{R}^{n} $ be a bounded open set.
 Let $ \alpha_{ p} = \alpha_{ p , D, \Gamma ^{\mathcal{G}} ,m } $ 
 be a $p$-core functional as in \cref{def:core_function}
 subordinated to $ \Gamma ^{\mathcal{G}} $ as in \cref{def:Gamma_G}.
 For any unitary cube $ \tilde{Q} $ centered in the origin,
 \begin{equation*}
 \sup\limits_{0< \varepsilon \leq1} G^{ \alpha_{ p}}_{ \varepsilon }\tonde*{ l^{ A} , \tilde{Q}} 
 =
 G^{ \alpha_{ p}}\tonde*{ l^{ A} , \tilde{Q}} 
 =
 G^{ \alpha_{ p}}\tonde*{ l^{ A} , Q } ,
 \end{equation*}
 for all $ A \in \mathbb{R}^{m\times n} $.
\end{proposition}
We can now define the function $ \psi_{\alpha_{p}} : \mathbb{R}^{m\times n} \to \left[0,+\infty\right) $.
\begin{definition}[][def:psi_alphap]
 Let $ \Omega \subseteq \mathbb{R}^{n} $ be an open set, $ m \geq1$ a natural number, $ p \in \left[1,\infty\right) $,
 $ \mathcal{G} $ a subgroup of $ \mathrm{SO}\tonde*{n} $ and $ D \subseteq \mathbb{R}^{n} $ be a bounded open set.
 Let $ \alpha_{ p} = \alpha_{ p , D, \Gamma ^{\mathcal{G}} ,m } $ 
 be a $p$-core functional as in \cref{def:core_function}
 subordinated to $ \Gamma ^{\mathcal{G}} $ as in \cref{def:Gamma_G}.
 We define,
 \begin{equation*}
 \psi_{\alpha_{p}}\tonde*{ A } 
 \coloneqq
 G^{ \alpha_{ p}}\tonde*{ l^{ A} , Q } ,
 \end{equation*}
 for every $ A \in \mathbb{R}^{m\times n} $.
\end{definition}
The next result shows that $ \psi_{\alpha_{p}} $
is well defined, i.e. the values of the function do
not depend on the choice of the cube $ Q $ at the right hand side of \eqref{eq:psi_ind_cubo},
and lists some useful properties.
\begin{proposition}[][res:prop_psi_alphap]
 Under the assumptions of \cref{def:psi_alphap},
 $ \psi_{\alpha_{p}} $ is a well-defined function with values in $ \left[0,+\infty\right) $
 and for any unitary cube $ \tilde{Q} $ centered in the origin
 \begin{equation}\label{eq:psi_ind_cubo}
 \psi_{\alpha_{p}}\tonde*{ A } = G^{ \alpha_{ p}}\tonde*{ l^{ A} , \tilde{Q}} ,
 \end{equation}
 for all $ A \in \mathbb{R}^{m\times n} $.

 Moreover, $ \psi_{\alpha_{p}} $ is $ p $-homogeneous, convex and locally Lipschitz.
 Precisely,
 \begin{equation}\label{eq:loc_lip_fine}
 \abs*{ \psi_{\alpha_{p}}\tonde*{ A } - \psi_{\alpha_{p}}\tonde*{ B }}
 \leq
 \frac{ C_{ \Sigma }}{ \abs*{ D }} \abs*{ A - B }_{2} 
 \text{ }
 \forall A , B \in \Sigma ,
 \end{equation}
 where $ \Sigma $ and $ C_{ \Sigma } >0$ are as in \cref{res:Amat_loc_lip}.

 Furthermore, the set
 \begin{equation*}
 \mathcal{N}_{\psi_{\alpha_{p}}} \coloneqq\set*{ A \in \mathbb{R}^{m\times n} : \psi_{\alpha_{p}}\tonde*{ A } =0},
 \end{equation*}
 is a linear subspace of $ \mathbb{R}^{m\times n} $.

 Finally, $\psi_{\alpha_p}^{1/p}$ is a seminorm and \begin{equation}\label{psi_alphap_tran_null}
 \psi_{\alpha_{p}}\tonde*{ A } = \psi_{\alpha_{p}}\tonde*{ A + B } ,
 \end{equation}
 for every $ A \in \mathbb{R}^{m\times n} $ and $ B \in \mathcal{N}_{\psi_{\alpha_{p}}} $.
\end{proposition}

\begin{proof}
 $ \psi_{\alpha_{p}} $ is well-defined and \eqref{eq:psi_ind_cubo} holds as a direct consequence of \cref{res:well-def_psi},
 while $ \psi_{\alpha_{p}}\tonde*{ A } \in \left[0,+\infty\right) $ for every $ A \in \mathbb{R}^{m\times n} $ by \cref{res:tip_BMO_bound}.
 
We now prove that $ \psi_{\alpha_{p}} $ is $ p $-homogeneous.
We fix $ t \in \mathbb{R} $ and we notice that
\begin{equation*}
 \psi_{\alpha_{p}}\tonde*{ t A } 
 =
 \lim\limits_{\varepsilon\to 0^{+}} G^{\alpha_{p}}_{ \varepsilon }\tonde*{ l^{ t A } , Q } 
 =
 \abs*{ t }^ p \lim\limits_{\varepsilon\to 0^{+}} G^{\alpha_{p}}_{ \varepsilon }\tonde*{ l^{ A} , Q } 
 =
 \abs*{ t }^ p \psi_{\alpha_{p}}\tonde*{ A } ,
\end{equation*}
thanks to \eqref{alpha_p:p-homogeneous}.

We prove that $ \psi_{\alpha_{p}} $ is convex.
Let $ t \in \left[0,1\right] $ and $ A , B \in \mathbb{R}^{m\times n} $.
Let $ \varepsilon >0$ and $ \mathcal{G}_{\varepsilon} $ be a family made of pairwise 
disjoint elements of $ \mathcal{G}^{ D}_{\varepsilon}\tonde*{ Q} $.
Then,
\begin{align*}
 & \varepsilon ^{ n - p } \sum_{ D'\in\mathcal{G}_{\varepsilon}} \alpha_{ p}\tonde*{ t l^{ A} +(1- t ) l^{ B}, D' } 
 \\
 &\leq
 t \varepsilon ^{ n - p } \sum_{ D'\in\mathcal{G}_{\varepsilon}} \alpha_{ p}\tonde*{ l^{ A}, D' } 
 +
 (1- t ) \varepsilon ^{ n - p } \sum_{ D'\in\mathcal{G}_{\varepsilon}} \alpha_{ p}\tonde*{ l^{ B}, D' } 
 \\
 &\leq
 t G^{\alpha_{p}}_{ \varepsilon }\tonde*{ l^{ A} , Q } 
 +
 (1- t ) G^{\alpha_{p}}_{ \varepsilon }\tonde*{ l^{ B} , Q } ,
\end{align*}
where we used \eqref{alpha_p:conv}.
Taking the supremum with respect to $ \mathcal{G}_{\varepsilon} $
we conclude that $ \mathbb{R}^{m\times n} \ni A \mapsto G^{\alpha_{p}}_{ \varepsilon }\tonde*{ l^{ A} , Q } \in \left[0,+\infty\right) $
is convex.
Letting $ \varepsilon \to 0^{+} $ this implies that $ \psi_{\alpha_{p}} $
is convex.

From the convexity of $ \psi_{\alpha_{p}} $, being $ \mathbb{R}^{m\times n} $ a finite dimension space,
we immediately deduce that $ \psi_{\alpha_{p}} $ is locally Lipschitz.

Nevertheless, we now prove \eqref{eq:loc_lip_fine}
to show how the locally Lipschitz
constant of $ \psi_{\alpha_{p}} $ on the set
$ \Sigma \coloneqq\set*{ A \in \mathbb{R}^{m\times n} : \abs*{ A }_{2} \leq1}$
is related to the one of $ \mathbb{R}^{m\times n} \ni A \mapsto \alpha_{ p}\tonde*{ l^{ A}, D } $ on $ \Sigma $.

Let $ A , B \in \Sigma $
and $ \delta >0$.
By definition of $ \psi_{\alpha_{p}} $ there exists $ \varepsilon_{0} >0$ such that,
\begin{equation}\label{eq:Psip_loclip1}
 \psi_{\alpha_{p}}\tonde*{ A } - \psi_{\alpha_{p}}\tonde*{ B } 
 \leq
 G^{\alpha_{p}}_{ \varepsilon }\tonde*{ l^{ A} , Q } - G^{\alpha_{p}}_{ \varepsilon }\tonde*{ l^{ B} , Q } + \delta 
 \text{ }
 \forall
 0< \varepsilon < \varepsilon_{0} .
\end{equation}
Moreover, by definition of
$ G^{\alpha_{p}}_{ \varepsilon }\tonde*{ l^{ A} , Q } $ and $ G^{\alpha_{p}}_{ \varepsilon }\tonde*{ l^{ B} , Q } $,
for every $ \varepsilon >0$
there exists a family $ \mathcal{G}_{\varepsilon} $, made of pairwise disjoint elements
of $ \mathcal{G}^{ D}_{\varepsilon}\tonde*{ Q} $, whose cardinality
does not exceed $\tonde*{ \varepsilon ^ n \abs*{ D }}^{-1}$ and
\begin{equation}\label{eq:Psip_loclip2}
 G^{\alpha_{p}}_{ \varepsilon }\tonde*{ l^{ A} , Q } - G^{\alpha_{p}}_{ \varepsilon }\tonde*{ l^{ B} , Q } 
 \leq
 \varepsilon ^{ n - p } \sup\limits_{\mathcal{G}_{\varepsilon}} \sum_{ D'\in\mathcal{G}_{\varepsilon}}\tonde*{ \alpha_{ p}\tonde*{ l^{ A}, D' } - \alpha_{ p}\tonde*{ l^{ B}, D' }} 
 + \delta .
\end{equation}
Combining \eqref{eq:Psip_loclip1} and \eqref{eq:Psip_loclip2}
with \eqref{eq:loclip_keyfinal}
we find,
\begin{equation*}
 \psi_{\alpha_{p}}\tonde*{ A } - \psi_{\alpha_{p}}\tonde*{ B } 
 \leq
 \frac{ C_{ \Sigma }}{ \abs*{ D }} \abs*{ A - B }_{2} +2 \delta ,
\end{equation*}
where $ C_{ \Sigma } $ is as in \cref{res:Amat_loc_lip}.
By letting $ \delta \to0^+$ and then interchanging the role of
$ A $ and $ B $ we deduce that \eqref{eq:loc_lip_fine} holds.

The conclusion of the proof is an immediate consequence of \cref{res:null_cx_p_hom}.
\end{proof}

\begin{remark}[][res:psi_alphap_Pspace]
 Under the assumptions of \cref{def:psi_alphap},
 as a consequence of \cref{res:prop_psi_alphap},
 for every $ \mathcal{P} $ linear subspace of $ \mathbb{R}^{m\times n} $ such that,
 \begin{equation}\label{eq:dir_sum_dec}
 \mathbb{R}^{m\times n} = \mathcal{P} \oplus \mathcal{N}_{\psi_{\alpha_{p}}} 
 \end{equation}
 it follows that $\psi_{\alpha_p}^{1/p}$ restricted to $\mathcal{P}$ is a norm and for   every $ A \in \mathbb{R}^{m\times n} $,
 \begin{equation*}
 \psi_{\alpha_{p}}\tonde*{ \pi_{\mathcal{P}}\tonde*{ A }} = \psi_{\alpha_{p}}\tonde*{ A } ,
 \end{equation*}
 where $ \pi_{\mathcal{P}} : \mathbb{R}^{m\times n} \to \mathcal{P} $ is
 the projection on $ \mathcal{P} $ (associated to the direct sum decomposition in \eqref{eq:dir_sum_dec}).
 Then, under the assumptions of \cref{thm:Main_BMO},
 \eqref{eq:tesi_Main_BMO} is equivalent to
 \begin{equation}\label{eq:tesi_Main_BMO2}
 \lim\limits_{\varepsilon\to 0^{+}} G^{ \alpha_{ p}}_{ \varepsilon }\tonde*{ u,\Omega } = \int_{ \Omega } \psi_{\alpha_{p}}\tonde*{ \pi_{\mathcal{P}}\tonde*{ \nabla u }} \thinspace\mathrm{d} x .
 \end{equation}
 Moreover, making use of \cref{res:prop_psi_alphap} again, for every $ u \in W^{1,p}_{\mathrm{loc}}\tonde*{\Omega;\mathbb{R}^{m}} $,
 \begin{equation*}
 c \int_{ \Omega } \abs*{ \pi_{\mathcal{P}}\tonde*{ \nabla u }}^{ p }_{2} \thinspace\mathrm{d} x 
 \leq
 \int_{ \Omega } \psi_{\alpha_{p}}\tonde*{ \pi_{\mathcal{P}}\tonde*{ \nabla u }} \thinspace\mathrm{d} x 
 \leq
 C \int_{ \Omega } \abs*{ \pi_{\mathcal{P}}\tonde*{ \nabla u }}^{ p }_{2} \thinspace\mathrm{d} x ,
 \end{equation*}
 for some $ c , C >0$ depending on $ \alpha_{ p} $.
\end{remark}

Now we are ready to prove our main result.

\begin{proof}[Proof of \cref{thm:Main_BMO}]
We divide the proof in several steps. 

\hypertarget{STEP1thm:Main_BMO}{\textbf{Step 1.}}
We first assume that $ \Omega $
is a bounded open set and $\restr{ u }{ \overline{\Omega}}\in C^{1}\tonde*{\overline{\Omega};\mathbb{R}^{m}} $.
For every $ t >0$, we define $ U_{t} \coloneqq\set*{ x \in \Omega : \abs*{ \nabla u\tonde*{ x }}_{2} > t }$.
Let $ \sigma \in \left(0,1\right) $ and $ t >0$ s.t $ \abs*{ \partial U_{t}} =0$ (this happens for every $ t >0$
except for countably many).
Let $ r >0$, $ \graffe*{ Q\tonde*{x_{i};r}}_{1\leq i \leq l } $ and $ W_{t, \sigma } $ as in \cref{res:cover_cubes}.

In order to prove \eqref{eq:tesi_Main_BMO} for $ u $,
we consider two substeps.

\textbf{Step 1.1.} First we prove that 
\begin{equation}\label{Hliminf}
 G^{ \alpha_{ p}}_{ -}\tonde*{ u,\Omega } \geq \int_{ \Omega } \psi_{\alpha_{p}}\tonde*{ \nabla u\tonde*{ x }} \thinspace\mathrm{d} x .
\end{equation}

Fix $ i \in\set*{1,\dots, l }$ and $ \varepsilon >0$ and consider a family $ \mathcal{G}_{\varepsilon} $ as in \cref{def:tipo_BMO}
with $ Q\tonde*{x_{i};r} $ as ambient space.
This means that $ \mathcal{G}_{\varepsilon} $ is a family of pairwise disjoint sets $ \graffe*{ D_{j}}_{1\leq j \leq k } $,
where for every $ j \in\set*{1,\dots, k }$ we have 
$ D_{j} = z_{j} + \varepsilon \mathcal{R}_{j} D \subset Q\tonde*{x_{i};r} $, with
$ \mathcal{R}_{j} \in \mathcal{G} $, $ z_{j} \in \mathbb{R}^{n} $ and
\begin{equation}\label{stima_su_k}
 k \leq \abs*{ Q\tonde*{x_{i};r}} /\tonde*{ \varepsilon ^ n \abs*{ D }}= r ^ n /\tonde*{ \varepsilon ^ n \abs*{ D }}.
\end{equation}

For every $ j \in\set*{1,\dots, k }$ and $ x \in D_{j} $ we may write
\begin{equation*}
 u\tonde*{ x } = u\tonde*{ z_{j}} + \nabla u\tonde*{ z_{j}} \tonde*{ x - z_{j}}+ R_{j}\tonde*{ x } , 
\end{equation*}
with $ R_{j}\tonde*{ x } \coloneqq( \nabla u\tonde*{ \bar{x}} - \nabla u\tonde*{ z_{j}} )( x - z_{j} )$
for all $ x \in \mathbb{R}^{n} $, where $ \bar{x} \in Q\tonde*{x_{i};r} $. Thus,
\begin{equation}\label{Hcubi}
\begin{split}
 G^{ \alpha_{ p}}_{ \varepsilon }\tonde*{ u , Q\tonde*{x_{i};r}} 
&\geq \varepsilon ^{ n - p } \sommatoria[ j =1][ k ]{ \alpha_{ p}\tonde*{ u , D_{j}}} \\
&= \varepsilon ^{ n - p }\sommatoria[ j =1][ k ]{ \alpha_{ p}\tonde*{ \nabla u\tonde*{ z_{j}} \tonde*{\cdot- z_{j}}+ R_{j} , D_{j}}}\\
& \geq \frac{ \varepsilon ^{ n - p }}{(1+ \delta )^ p } \sommatoria[ j =1][ k ]{ \alpha_{ p}\tonde*{ \nabla u\tonde*{ x_{i}} \tonde*{\cdot- z_{j}}, D_{j}}}
-\frac{ \varepsilon ^{ n - p }}{ \delta ^ p }\sommatoria[ j =1][ k ]{ \alpha_{ p}\tonde*{ \bar{R}_{j} , D_{j}}}\\
& = \frac{ \varepsilon ^{ n - p }}{(1+ \delta )^ p } \sommatoria[ j =1][ k ]{ \alpha_{ p}\tonde*{ l^{ \nabla u\tonde*{ x_{i}}}, D_{j}}}
-\frac{ \varepsilon ^{ n - p }}{ \delta ^ p }\sommatoria[ j =1][ k ]{ \alpha_{ p}\tonde*{ \bar{R}_{j} , D_{j}}}
\end{split}
\end{equation}
where in the first equality we use \eqref{alpha_p:outer_trans_inv} and \eqref{alpha_p:ext_ind},
in the last equality we use \eqref{alpha_p:outer_trans_inv},
while in the last inequality we use \eqref{alpha_p:utile_2} having set, 
\begin{align*}
 \bar{R}_{j}\tonde*{ x } &\coloneqq( \nabla u\tonde*{ x_{i}} - \nabla u\tonde*{ z_{j}} )( x - z_{j} )- R_{j}\tonde*{ x } \\
 &=( \nabla u\tonde*{ x_{i}} - \nabla u\tonde*{ \bar{x}} )( x - z_{j} ),
\end{align*}
for every $ x \in \mathbb{R}^{n} $.

Using \eqref{alpha_p:poincare_+} and part \ref{item:cover_cubes_b} in \cref{res:cover_cubes} , we estimate
\begin{equation}\label{stima_resti}
\begin{split}
\frac{1}{ \delta ^ p }
\sommatoria[ j =1][ k ]{ \alpha_{ p}\tonde*{ \bar{R}_{j} , D_{j}}} 
&\leq \frac{ c_{2}}{ \delta ^ p }\sommatoria[ j =1][ k ]{ \operatorname{diam}\tonde*{ D_{j}} ^{ p - n } \int_{ D_{j}} \abs*{ \nabla u\tonde*{ x_{i}} - \nabla u\tonde*{ \bar{x}}}^{ p }_{2} \thinspace\mathrm{d} x }\\
&\leq \frac{ c_{2}}{ \delta ^ p } \varepsilon ^{ p - n }\sommatoria[ j =1][ k ]{ \operatorname{diam}\tonde*{ D } ^{ p - n } \sigma ^{ p } \varepsilon ^ n \abs*{ D }}\\
&\leq \frac{ c_{2}}{ \delta ^ p } \operatorname{diam}\tonde*{ D } ^{ p - n } \varepsilon ^{ p - n } \sigma ^{ p } r ^ n ,
\end{split}
\end{equation}
where in the last inequality we used \eqref{stima_su_k} and $ c_{2} >0$ is as in \eqref{alpha_p:poincare_+}.
Thus combining \eqref{Hcubi} and \eqref{stima_resti},
taking the supremum with respect to all families $ \mathcal{G}_{\varepsilon} $ and the $\liminf$ with respect to $ \varepsilon $, we have 
\begin{equation}\label{etichetta}
 \begin{split}
 G^{ \alpha_{ p}}_{ -}\tonde*{ u , Q\tonde*{x_{i};r}} 
 &\geq
 \frac{1}{(1+ \delta )^ p } G^{ \alpha_{ p}}_{ -}\tonde*{ l^{ \nabla u\tonde*{ x_{i}}} , Q\tonde*{x_{i};r}} 
 -
 \frac{ c_{2}}{ \delta ^ p } \operatorname{diam}\tonde*{ D } ^{ p - n } \sigma ^{ p } r ^ n \\
 &=
 \frac{1}{(1+ \delta )^ p } r ^ n \psi_{\alpha_{p}}\tonde*{ \nabla u\tonde*{ x_{i}}} 
 -
 \frac{ c_{2}}{ \delta ^ p } \operatorname{diam}\tonde*{ D } ^{ p - n } \sigma ^{ p } r ^ n ,
 \end{split}
\end{equation}
where in the last inequality we used \eqref{eq:G_lA-scaling} and \eqref{eq:G_lA-trans}.

Now we observe that,
\begin{equation}\label{spezzo}
 \begin{split}
 r ^ n \psi_{\alpha_{p}}\tonde*{ \nabla u\tonde*{ x_{i}}} 
 &=
 \int_{ Q\tonde*{x_{i};r}} \psi_{\alpha_{p}}\tonde*{ \nabla u\tonde*{ x }} \thinspace\mathrm{d} x 
 \\
 &+
 \int_{ Q\tonde*{x_{i};r}} \psi_{\alpha_{p}}\tonde*{ \nabla u\tonde*{ x_{i}}} - \psi_{\alpha_{p}}\tonde*{ \nabla u\tonde*{ x }} \thinspace\mathrm{d} x 
 \end{split}
\end{equation}
Recalling \cref{res:prop_psi_alphap},
\eqref{eq:differenze_p} and \eqref{eq:diff_normalizzati} we observe that,
\begin{align*}
 & \int_{ Q\tonde*{x_{i};r}} \abs*{ \psi_{\alpha_{p}}\tonde*{ \nabla u\tonde*{ x_{i}}} - \psi_{\alpha_{p}}\tonde*{ \nabla u\tonde*{ x }}} \thinspace\mathrm{d} x 
 \\
 &=
 \int_{ Q\tonde*{x_{i};r}} \abs*{ \abs*{ \nabla u\tonde*{ x_{i}}}^{ p }_{2} \psi_{\alpha_{p}}\tonde*{\frac{ \nabla u\tonde*{ x_{i}}}{ \abs*{ \nabla u\tonde*{ x_{i}}}_{2}}} 
 - \abs*{ \nabla u\tonde*{ x }}^{ p }_{2} \psi_{\alpha_{p}}\tonde*{\frac{ \nabla u\tonde*{ x }}{ \abs*{ \nabla u\tonde*{ x }}_{2}}}} \thinspace\mathrm{d} x 
 \\
 &\leq
 \int_{ Q\tonde*{x_{i};r}} 
 \abs*{ \nabla u\tonde*{ x_{i}}}^{ p }_{2} \abs*{ \psi_{\alpha_{p}}\tonde*{\frac{ \nabla u\tonde*{ x_{i}}}{ \abs*{ \nabla u\tonde*{ x_{i}}}_{2}}} - \psi_{\alpha_{p}}\tonde*{\frac{ \nabla u\tonde*{ x }}{ \abs*{ \nabla u\tonde*{ x }}_{2}}}}
 \thinspace\mathrm{d} x 
 \\
 &+
 \int_{ Q\tonde*{x_{i};r}} 
 \psi_{\alpha_{p}}\tonde*{\frac{ \nabla u\tonde*{ x }}{ \abs*{ \nabla u\tonde*{ x }}_{2}}} \abs*{ \abs*{ \nabla u\tonde*{ x_{i}}}^{ p }_{2} - \abs*{ \nabla u\tonde*{ x }}^{ p }_{2}}
 \thinspace\mathrm{d} x 
 \\
 &\leq
 \frac{ C_{ \Sigma }}{ \abs*{ D }}
 \int_{ Q\tonde*{x_{i};r}} 
 \abs*{ \nabla u\tonde*{ x_{i}}}^{ p }_{2} \abs*{\frac{ \nabla u\tonde*{ x_{i}}}{ \abs*{ \nabla u\tonde*{ x_{i}}}_{2}}-\frac{ \nabla u\tonde*{ x }}{ \abs*{ \nabla u\tonde*{ x }}_{2}}}_{2} 
 \thinspace\mathrm{d} x 
 \\
 &+
 \norm*{ \psi_{\alpha_{p}}}_{L^{\infty}\tonde*{ \Sigma }} 
 \int_{ Q\tonde*{x_{i};r}} 
 \abs*{ \abs*{ \nabla u\tonde*{ x_{i}}}^{ p }_{2} - \abs*{ \nabla u\tonde*{ x }}^{ p }_{2}}
 \thinspace\mathrm{d} x ,
 \\
 &\leq
 2\frac{ C_{ \Sigma }}{ \abs*{ D }}
 \norm*{ \abs*{ \nabla u }_{2}}^{ p -1}_{L^{\infty}\tonde*{\Omega}} 
 \int_{ Q\tonde*{x_{i};r}} 
 \abs*{ \nabla u\tonde*{ x_{i}} - \nabla u\tonde*{ x }}_{2} 
 \thinspace\mathrm{d} x 
 \\
 &+
 p \tonde*{2 \norm*{ \abs*{ \nabla u }_{2}}_{L^{\infty}\tonde*{\Omega}}}^{ p -1}
 \norm*{ \psi_{\alpha_{p}}}_{L^{\infty}\tonde*{ \Sigma }} 
 \int_{ Q\tonde*{x_{i};r}} 
 \abs*{ \nabla u\tonde*{ x_{i}} - \nabla u\tonde*{ x }}_{2} 
 \thinspace\mathrm{d} x ,
 \\
 &\leq
 C \sigma r ^ n 
\end{align*}
where $ \Sigma \coloneqq\set*{ A \in \mathbb{R}^{m\times n} : \abs*{ A }_{2} \leq1}$, $ C_{ \Sigma } >0$
is as in the proof of \cref{res:prop_psi_alphap} and
$ C >0$ is a constant depending on $ \alpha_{ p} $ and $ \nabla u $.

Recalling \eqref{spezzo} and \eqref{etichetta}, we then get
\begin{equation*}
 G^{ \alpha_{ p}}_{ -}\tonde*{ u , Q\tonde*{x_{i};r}} 
 \geq
 \frac{1}{(1+ \delta )^ p } \int_{ Q\tonde*{x_{i};r}} \psi_{\alpha_{p}}\tonde*{ \nabla u\tonde*{ x }} \thinspace\mathrm{d} x 
 -
 C \sigma r ^ n ,
\end{equation*}
where $ C >0$ is a constant depending on $ \alpha_{ p} $, $ \nabla u $ and $ \delta $.

Thus, summing up with respect to $ i $, by the monotonicity and the superadditivity
of $ G^{ \alpha_{ p}}_{ -}\tonde*{ u ,\cdot } $ (recall \cref{res:prop_funz_insieme}),
we get
\begin{equation*}
 \begin{split}
 G^{ \alpha_{ p}}_{ -}\tonde*{ u , \Omega } 
 \geq
 \sommatoria[ i =1][ l ]{ G^{ \alpha_{ p}}_{ -}\tonde*{ u , Q\tonde*{x_{i};r}}}
 & \geq 
 \frac{1}{(1+ \delta )^ p } 
 \sommatoria[ i =1][ l ]{ \int_{ Q\tonde*{x_{i};r}} \psi_{\alpha_{p}}\tonde*{ \nabla u\tonde*{ x }} \thinspace\mathrm{d} x - C \sigma \abs*{ \Omega }}
 \\& \geq 
 \frac{1}{(1+ \delta )^ p }
 \int_{ U_{t}} \psi_{\alpha_{p}}\tonde*{ \nabla u\tonde*{ x }} \thinspace\mathrm{d} x - C \sigma ,
 \end{split}
\end{equation*}
where the last inequality follows from part \ref{item:cover_cubes_a} in
\cref{res:cover_cubes} and the constant $ C $ depends on 
$ \alpha_{ p} $, $ \nabla u $, $ \abs*{ \Omega } $ and $ \delta $. Then \eqref{Hliminf} follows
at once letting first $ \sigma \to0$, $ t \to0$ and then $ \delta \to 0$ in the previous inequality.

\textbf{Step 1.2.} We prove now,
\begin{equation}\label{Hlimsup}
 G^{ \alpha_{ p}}_{ +}\tonde*{ u,\Omega } \leq \int_{ \Omega } \psi_{\alpha_{p}}\tonde*{ \nabla u\tonde*{ x }} \thinspace\mathrm{d} x .
\end{equation}

Using the subadditivity of $ G^{ \alpha_{ p}}_{ +}\tonde*{ u ,\cdot } $ (recall \cref{res:prop_funz_insieme}) we have
\begin{equation}\label{a1}
 G^{ \alpha_{ p}}_{ +}\tonde*{ u , \Omega } 
 \leq
 \sommatoria[ i =1][ l ]{ G^{ \alpha_{ p}}_{ +}\tonde*{ u , Q\tonde*{x_{i};r}}}
 +
 G^{ \alpha_{ p}}_{ +}\tonde*{ u , \Omega \setminus \overline{ U_{t}}} 
 +
 G^{ \alpha_{ p}}_{ +}\tonde*{ u , W_{t, \sigma }} .
\end{equation}
Let us estimate the three terms in the last inequality.
Arguing as in the previous step, for every $i=1,\dots, l $ we get that, 
\begin{equation}\label{a2}
 G^{ \alpha_{ p}}_{ +}\tonde*{ u , Q\tonde*{x_{i};r}} 
 \leq
 \tonde*{1+ \delta }^ p \int_{ Q\tonde*{x_{i};r}} \psi_{\alpha_{p}}\tonde*{ \nabla u\tonde*{ x }} \thinspace\mathrm{d} x 
 +
 C \sigma r ^ n ,
\end{equation}
for some positive constant $ C $ depending only on $ \alpha_{ p} $, $ \nabla u $ and and $ \delta $. 

By using \cref{res:tip_BMO_bound} we get
\begin{equation}\label{a3}
 \begin{split}
 & G^{ \alpha_{ p}}_{ +}\tonde*{ u , \Omega \setminus \overline{ U_{t}}} 
 +
 G^{ \alpha_{ p}}_{ +}\tonde*{ u , W_{t, \sigma }} 
 \\
 &\leq
 c_{1} \int_{\set*{ \abs*{ \nabla u }_{2} \leq t }} \abs*{ \nabla u }^{ p }_{2} \thinspace\mathrm{d} x 
 +
 c_{1} \int_{ W_{t, \sigma }} \abs*{ \nabla u }^{ p }_{2} \thinspace\mathrm{d} x 
 \leq
 c_{1} \tonde*{ \abs*{ \Omega } t ^ p + \sigma \norm*{ \nabla u }^{ p }_{L^{\infty}\tonde*{\Omega}}}
 \end{split}
\end{equation}
Then \eqref{Hlimsup} follows combining \eqref{a1}, \eqref{a2}, \eqref{a3} and letting first $ \sigma \to0$, $ t \to0$ and then $ \delta \to0$.

Finally, \eqref{Hliminf} and \eqref{Hlimsup} imply that \eqref{eq:tesi_Main_BMO} holds under
the assumptions on $ u $ made in this first step.

\hypertarget{STEP2thm:Main_BMO}{\textbf{Step 2.}}
We assume that $ \Omega $
is a bounded open set and
that $ u \in L^{p}_{\mathrm{loc}}\tonde*{\mathbb{R}^{n};\mathbb{R}^{m}} $ with $ \nabla \restr{ u }{ \Omega } \in L^{p}\tonde*{\Omega;\mathbb{R}^{m\times n}} $.

Thanks to the previous step we know that,
\begin{equation}\label{a48d0}
 \lim\limits_{\varepsilon\to 0^{+}} G^{ \alpha_{ p}}_{ \varepsilon }\tonde*{ v , \Omega' } 
 =
 \int_{ \Omega' } \psi_{\alpha_{p}}\tonde*{ \nabla v } \thinspace\mathrm{d} x ,
\end{equation}
holds for every $ \Omega' \in \mathcal{A}_{\Omega} $ and for every 
$ v \in L^{p}_{\mathrm{loc}}\tonde*{\mathbb{R}^{n};\mathbb{R}^{m}} $ s.t. $\restr{ v }{ \overline{\Omega'}}\in C^{1}\tonde*{\overline{\Omega'};\mathbb{R}^{m}} $.

Let $ \Omega' \in \mathcal{A}_{\Omega} $ s.t $ \Omega' \Subset \Omega $ and $ \sigma >0$.
Then, there exists $ v^{\Omega'}_{ \sigma } \in C^{\infty}_{c}\tonde*{\mathbb{R}^{n};\mathbb{R}^{m}} $ s.t. 
$ \norm*{ u - v^{\Omega'}_{ \sigma }}_{W^{1,p}\tonde*{\Omega';\mathbb{R}^{m}}} < \sigma $.

Fix $ \varepsilon >0$ and consider a family $ \mathcal{G}_{\varepsilon} $ of pairwise disjoint elements of $ \mathcal{G}^{ D}_{\varepsilon}\tonde*{\Omega'} $.
Then,
\begin{equation}\label{nome}
 \begin{split}
 & G^{ \alpha_{ p}}_{ \varepsilon }\tonde*{ u , \Omega' } 
 \geq
 \varepsilon ^{ n - p } \sum_{ D'\in\mathcal{G}_{\varepsilon}} \alpha_{ p}\tonde*{ u , D' } 
 \\
 &\geq
 \frac{1}{\tonde*{1+ \delta }^ p } \varepsilon ^{ n - p } \sum_{ D'\in\mathcal{G}_{\varepsilon}} \alpha_{ p}\tonde*{ v^{\Omega'}_{ \sigma }, D' } 
 -
 \frac{1}{ \delta ^ p } \varepsilon ^{ n - p } \sum_{ D'\in\mathcal{G}_{\varepsilon}} \alpha_{ p}\tonde*{ u - v^{\Omega'}_{ \sigma }, D' } ,
 \end{split}
\end{equation}
where we used \eqref{alpha_p:utile_2}.

Using \eqref{alpha_p:poincare_+}, we estimate,
\begin{equation}\label{stima_resti2}
\begin{split}
\frac{1}{ \delta ^ p }
 \sum_{ D'\in\mathcal{G}_{\varepsilon}} \alpha_{ p}\tonde*{ u - v^{\Omega'}_{ \sigma }, D' } 
&\leq \frac{ c_{2}}{ \delta ^ p } \sum_{ D'\in\mathcal{G}_{\varepsilon}} \operatorname{diam}\tonde*{ D' } ^{ p - n } \int_{ D' } \abs*{ \nabla u - \nabla v^{\Omega'}_{ \sigma }}^{ p }_{2} \thinspace\mathrm{d} x \\
&\leq \frac{ c_{2}}{ \delta ^ p } \operatorname{diam}\tonde*{ D } ^{ p - n } \varepsilon ^{ p - n } \int_{ \Omega' } \abs*{ \nabla u - \nabla v^{\Omega'}_{ \sigma }}^{ p }_{2} \thinspace\mathrm{d} x \\
&\leq \frac{ c_{1}}{ \delta ^ p } \varepsilon ^{ p - n } \sigma ^ p ,
\end{split}
\end{equation}
where in the last inequalities $ c_{2} >0$ is as in \eqref{alpha_p:poincare_+}
and $ c_{1} $ is as in \cref{alpha_p:poincare_type_bound}.
Combining \eqref{nome} with \eqref{stima_resti2}, taking the supremum with respect to $ \mathcal{G}_{\varepsilon} $
and letting $ \varepsilon \to 0^{+} $, we obtain,
\begin{equation}\label{b1}
 \begin{split}
 & G^{ \alpha_{ p}}_{ -}\tonde*{ u,\Omega } 
 \geq
 G^{ \alpha_{ p}}_{ -}\tonde*{ u , \Omega' } 
 \geq
 \frac{1}{\tonde*{1+ \delta }^ p } G^{ \alpha_{ p}}_{ -}\tonde*{ v^{\Omega'}_{ \sigma } , \Omega' } -\frac{ C \sigma ^ p }{ \delta ^ p }
 \\
 &=
 \frac{1}{\tonde*{1+ \delta }^ p } \int_{ \Omega' } \psi_{\alpha_{p}}\tonde*{ \nabla v^{\Omega'}_{ \sigma }\tonde*{ x }} \thinspace\mathrm{d} x -\frac{ C \sigma ^ p }{ \delta ^ p },
 \end{split}
\end{equation}
where in the first inequality we used \cref{res:prop_funz_insieme},
$ C >0$ is a constant depending on $ \alpha_{ p} $
and in the last equality we made use of \eqref{a48d0}.

Similarly it holds,
\begin{equation}\label{b2}
 G^{ \alpha_{ p}}_{ +}\tonde*{ u , \Omega' } 
 \leq
 \tonde*{1+ \delta }^ p \int_{ \Omega' } \psi_{\alpha_{p}}\tonde*{ \nabla v^{\Omega'}_{ \sigma }\tonde*{ x }} \thinspace\mathrm{d} x +\frac{ C \tonde*{1+ \delta }^ p \sigma ^ p }{ \delta ^ p }.
\end{equation}

Thanks to the dominated convergence theorem,
\begin{equation}\label{dominata}
 \lim\limits_{ \sigma \to 0^{+}} \int_{ \Omega' } \psi_{\alpha_{p}}\tonde*{ \nabla v^{\Omega'}_{ \sigma }\tonde*{ x }} \thinspace\mathrm{d} x 
 = \int_{ \Omega' } \psi_{\alpha_{p}}\tonde*{ \nabla u\tonde*{ x }} \thinspace\mathrm{d} x ,
\end{equation}
since $ \lim\limits_{ \sigma \to 0^{+}} \norm*{ u - v^{\Omega'}_{ \sigma }}_{W^{1,p}\tonde*{\Omega';\mathbb{R}^{m}}} =0$, $ \psi_{\alpha_{p}} $ is continuous and
$ \psi_{\alpha_{p}}\tonde*{ A } \leq C \abs*{ A }^{ p }_{2} $ (cfr. \cref{res:prop_psi_alphap})
where $ C >0$ is a constant depending on $ \alpha_{ p} $.
Combining \eqref{b1}, \eqref{b2}, \eqref{dominata}
and letting first $ \sigma \to 0^{+} $, $ \delta \to 0^{+} $
and $ \Omega' \uparrow \Omega $ (recalling also \eqref{eq:sup_gpiu} in \cref{res:prop_funz_insieme})
we deduce that \eqref{eq:tesi_Main_BMO} holds for $ u $.

\hypertarget{STEP3thm:Main_BMO}{\textbf{Step 3.}}
We assume that $ u \in L^{p}_{\mathrm{loc}}\tonde*{\mathbb{R}^{n};\mathbb{R}^{m}} $ with $ \nabla \restr{ u }{ \Omega } \in L^{p}\tonde*{\Omega;\mathbb{R}^{m\times n}} $ and $ \Omega \subseteq \mathbb{R}^{n} $ is an arbitrary nonempty open set.
Let $ \Omega' \in \mathcal{A}_{\Omega} $ s.t. $ \Omega' \Subset \Omega $. By the previous step and \cref{res:prop_funz_insieme} we find,
\begin{equation}\label{asdf}
 G^{ \alpha_{ p}}_{ +}\tonde*{ u , \Omega' } 
 =
 \int_{ \Omega' } \psi_{\alpha_{p}}\tonde*{ \nabla u\tonde*{ x }} \thinspace\mathrm{d} x 
 =
 G^{ \alpha_{ p}}_{ -}\tonde*{ u , \Omega' } \leq G^{ \alpha_{ p}}_{ -}\tonde*{ u,\Omega } .
\end{equation}
\eqref{eq:tesi_Main_BMO} follows by letting $ \Omega' \uparrow \Omega $ in \eqref{asdf}
and using \eqref{eq:sup_gpiu} in \cref{res:prop_funz_insieme}.

\hypertarget{STEP4thm:Main_BMO}{\textbf{Step 4.}}
We assume that $ \psi_{\alpha_{p}} \not\equiv0$
and $ u \in L^{p}_{\mathrm{loc}}\tonde*{\mathbb{R}^{n};\mathbb{R}^{m}} $ with $\restr{ u }{ \Omega }\in W^{1,p}_{\mathrm{loc}}\tonde*{\Omega;\mathbb{R}^{m}} $
and $ \Omega \subseteq \mathbb{R}^{n} $ is an arbitrary nonempty open set.

We may also assume that $ \nabla \restr{ u }{ \Omega } \not\in L^{p}\tonde*{\Omega;\mathbb{R}^{m\times n}} $,
otherwise we can apply the previous step.
Let $ \mathcal{P} $ be a linear subspace of $ \mathbb{R}^{m\times n} $ as in \eqref{eq:dir_sum_dec}
 and $ \pi_{\mathcal{P}} : \mathbb{R}^{m\times n} \to \mathcal{P} $ as in \cref{res:psi_alphap_Pspace}.
 Since $ \psi_{\alpha_{p}} \not\equiv0$, $ \mathcal{P} $ is a non-trivial linear subspace of $ \mathbb{R}^{m\times n} $.
Then, by \cref{res:prop_psi_alphap} and \cref{res:psi_alphap_Pspace},
\begin{equation}\label{89234h}
 \begin{split}
 & \int_{ \Omega } \psi_{\alpha_{p}}\tonde*{ \nabla u } \thinspace\mathrm{d} x 
 \\
 &=
 \int_{ \Omega \setminus\set*{ \abs*{ \nabla u }_{2} =0}} 
 \abs*{ \nabla u }^{ p }_{2} \psi_{\alpha_{p}}\tonde*{ \pi_{\mathcal{P}}\tonde*{\frac{ \nabla u }{ \abs*{ \nabla u }_{2}}}} 
 \thinspace\mathrm{d} x 
 \geq
 C \int_{ \Omega } 
 \abs*{ \nabla u }^{ p }_{2} 
 \thinspace\mathrm{d} x 
 =+\infty,
 \end{split}
\end{equation}
where,
\begin{equation*}
 C \coloneqq\min_{ A \in \mathcal{P} , \abs*{ A }_{2} =1} \psi_{\alpha_{p}}\tonde*{ \pi_{\mathcal{P}}\tonde*{ A }} >0,
\end{equation*}
is a positive constant. 

Let $ \Omega' \in \mathcal{A}_{\Omega} $ s.t. $ \Omega' \Subset \Omega $.
By the previous step and the monotonicity of $ G^{ \alpha_{ p}}_{\pm}\tonde*{ u ,\cdot } $
(cfr. \cref{res:prop_funz_insieme}) we find
that \eqref{asdf} holds.

\eqref{eq:tesi_Main_BMO} follows by letting $ \Omega' \uparrow \Omega $ in \eqref{asdf}
and recalling \eqref{89234h}.
This concludes the proof.
\end{proof}

We now introduce a stronger version of the concept of $p$-core functional seen in \cref{def:core_function}.
This stronger definition will allow us to extend \cref{thm:Main_BMO}
to functions less regular than Sobolev ones.

\begin{definition}[strong $p$-core functional][def:strong_core_function]
Let $ \alpha_{ p} $ be a $p$-core functional as in \cref{def:core_function}.
Recalling \cref{def:psi_alphap} and \cref{res:prop_psi_alphap},
we say that $ \alpha_{ p} $ is a strong $p$-core functional if 
there exists $ \mathcal{P} $, a linear subspace of $ \mathbb{R}^{m\times n} $, such that 
\cref{eq:dir_sum_dec} holds and,
\begin{equation}\label{alpha_p:poin_typ_bdd_Pspace}
 \tag{$\text{GB}_{ \mathcal{P}}$}
 \begin{split}
 &\exists c_{3} >0
 \text{ s.t. }
 \alpha_{ p}\tonde*{ u , D } 
 \leq
 c_{3} \int_{ D } \abs*{ D_{ \pi_{\mathcal{P}}} u }^{ p }_{2} \thinspace\mathrm{d} x 
 \\
 &\forall u \in L^{p}_{\mathrm{loc}}\tonde*{\mathbb{R}^{n};\mathbb{R}^{m}} \text{ s.t. }\restr{ u }{ D }\in W^{\pi_{\mathcal{P}},p}\tonde*{D;\mathbb{R}^{m}} .
 \end{split}
 \end{equation}
\end{definition}

We observe that \cref{alpha_p:poin_typ_bdd_Pspace} implies \cref{alpha_p:poincare_type_bound}.
Indeed, if $ v=\restr{ u }{ D } \in W^{1,p}\tonde*{D;\mathbb{R}^{m}} $, then,
\begin{equation*}
 \abs*{ D_{ \pi_{\mathcal{P}}} v }_{2} 
 =
 \abs*{ \pi_{\mathcal{P}}\tonde*{ \nabla v }}_{2} \leq\abs*{ \pi_{\mathcal{P}}} \abs*{ \nabla v }_{2} ,
\end{equation*}
where $\abs*{ \pi_{\mathcal{P}}}\coloneqq \sup\limits_{ } \set*{ \abs*{ \pi_{\mathcal{P}}\tonde*{ A }}_{2} : A \in \mathbb{R}^{m\times n} , \abs*{ A }_{2} \leq1}$.
Therefore, \cref{alpha_p:poin_typ_bdd_Pspace} implies that \cref{alpha_p:poincare_type_bound} holds with $ c_{1} =\abs*{ \pi_{\mathcal{P}}}^{ p } c_{3} $ if $\abs*{ \pi_{\mathcal{P}}}>0$. If $\abs*{ \pi_{\mathcal{P}}}=0$ \cref{alpha_p:poincare_type_bound} holds for any $c_1>0$ (cfr. \cref{res:obs_strong_p_core}).

\begin{remark}[][res:obs_strong_p_core]
 Every $p$-core functional such that $ \mathcal{N}_{\psi_{\alpha_{p}}} =\set*{0}$ (cfr. \cref{res:prop_psi_alphap}) is 
 a strong $p$-core functional, since this condition forces $ \mathcal{P} $ to be equal to $ \mathbb{R}^{m\times n} $
 and then $ D_{ \pi_{\mathcal{P}}} = \nabla $.

 We also highlight that if a strong $p$-core functional is such that $ \mathcal{N}_{\psi_{\alpha_{p}}} = \mathbb{R}^{m\times n} $,
 then $ \mathcal{P} =\set*{0}$. Therefore, $ D_{ \pi_{\mathcal{P}}} =0$
 and from \cref{alpha_p:poin_typ_bdd_Pspace} and \cref{alpha_p:change_of_variable}
 we deduce that $ \alpha_{ p} $ is the trivial $p$-core functional (cfr. \cref{ex:trivial}).
\end{remark}

Some consequences of \eqref{alpha_p:poincare_type_bound} are reinforced with \eqref{alpha_p:poin_typ_bdd_Pspace}. 
If $ \mathcal{G} =\set*{ \operatorname{id}_{\mathbb{R}^{n}}}$ and 
 $ \alpha_{ p} = \alpha_{ p , D, \Gamma ^{\mathcal{G}} ,m } $ 
 is a strong $p$-core functional as in \cref{def:strong_core_function}
 subordinated to $ \Gamma ^{\mathcal{G}} $ as in \cref{def:Gamma_G}, then, \eqref{alpha_p:change_of_variable}, \eqref{alpha_p:ext_ind} and \eqref{alpha_p:poin_typ_bdd_Pspace}
 imply,
 \begin{equation}\label{alpha_p:poincare_+_Pspace}
 \tag{$\text{GB}_{ \mathcal{P}}$+}
 \begin{split}
 & \alpha_{ p}\tonde*{ u , D' } 
 \leq
 c_{4} \operatorname{diam}\tonde*{ D' } ^{ p - n } \int_{ D' } \abs*{ D_{ \pi_{\mathcal{P}}}\tonde*{\restr{ u }{ D' }}}^{ p }_{2} \thinspace\mathrm{d} x 
 \text{ }
 \\
 &\forall D' \in \mathcal{G}^{ D}\tonde*{\mathbb{R}^{n}} \forall u \in L^{p}_{\mathrm{loc}}\tonde*{\mathbb{R}^{n};\mathbb{R}^{m}} \text{ s.t. } D_{ \pi_{\mathcal{P}}}\tonde*{\restr{ u }{ D' }} \in L^{p}\tonde*{D';\mathbb{R}^{m\times n}} ,
 \end{split}
 \end{equation}
 where $ c_{4} = c_{3} \operatorname{diam}\tonde*{ D } ^{ n - p }$.

 If $ u \in L^{p}_{\mathrm{loc}}\tonde*{\mathbb{R}^{n};\mathbb{R}^{m}} $ and $ D_{ \pi_{\mathcal{P}}}\tonde*{\restr{ u }{ \Omega }} \in L^{p}\tonde*{\Omega;\mathbb{R}^{m\times n}} $,
 thanks to \eqref{alpha_p:poincare_+_Pspace},
 we obtain that
 \begin{equation*}
 \begin{split}
 G^{\alpha_{p}}_{ \varepsilon }\tonde*{ u,\Omega } 
 &\leq
 \varepsilon ^{ n - p } \sup\limits_{\mathcal{G}_{\varepsilon}} \sum_{ D'\in\mathcal{G}_{\varepsilon}} c_{4} \operatorname{diam}\tonde*{ D' } ^{ p - n } \int_{ D' } \abs*{ D_{ \pi_{\mathcal{P}}}\tonde*{\restr{ u }{ D' }}}^{ p }_{2} \thinspace\mathrm{d} x 
 \\
 &\leq
 c_{3} \int_{ \Omega } \abs*{ D_{ \pi_{\mathcal{P}}}\tonde*{\restr{ u }{ \Omega }}}^{ p }_{2} \thinspace\mathrm{d} x ,
 \end{split}
 \end{equation*}
 for every $ \varepsilon >0$.
 Therefore,
 \begin{equation}\label{eq:utile}
 G^{ \alpha_{ p}}_{\pm}\tonde*{ u,\Omega } \leq c_{3} \int_{ \Omega } \abs*{ D_{ \pi_{\mathcal{P}}}\tonde*{\restr{ u }{ \Omega }}}^{ p }_{2} \thinspace\mathrm{d} x .
 \end{equation}

 \begin{remark}[][res:tip_BMO_bound_Pspace]
 Under the above assumptions, $ g_{+} $ is inner-regular and $ \sigma $-subadditive: that is, 
 \ref{it:supremum} and \ref{it:subadditive} in \cref{res:prop_funz_insieme} hold
 (requiring only the milder assumption that $ D_{ \pi_{\mathcal{P}}}\tonde*{\restr{ u }{ \Omega }}$ belongs to $ L^{p}\tonde*{\Omega;\mathbb{R}^{m\times n}} $
 instead of the $ p $-integrability of the whole gradient).
Indeed, to prove that $ g_{+} $ is inner-regular and $ \sigma $-subadditive we refer again
 to the proof of \cite[Prop. 3.1]{FarFusGuaSch20-MR4062330}
 taking into account \eqref{eq:utile}.
\end{remark}

We now prove the extension of \cref{thm:Main_BMO}.
\begin{proof}[Proof of \cref{thm:Main_BMO_WL}]
 Let $ u \in L^{p}_{\mathrm{loc}}\tonde*{\mathbb{R}^{n};\mathbb{R}^{m}} $ such that $ D_{ \pi_{\mathcal{P}}}\tonde*{\restr{ u }{ \Omega }} \in L^{p}\tonde*{\Omega;\mathbb{R}^{m\times n}} $.
  Thanks to \cref{thm:Main_BMO} and \cref{res:psi_alphap_Pspace} we know that,
 \begin{equation}\label{a48d0_bis}
 \lim\limits_{\varepsilon\to 0^{+}} G^{ \alpha_{ p}}_{ \varepsilon }\tonde*{ v , \Omega' } 
 =
 \int_{ \Omega' } \psi_{\alpha_{p}}\tonde*{ D_{ \pi_{\mathcal{P}}}\tonde*{\restr{ v }{ \Omega' }}} \thinspace\mathrm{d} x ,
 \end{equation}
 holds for every $ \Omega' \in \mathcal{A}_{\Omega} $ and
 for every $ v \in L^{p}_{\mathrm{loc}}\tonde*{\mathbb{R}^{n};\mathbb{R}^{m}} $ s.t. $ \nabla \tonde*{\restr{ v }{ \Omega' }} \in L^{p}\tonde*{\Omega';\mathbb{R}^{m\times n}} $.

 Let $ \Omega' \in \mathcal{A}_{\Omega} $ s.t $ \Omega' \Subset \Omega $ and $ \sigma >0$.
 Then, thanks to \cref{res:DL_density}, there exists $ v^{\Omega'}_{ \sigma } \in C^{\infty}_{c}\tonde*{\mathbb{R}^{n};\mathbb{R}^{m}} $ s.t. 
 $ \norm*{ u - v^{\Omega'}_{ \sigma }}_{W^{\pi_{\mathcal{P}},p}\tonde*{\Omega';\mathbb{R}^{m}}} < \sigma $.

 Now we argue as in \hyperlink{STEP2thm:Main_BMO}{\textbf{Step 2}} in the proof of \cref{thm:Main_BMO},
 making use of \cref{a48d0_bis}, \cref{alpha_p:poincare_+_Pspace} (instead of \cref{alpha_p:poincare_+})
 and \cref{res:prop_funz_insieme} (reinforced with \cref{res:tip_BMO_bound_Pspace})
 to obtain
 \cref{eq:tesi_Main_BMO_WL}.

 Finally, we assume that $\restr{ u }{ \Omega }\in W^{\pi_{\mathcal{P}},p}_{\mathrm{loc}}\tonde*{\Omega;\mathbb{R}^{m}} $ and $ \psi_{\alpha_{p}} \not\equiv0$.
 Arguing as in \hyperlink{STEP4thm:Main_BMO}{\textbf{Step 4}} in the proof of \cref{thm:Main_BMO}
 we deduce that,
 \begin{equation*}\begin{split}
 & \int_{ \Omega } \psi_{\alpha_{p}}\tonde*{ D_{ \pi_{\mathcal{P}}}\tonde*{\restr{ u }{ \Omega }}} \thinspace\mathrm{d} x 
 \\
 &=
 \int_{ \Omega \setminus\set*{ \abs*{ D_{ \pi_{\mathcal{P}}}\tonde*{\restr{ u }{ \Omega }}}_{2} =0}} 
 \abs*{ D_{ \pi_{\mathcal{P}}}\tonde*{\restr{ u }{ \Omega }}}^{ p }_{2} \psi_{\alpha_{p}}\tonde*{ \pi_{\mathcal{P}}\tonde*{\frac{ D_{ \pi_{\mathcal{P}}}\tonde*{\restr{ u }{ \Omega }}}{ \abs*{ D_{ \pi_{\mathcal{P}}}\tonde*{\restr{ u }{ \Omega }}}_{2}}}} 
 \thinspace\mathrm{d} x 
 \\
 &\geq
 C \int_{ \Omega } 
 \abs*{ D_{ \pi_{\mathcal{P}}}\tonde*{\restr{ u }{ \Omega }}}^{ p }_{2} 
 \thinspace\mathrm{d} x ,
 \end{split}
 \end{equation*}
 where,
 \begin{equation*}
 C \coloneqq\min_{ A \in \mathcal{P} , \abs*{ A }_{2} =1} \psi_{\alpha_{p}}\tonde*{ \pi_{\mathcal{P}}\tonde*{ A }} >0,
 \end{equation*}
 is a positive constant. At this point,
 the same argument in \hyperlink{STEP4thm:Main_BMO}{\textbf{Step 4}} in the proof of \cref{thm:Main_BMO} 
 allows us to conclude that \eqref{eq:tesi_Main_BMO_WL} holds.
\end{proof} \section{Non-distributional characterization of Sobolev-type spaces}\label{sec:characterizations}

As a consequence of \cref{thm:Main_BMO} we obtain the following.

\begin{corollary}[][res:cor_Main_BMO]
 Let $ \Omega \subseteq \mathbb{R}^{n} $ be an open set, $ m \geq1$ a natural number, $ p \in \left[1,\infty\right) $,
 $ \mathcal{G} $ a subgroup of $ \mathrm{SO}\tonde*{n} $ and $ D \subseteq \mathbb{R}^{n} $ be a bounded open set.
 Let $ \alpha_{ p} = \alpha_{ p , D, \Gamma ^{\mathcal{G}} ,m } $ 
 be a $p$-core functional as in \cref{def:core_function}
 subordinated to $ \Gamma ^{\mathcal{G}} $ as in \cref{def:Gamma_G}.
 Assume also that
 for every $ D' \in \mathcal{G}^{ D}\tonde*{\mathbb{R}^{n}} $ the functional
 $\restr{ \alpha_{ p}\tonde*{\cdot, D' }}{ L^{p}\tonde*{\mathbb{R}^{n};\mathbb{R}^{m}}}: L^{p}\tonde*{\mathbb{R}^{n};\mathbb{R}^{m}} \to \left[0,+\infty\right) $
 is lower semicontinuous.

 Let $ \mathcal{P} $ be a linear subspace of $ \mathbb{R}^{m\times n} $ as in \eqref{eq:dir_sum_dec}
 and $ \pi_{\mathcal{P}} : \mathbb{R}^{m\times n} \to \mathcal{P} $ as in \cref{res:psi_alphap_Pspace}.
 Let $ \graffe*{ \rho _{ \sigma }}_{ \sigma >0} \subset C^{\infty}_{c}\tonde*{\mathbb{R}^{n}} $ be a family of standard mollifiers
 and $ \tilde{\Omega} \in \mathcal{A}_{\Omega} $ s.t. $ \tilde{\Omega} \Subset \Omega $.
 For every $ u \in L^{p}_{\mathrm{loc}}\tonde*{\mathbb{R}^{n};\mathbb{R}^{m}} $ and
 for every $0< \sigma < \operatorname{dist}\tonde*{ \tilde{\Omega} , \partial \Omega } $,
 \begin{equation}\label{tesi:cor}
 C \int_{ \tilde{\Omega}} \abs*{ \pi_{\mathcal{P}}\tonde*{ \nabla u_{ \sigma }}}^{ p }_{2} \thinspace\mathrm{d} x 
 \leq
 \int_{ \tilde{\Omega}} \psi_{\alpha_{p}}\tonde*{ \nabla u_{ \sigma }} \thinspace\mathrm{d} x 
 =
 G^{ \alpha_{ p}}\tonde*{ u_{ \sigma } , \tilde{\Omega}} \leq G^{ \alpha_{ p}}_{ -}\tonde*{ u,\Omega } ,
 \end{equation}
 where $ u_{ \sigma } \coloneqq \rho _{ \sigma } \ast u $ for every $ \sigma >0$
 and $ C >0$ is a constant depending on $ \psi_{\alpha_{p}} $.

\end{corollary}
\begin{proof}
 Let $0< \sigma < \operatorname{dist}\tonde*{ \tilde{\Omega} , \partial \Omega } $.
 Without loss of generality, having \eqref{tipo_BMO:ext_ind} in mind, we may assume that $\restr{ u }{ \mathbb{R}^{n} \setminus \tilde{\Omega}}=0$
 so that $ u \in L^{p}\tonde*{\mathbb{R}^{n};\mathbb{R}^{m}} $.
 Let $ \varepsilon >0$ and
 $ \mathcal{G}_{\varepsilon} $ be a family made of pairwise disjoint elements of $ \mathcal{G}^{ D}_{\varepsilon}\tonde*{\tilde{\Omega}} $ and
 we observe that $ \sigma >0$ is so small that $ D' - \sigma y \subset \Omega $ for every $ y \in B_{1}\tonde*{0} $.
 Applying \cref{res:jensen_alpha_p}
 (with $ f\tonde*{ v } \coloneqq \alpha_{ p}\tonde*{ v , D' } $ for all $ v \in L^{p}\tonde*{\mathbb{R}^{n};\mathbb{R}^{m}} $) and \eqref{alpha_p:inner_trans_inv}, we get,
 \begin{align*}
 & \varepsilon ^{ n - p } \sum_{ D'\in\mathcal{G}_{\varepsilon}} \alpha_{ p}\tonde*{ u_{ \sigma }, D' } 
 \leq
 \varepsilon ^{ n - p } \sum_{ D'\in\mathcal{G}_{\varepsilon}} \int_{ B_{1}\tonde*{0}} \alpha_{ p}\tonde*{ u\tonde*{.- \sigma y }, D' } \rho \tonde*{ y } \thinspace\mathrm{d} y 
 \\
 &=
 \varepsilon ^{ n - p } \sum_{ D'\in\mathcal{G}_{\varepsilon}} \int_{ B_{1}\tonde*{0}} \alpha_{ p}\tonde*{ u , D' - \sigma y } \rho \tonde*{ y } \thinspace\mathrm{d} y 
 \leq
 G^{\alpha_{p}}_{ \varepsilon }\tonde*{ u,\Omega } .
 \end{align*}
 Taking the supremum with respect to $ \mathcal{G}_{\varepsilon} $ and letting $ \varepsilon \to 0^{+} $, we obtain
 \begin{equation*}
 G^{ \alpha_{ p}}_{ -}\tonde*{ u_{ \sigma } , \tilde{\Omega}} \leq G^{ \alpha_{ p}}_{ -}\tonde*{ u,\Omega } .
 \end{equation*}
 \eqref{tesi:cor} follows by applying \cref{thm:Main_BMO} to $ u_{ \sigma } $
 and invoking \cref{res:prop_psi_alphap}
 and \cref{res:psi_alphap_Pspace}.
\end{proof}

The following result is a first step in characterizing Sobolev-type spaces.
\begin{corollary}[][res:cor_Main_BMO2]
 Let $ \Omega \subseteq \mathbb{R}^{n} $ be an open set, $ m \geq1$ a natural number, $ p \in \left(1,\infty\right) $,
 $ \mathcal{G} $ a subgroup of $ \mathrm{SO}\tonde*{n} $ and $ D \subseteq \mathbb{R}^{n} $ be a bounded open set.
 Let $ \alpha_{ p} = \alpha_{ p , D, \Gamma ^{\mathcal{G}} ,m } $ 
 be a $p$-core functional as in \cref{def:core_function}
 subordinated to $ \Gamma ^{\mathcal{G}} $ as in \cref{def:Gamma_G}.
 Assume also that
 for every $ D' \in \mathcal{G}^{ D}\tonde*{\mathbb{R}^{n}} $ the functional
 $\restr{ \alpha_{ p}\tonde*{\cdot, D' }}{ L^{p}\tonde*{\mathbb{R}^{n};\mathbb{R}^{m}}}: L^{p}\tonde*{\mathbb{R}^{n};\mathbb{R}^{m}} \to \left[0,+\infty\right) $
 is lower semicontinuous.

 Let $ \mathcal{P} $ be a linear subspace of $ \mathbb{R}^{m\times n} $ as in \eqref{eq:dir_sum_dec}
 and $ \pi_{\mathcal{P}} : \mathbb{R}^{m\times n} \to \mathcal{P} $ as in \cref{res:psi_alphap_Pspace}.
 Let $ u \in L^{p}_{\mathrm{loc}}\tonde*{\mathbb{R}^{n};\mathbb{R}^{m}} $. Then,
 \begin{equation}\label{abcjh219h}
 G^{ \alpha_{ p}}_{ -}\tonde*{ u,\Omega } <+\infty\implies D_{ \pi_{\mathcal{P}}}\tonde*{\restr{ u }{ \Omega }} \in L^{p}\tonde*{\Omega;\mathbb{R}^{m\times n}} .
 \end{equation}
 Moreover, if $ \psi_{\alpha_{p}} >0$ in $ \mathbb{R}^{m\times n} \setminus\set*{0}$,
 then $ \nabla \restr{ u }{ \Omega } \in L^{p}\tonde*{\Omega;\mathbb{R}^{m\times n}} $ if and only if $ G^{ \alpha_{ p}}_{ -}\tonde*{ u,\Omega } <+\infty$.
\end{corollary}
\begin{remark}[][res:optimality]
 Under the assumptions of \cref{res:cor_Main_BMO2},
 if $ u \in L^{p}_{\mathrm{loc}}\tonde*{\mathbb{R}^{n};\mathbb{R}^{m}} $ is such that
 $ D_{ \pi_{\mathcal{P}}}\tonde*{\restr{ u }{ \Omega }} \in L^{p}\tonde*{\Omega;\mathbb{R}^{m\times n}} $,
 in general we cannot conclude that $ G^{ \alpha_{ p}}_{ -}\tonde*{ u,\Omega } <+\infty$ (cfr. \cref{ex:counterex_psi_null}).
\end{remark}
\begin{proof}[Proof of \cref{res:cor_Main_BMO2}]
 Let $ p \in \left(1,\infty\right) $, $ u \in L^{p}_{\mathrm{loc}}\tonde*{\mathbb{R}^{n};\mathbb{R}^{m}} $ and $ u_{ \sigma } \coloneqq \rho _{ \sigma } \ast u $
 for every $ \sigma >0$.
 If $ G^{ \alpha_{ p}}_{ -}\tonde*{ u,\Omega } <+\infty$, using \eqref{tesi:cor}
 we obtain that 
 \begin{equation}\label{s43f}
 \limsup\limits_{ \sigma \to 0^{+}} \int_{ \tilde{\Omega}} \abs*{ \pi_{\mathcal{P}}\tonde*{ \nabla u_{ \sigma }}}^{ p }_{2} \thinspace\mathrm{d} x 
 \leq
 \frac{1}{ C }
 G^{ \alpha_{ p}}_{ -}\tonde*{ u,\Omega } ,
 \end{equation} 
 for every $ \tilde{\Omega} \in \mathcal{A}_{\Omega} $ s.t. $ \tilde{\Omega} \Subset \Omega $.
 Since $ u_{ \sigma } \to u $ in $ L^{p}_{\mathrm{loc}}\tonde*{\mathbb{R}^{n};\mathbb{R}^{m}} $ as $ \sigma \to 0^{+} $,
 we deduce that 
 $ \pi_{\mathcal{P}}\tonde*{ \nabla \restr{ u_{ \sigma }}{ \Omega }} \to D_{ \pi_{\mathcal{P}}}\tonde*{\restr{ u }{ \Omega }} $
 as $ \sigma \to 0^{+} $ in the sense of distributions.
 Then, from \eqref{s43f} (since it provides a uniform bound with respect to $ \tilde{\Omega} $)
 we get that the distribution 
 $ D_{ \pi_{\mathcal{P}}}\tonde*{\restr{ u }{ \Omega }} $
 is a linear and continuous functional
 on $ L^{ p^{\prime}}\tonde*{\Omega;\mathbb{R}^{m}} $.
 This proves that $ D_{ \pi_{\mathcal{P}}}\tonde*{\restr{ u }{ \Omega }} \in L^{p}\tonde*{\Omega;\mathbb{R}^{m\times n}} $.

 Finally, if $ \psi_{\alpha_{p}} >0$ in $ \mathbb{R}^{m\times n} \setminus\set*{0}$,
 then
 $ \mathcal{P} = \mathbb{R}^{m\times n} $ and $ \pi_{\mathcal{P}} = \operatorname{id}_{\mathbb{R}^{m\times n}} $.
 Invoking what proved above and \cref{thm:Main_BMO} we deduce that 
 $ \nabla \restr{ u }{ \Omega } = \pi_{\mathcal{P}}\tonde*{ \nabla \restr{ u }{ \Omega }} \in L^{p}\tonde*{\Omega;\mathbb{R}^{m\times n}} $
 if and only if $ G^{ \alpha_{ p}}_{ -}\tonde*{ u,\Omega } <+\infty$.
\end{proof}

Now for strong $p$-core functionals as a consequence of \cref{thm:Main_BMO_WL} we obtain the following characterization.

\begin{corollary}[][res:cor_Main_BMO2_WL]
 Let $ \Omega \subseteq \mathbb{R}^{n} $ be an open set, $ m \geq1$ a natural number, $ p \in \left(1,\infty\right) $,
 $ \mathcal{G} =\set*{ \operatorname{id}_{\mathbb{R}^{n}}}$ and $ D \subseteq \mathbb{R}^{n} $ be a bounded open set.
 Let $ \alpha_{ p} = \alpha_{ p , D, \Gamma ^{\mathcal{G}} ,m } $ 
 be a strong $p$-core functional as in \cref{def:strong_core_function}
 subordinated to $ \Gamma ^{\mathcal{G}} $ as in \cref{def:Gamma_G}.
 Assume also that
 for every $ D' \in \mathcal{G}^{ D}\tonde*{\mathbb{R}^{n}} $ the functional
 $\restr{ \alpha_{ p}\tonde*{\cdot, D' }}{ L^{p}\tonde*{\mathbb{R}^{n};\mathbb{R}^{m}}}: L^{p}\tonde*{\mathbb{R}^{n};\mathbb{R}^{m}} \to \left[0,+\infty\right) $
 is lower semicontinuous.

 Let $ \mathcal{P} $ be a linear subspace of $ \mathbb{R}^{m\times n} $ as in \eqref{eq:dir_sum_dec}
 and $ \pi_{\mathcal{P}} : \mathbb{R}^{m\times n} \to \mathcal{P} $ as in \cref{res:psi_alphap_Pspace}.
 Let $ u \in L^{p}_{\mathrm{loc}}\tonde*{\mathbb{R}^{n};\mathbb{R}^{m}} $. Then,
 \begin{equation}\label{abcjh219h_new}
 G^{ \alpha_{ p}}_{ -}\tonde*{ u,\Omega } <+\infty\iff D_{ \pi_{\mathcal{P}}}\tonde*{\restr{ u }{ \Omega }} \in L^{p}\tonde*{\Omega;\mathbb{R}^{m\times n}} .
 \end{equation}
 Moreover, if one of the equivalent conditions in \cref{abcjh219h_new} is in force,
 then \eqref{eq:tesi_Main_BMO_WL} holds.
\end{corollary}
\begin{proof}
 We first assume that $ D_{ \pi_{\mathcal{P}}}\tonde*{\restr{ u }{ \Omega }} \in L^{p}\tonde*{\Omega;\mathbb{R}^{m\times n}} $. Then, by 
 \cref{eq:utile} in \cref{res:tip_BMO_bound_Pspace} we obtain that $ G^{ \alpha_{ p}}_{ -}\tonde*{ u,\Omega } <+\infty$.

 The reversed implication follows directly from \cref{res:cor_Main_BMO2}.

 Finally, if one of the equivalent conditions in \cref{abcjh219h_new} is in force, \eqref{eq:tesi_Main_BMO_WL} follows
 by \cref{thm:Main_BMO_WL}.
\end{proof}

The following result is a constancy theorem.
\begin{corollary}[][res:constancy]
 Let $ \Omega \subseteq \mathbb{R}^{n} $ be an open connected set, $ m \geq1$ a natural number, $ p \in \left(1,\infty\right) $,
 $ \mathcal{G} $ a subgroup of $ \mathrm{SO}\tonde*{n} $ and $ D \subseteq \mathbb{R}^{n} $ be a bounded open set.
 Let $ \alpha_{ p} = \alpha_{ p , D, \Gamma ^{\mathcal{G}} ,m } $ 
 be a $p$-core functional as in \cref{def:core_function}
 subordinated to $ \Gamma ^{\mathcal{G}} $ as in \cref{def:Gamma_G}.
 Let $ u \in L^{p}_{\mathrm{loc}}\tonde*{\mathbb{R}^{n};\mathbb{R}^{m}} $.
 Assume also that
 for every $ D' \in \mathcal{G}^{ D}\tonde*{\mathbb{R}^{n}} $ the functional
 $\restr{ \alpha_{ p}\tonde*{\cdot, D' }}{ L^{p}\tonde*{\mathbb{R}^{n};\mathbb{R}^{m}}}: L^{p}\tonde*{\mathbb{R}^{n};\mathbb{R}^{m}} \to \left[0,+\infty\right) $
 is lower semicontinuous and that $ G^{ \alpha_{ p}}_{ -}\tonde*{ u,\Omega } =0$.

 Then, $ D_{ \pi_{\mathcal{P}}}\tonde*{\restr{ u }{ \Omega }} =0$.
 As a consequence, if $ \psi_{\alpha_{p}} >0$ in $ \mathbb{R}^{m\times n} $, then
 $ \nabla \restr{ u }{ \Omega } =0$ and $ u $ is constant $ \operatorname{a.e.} $
 on $ \Omega $.

 Moreover, if we also assume that $ \mathcal{G} =\set*{ \operatorname{id}_{\mathbb{R}^{n}}}$ and $\alpha_p$
 is a strong $p$-core functional,
 then $ \alpha_{ p}\tonde*{ u , D' } =0$ for every
 $ D' \in \mathcal{G}^{ D}\tonde*{\Omega} $.
 
\end{corollary}
\begin{proof}
We let $ u_{ \sigma } \coloneqq \rho _{ \sigma } \ast u $ for every $ \sigma >0$.
Since $ u_{ \sigma } \to u $ in $ L^{p}_{\mathrm{loc}}\tonde*{\mathbb{R}^{n};\mathbb{R}^{m}} $ as $ \sigma \to 0^{+} $,
as in the proof of \cref{res:cor_Main_BMO2},
we deduce that 
$ \pi_{\mathcal{P}}\tonde*{ \nabla \restr{ u_{ \sigma }}{ \Omega }} \to D_{ \pi_{\mathcal{P}}}\tonde*{\restr{ u }{ \Omega }} \in L^{p}\tonde*{\Omega;\mathbb{R}^{m\times n}} $.
Since $ G^{ \alpha_{ p}}_{ -}\tonde*{ u,\Omega } =0$, by \cref{res:cor_Main_BMO} we obtain that 
$ \pi_{\mathcal{P}}\tonde*{ \nabla \restr{ u_{ \sigma }}{ \Omega }} $ vanishes on every 
$ \tilde{\Omega} \in \mathcal{A}_{\Omega} $ s.t. $ \tilde{\Omega} \Subset \Omega $ for every $0< \sigma < \operatorname{dist}\tonde*{ \tilde{\Omega} , \partial \Omega } $. 
Therefore, $ D_{ \pi_{\mathcal{P}}}\tonde*{\restr{ u }{ \Omega }} =0$.
Moreover, if $ \psi_{\alpha_{p}} >0$ in $ \mathbb{R}^{m\times n} $,
then $ \nabla \restr{ u }{ \Omega } = D_{ \pi_{\mathcal{P}}}\tonde*{\restr{ u }{ \Omega }} =0$.

Finally,
if we also assume that $ \mathcal{G} =\set*{ \operatorname{id}_{\mathbb{R}^{n}}}$ and $\alpha_p$
is a strong $p$-core functional,
then, by \cref{alpha_p:poincare_+_Pspace} we deduce that $ \alpha_{ p}\tonde*{ u , D' } =0$ for every $ D' \in \mathcal{G}^{ D}\tonde*{\Omega} $.

\end{proof}

\section{Examples and applications}\label{sec:applications}

In this section we discuss some examples of $p$-core functionals 
which can be used to apply the results in \cref{sec:genaralBMO} and \cref{sec:characterizations}.
We start with a proposition which will be used to provide more explicit representation of the function $\psi_{\alpha_p}$ for some $p$-core functionals $\alpha_p$.

\begin{proposition}[][res:formula_psi]
 Let $ m \geq1$ be a natural number, $ p \in \left[1,\infty\right) $,
 $ \mathcal{G} $ a subgroup of $ \mathrm{SO}\tonde*{n} $ and $ D \subseteq \mathbb{R}^{n} $ be a bounded open set.
 Let $ \alpha_{ p} = \alpha_{ p , D, \Gamma ^{\mathcal{G}} ,m } $ 
 be a $p$-core functional as in \cref{def:core_function}
 subordinated to $ \Gamma ^{\mathcal{G}} $ as in \cref{def:Gamma_G}.

 Then, for every $ A \in \mathbb{R}^{m\times n} $, $ \psi_{\alpha_{p}} $ is constant on $ A \mathcal{G} \coloneqq\set*{ A O : O \in \mathcal{G}}$ and
 \begin{equation}\label{eq:char_psip_Gsub}
 \psi_{\alpha_{p}}\tonde*{ A } \leq\frac{1}{ \abs*{ D }} \max\limits_{ B \in A \mathcal{G}} \alpha_{ p}\tonde*{ l^{ B}, D } ,
 \end{equation}
 where $ \psi_{\alpha_{p}} $ is defined in \cref{def:psi_alphap}.

 Moreover, assume that for every $ \mathcal{R} \in \mathcal{G} $ there exists a family $ \graffe*{ \mathcal{G}_{\varepsilon}}_{0< \varepsilon <1} $
 such that:
 \begin{enumerate}
 \item\label{qwe1} for every $ \varepsilon \in \left(0,1\right) $, $ \mathcal{G}_{\varepsilon} $ is a family made of pairwise disjoint elements of $ \mathcal{G}^{ D}_{\varepsilon}\tonde*{ Q} $ of the form:
 \begin{equation*}
 D' = \varepsilon \mathcal{R} D + h_{D'} ,
 \end{equation*}
 where $ h_{D'} \in \mathbb{R}^{n} $ depends on $ D' $;
 \item\label{qwe2}the following \emph{no-gaps condition} holds,
 \begin{equation}\label{eq:nostra_nogaps}
 \lim\limits_{\varepsilon\to 0^{+}} \abs*{ \mathcal{G}_{\varepsilon}^{\cup}} = \abs*{ Q } =1,
 \end{equation}
 where, for every $ \varepsilon \in \left(0,1\right) $, $ \mathcal{G}_{\varepsilon}^{\cup} \subset Q $ is the union of all the sets belonging to $ \mathcal{G}_{\varepsilon} $.
 \end{enumerate}
 Then, the equality holds in \eqref{eq:char_psip_Gsub}.
\end{proposition}
\begin{remark}
 We first observe that the condition \eqref{eq:nostra_nogaps}
 is equivalent to,
 \begin{equation}\label{eq:nostra_nogaps_vecchia}
 \lim\limits_{\varepsilon\to 0^{+}} \abs*{ Q \setminus \mathcal{G}_{\varepsilon}^{\cup}} =0,
 \end{equation}
 since $ \mathcal{G}_{\varepsilon}^{\cup} \subset Q $ for every $ \varepsilon \in \left(0,1\right) $.
 \eqref{eq:nostra_nogaps} is also equivalent to,
 \begin{equation*}\label{eq:nostra_nogaps2}
 \lim\limits_{\varepsilon\to 0^{+}} \varepsilon ^ n \cardinality[ \mathcal{G}_{\varepsilon}]=\frac{1}{ \abs*{ D }},
 \end{equation*}
 being 
 $ \abs*{ \mathcal{G}_{\varepsilon}^{\cup}} = \abs*{ D } \varepsilon ^ n \cardinality[ \mathcal{G}_{\varepsilon}]$ for every $ \varepsilon \in \left(0,1\right) $.

 We also observe that \eqref{eq:nostra_nogaps} and \eqref{eq:nostra_nogaps_vecchia}
 are related to \cite[Def. 2]{DiFFio20-MR4093787}. Indeed, by Lebesgue's dominated convergence theorem,
 the condition
 \begin{equation}\label{eq:nogaps_dif_fio}
 \lim\limits_{\varepsilon\to 0^{+}} \chi_{ \mathcal{G}_{\varepsilon}^{\cup}}=\chi_{ Q }\text{ } \operatorname{a.e.} ,
 \end{equation} 
 implies \eqref{eq:nostra_nogaps_vecchia}. 
 Moreover, we also observe that the weaker condition
 \begin{equation*}\label{eq:nogaps_dif_fio_weak}
 \liminf\limits_{\varepsilon\to 0^{+}} \chi_{ \mathcal{G}_{\varepsilon}^{\cup}}\geq\chi_{ Q }\text{ } \operatorname{a.e.} ,
 \end{equation*}
 implies \eqref{eq:nogaps_dif_fio},
 since $ \mathcal{G}_{\varepsilon}^{\cup} \subset Q $ for every $ \varepsilon \in \left(0,1\right) $.
\end{remark}

\begin{remark}[][res:tessellation_useful]
 Let $ D \subset \mathbb{R}^{n} $ be a bounded open set, $ \varepsilon_{0} >0$
 and assume that $ D $ is a reference cell 
 for an $ \varepsilon_{0} $-tessellation of $ \mathbb{R}^{n} $ (cfr. \cite[Def. 1]{DiFFio20-MR4093787}).
 Then, $ \mathcal{R}\tonde*{ D } $ is also a reference cell 
 for an $ \varepsilon_{0} $-tessellation of $ \mathbb{R}^{n} $ for every $ \mathcal{R} \in \mathrm{SO}\tonde*{n} $.
 Therefore,
 there exists a family $ \graffe*{ \mathcal{G}_{\varepsilon}}_{0< \varepsilon <1} $
 satisfying \eqref{qwe1} and \eqref{qwe2} in \cref{res:formula_psi}
 (cfr. \cite[Def. 2 and Prop. 1]{DiFFio20-MR4093787}).
\end{remark}

\begin{proof}[Proof of \cref{res:formula_psi}]

Let $ A \in \mathbb{R}^{m\times n} $.
Recalling \cref{res:prop_psi_alphap},
 \begin{equation}\label{32oiuy}
 \psi_{\alpha_{p}}\tonde*{ A } 
 =
 G^{ \alpha_{ p}}\tonde*{ l^{ A} , Q } 
 =
 \lim\limits_{\varepsilon\to 0^{+}} G^{\alpha_{p}}_{ \varepsilon }\tonde*{ l^{ A} , Q } 
 =
 \lim\limits_{\varepsilon\to 0^{+}} \varepsilon ^{ n - p } \sup\limits_{\mathcal{G}_{\varepsilon}} \sum_{ D'\in\mathcal{G}_{\varepsilon}} \alpha_{ p}\tonde*{ l^{ A}, D' } ,
 \end{equation}
where for every $ \varepsilon >0$, each $ \mathcal{G}_{\varepsilon} $ is a family made of pairwise disjoint elements of $ \mathcal{G}^{ D}_{\varepsilon}\tonde*{ Q} $ 
and the supremum is taken with respect to all of these families $ \mathcal{G}_{\varepsilon} $.
We recall that the cardinality of every family $ \mathcal{G}_{\varepsilon} $ is bounded from above by $\tonde*{ \varepsilon ^ n \abs*{ D }}^{-1}$.
Let $ A \in \mathbb{R}^{m\times n} $ and $ B , C \in A \mathcal{G} $.
Then, $ B = C O $ with $ O \in \mathcal{G} $.
By \eqref{alpha_p:change_of_variable} we obtain,
\begin{equation*}
 \alpha_{ p}\tonde*{ l^{ B}, D' } = \alpha_{ p}\tonde*{ l^{ C}, O D' } ,
\end{equation*}
for every $ D' \in \mathcal{G}^{ D}_{\varepsilon}\tonde*{\mathbb{R}^{n}} $.
In particular,
\begin{equation*}
 \psi_{\alpha_{p}}\tonde*{ B } 
 =
 \lim\limits_{\varepsilon\to 0^{+}} \varepsilon ^{ n - p } \sup\limits_{\mathcal{G}_{\varepsilon}} \sum_{ D'\in\mathcal{G}_{\varepsilon}} \alpha_{ p}\tonde*{ l^{ C}, O D' } 
 \leq
 G^{ \alpha_{ p}}\tonde*{ l^{ C} , O Q } = \psi_{\alpha_{p}}\tonde*{ C } ,
\end{equation*}
where we used \eqref{eq:psi_ind_cubo} in \cref{res:prop_psi_alphap}, with $ \tilde{Q} = O Q $, and the fact that
for every $ \varepsilon >0$,
if $ \mathcal{G}_{\varepsilon} $ is a family made of pairwise disjoint elements of $ \mathcal{G}^{ D}_{\varepsilon}\tonde*{ Q} $,
then the family $\set*{ O D' : D' \in \mathcal{G}}$
is made of pairwise disjoint elements
of $ \mathcal{G}^{ D }_{ \varepsilon }\tonde*{ O Q } $.
Interchanging the role of $ B $ and $ C $
we obtain that $ \psi_{\alpha_{p}} $ is constant on $ A \mathcal{G} $.

We now prove \eqref{eq:char_psip_Gsub}.
Let $ \varepsilon >0$ and $ D' \in \mathcal{G}^{ D}_{\varepsilon}\tonde*{\mathbb{R}^{n}} $.
Then, $ D' = \varepsilon \mathcal{R}_{D'} D + h_{D'} $ for some $ \mathcal{R}_{D'} \in \mathcal{G} $ and $ h_{D'} \in \mathbb{R}^{n} $.
By \eqref{eq:loclip_key2}, we get that
\begin{equation}\label{f9843wh}
 \alpha_{ p}\tonde*{ l^{ A}, D' } = \varepsilon ^ p \alpha_{ p}\tonde*{ l^{ A \mathcal{R}_{D'}}, D } .
\end{equation}
Then,
\begin{equation*}
 \varepsilon ^{ n - p } \sup\limits_{\mathcal{G}_{\varepsilon}} \sum_{ D'\in\mathcal{G}_{\varepsilon}} \alpha_{ p}\tonde*{ l^{ A}, D' } 
 =
 \varepsilon ^{ n } \sup\limits_{\mathcal{G}_{\varepsilon}} \sum_{ D'\in\mathcal{G}_{\varepsilon}} \alpha_{ p}\tonde*{ l^{ A \mathcal{R}_{D'}}, D } 
 \leq
 \varepsilon ^{ n }
 \sup\limits_{\mathcal{G}_{\varepsilon}}\cardinality[ \mathcal{G}_{\varepsilon}] 
 \max\limits_{ B \in A \mathcal{G}} \alpha_{ p}\tonde*{ l^{ B}, D } .
\end{equation*}
Therefore, from \eqref{32oiuy} we get,
\begin{equation}\label{dafd4}
 \psi_{\alpha_{p}}\tonde*{ A } 
 =
 \lim\limits_{\varepsilon\to 0^{+}} \varepsilon ^{ n } \sup\limits_{\mathcal{G}_{\varepsilon}} \sum_{ D'\in\mathcal{G}_{\varepsilon}} \alpha_{ p}\tonde*{ l^{ A \mathcal{R}_{D'}}, D } 
 \leq
 \frac{1}{ \abs*{ D }}
 \max\limits_{ B \in A \mathcal{G}} \alpha_{ p}\tonde*{ l^{ B}, D } .
\end{equation}
Then, \eqref{eq:char_psip_Gsub} holds.

Finally, let $ \mathcal{R} \in \mathcal{G} $ and consider a family $ \graffe*{ \mathcal{G}_{\varepsilon}}_{0< \varepsilon <1} $ satisfying \eqref{qwe1} and \eqref{qwe2}.
Then, $ \lim\limits_{\varepsilon\to 0^{+}} \varepsilon ^{ n }\cardinality[ \mathcal{G}_{\varepsilon}]=\frac{1}{ \abs*{ D }}$,
and consequently from \eqref{32oiuy} and \eqref{f9843wh} we obtain,
\begin{equation*}
 \psi_{\alpha_{p}}\tonde*{ A } 
 \geq
 \lim\limits_{\varepsilon\to 0^{+}} \varepsilon ^{ n } \sum_{ D'\in\mathcal{G}_{\varepsilon}} \alpha_{ p}\tonde*{ l^{ A \mathcal{R}}, D } 
 =
 \lim\limits_{\varepsilon\to 0^{+}} \varepsilon ^{ n }\cardinality[ \mathcal{G}_{\varepsilon}] \alpha_{ p}\tonde*{ l^{ A \mathcal{R}}, D } 
 =
 \frac{1}{ \abs*{ D }}
 \alpha_{ p}\tonde*{ l^{ A \mathcal{R}}, D } .
\end{equation*}
Taking the maximum with respect to $ \mathcal{R} \in \mathcal{G} $, we obtain,
\begin{equation}\label{j980}
 \psi_{\alpha_{p}}\tonde*{ A } 
 \geq
 \frac{1}{ \abs*{ D }}
 \max\limits_{ B \in A \mathcal{G}} \alpha_{ p}\tonde*{ l^{ B}, D } .
\end{equation}
Combining \eqref{j980} and \eqref{dafd4} we conclude the proof.
\end{proof}

\subsection{Characterization of \texorpdfstring{$W^{1,p}(\Omega;\mathbb{R}^m)$}{W1p(Omega;Rm)}}
\phantom{}

In this section we will consider examples of $p$-core functionals used to characterize the space $W^{1,p}(\Omega;\mathbb{R}^m)$ of (vector-valued) Sobolev functions when $p\in(1,\infty)$ and to approximate the following $L^p$ norm of the gradient of functions in $W^{1,p}(\Omega;\mathbb{R}^m)$ for $p\in[1,\infty)$:
\begin{equation*}
    \int_\Omega |\nabla u|^p_{\alpha_p}\mathrm{d} x,
\end{equation*}
where $|.|_{\alpha_p}\coloneqq\psi_{\alpha_p}^{1/p}(.)$
is a (possibly anisotropic) norm in $\mathbb{R}^{m\times n}$ (cfr. \cref{res:W1p_char}).

In this section $ m \geq1$ will always denote a natural number and $ D \subseteq \mathbb{R}^{n} $ a bounded connected open set with Lipschitz boundary since we will need the classical \emph{Poincaré} inequality.

This first example is a generalization of the 
 BMO-type seminorms considered in the literature for the
 scalar case (see \cite{FusMosSbo16-MR3419767,FusMosSbo18-MR3816417,FarFusGuaSch20-MR4062330,FarGuaSch20-MR4109099,DiFFio20-MR4093787}).
\begin{example}[][ex:fun_meno_media]
Let $ p \in \left[1,\infty\right) $.
 We define,
 \begin{equation*}
 \alpha_{ p}\tonde*{ u , D' } 
 \coloneqq
 \strokedint_{ D'} \abs*{ u\tonde*{ x } - u_{ D'}}^{ p }_{2} \thinspace\mathrm{d} x .
 \end{equation*}
 Then, $ \alpha_{ p} : L^{p}_{\mathrm{loc}}\tonde*{\mathbb{R}^{n};\mathbb{R}^{m}} \times \mathcal{E}^{ D}_{ \Gamma }\tonde*{\mathbb{R}^{n}} \to \left[0,+\infty\right) $ is a $p$-core functional for any $ \Gamma $ subgroup of $ \mathrm{Aff}\tonde*{\mathbb{R}^{n}} $.
 Moreover, by Fatou's lemma, $\restr{ \alpha_{ p}\tonde*{\cdot, D' }}{ L^{p}\tonde*{\mathbb{R}^{n};\mathbb{R}^{m}}}: L^{p}\tonde*{\mathbb{R}^{n};\mathbb{R}^{m}} \to \left[0,+\infty\right) $
 is lower semicontinuous for every $ D' \in \mathcal{E}^{ D}_{ \Gamma }\tonde*{\mathbb{R}^{n}} $.
\end{example}

\begin{proposition}[][res:fun_meno_media] 
 Let $ p \in \left[1,\infty\right) $ and
 $ \mathcal{G} $ be a subgroup of $ \mathrm{SO}\tonde*{n} $.
 Let $ \alpha_{ p} = \alpha_{ p , D, \Gamma ^{\mathcal{G}} ,m } $ 
 be the $p$-core functional in \cref{ex:fun_meno_media}
 subordinated to $ \Gamma ^{\mathcal{G}} $ as in \cref{def:Gamma_G}.
 
 Then, $ \alpha_{ p} $ is a strong $p$-core functional as in \cref{def:strong_core_function} and,
 \begin{equation}\label{9832h9hh}
 \mathcal{N}_{\psi_{\alpha_{p}}} =\set*{0},
 \end{equation}
 \begin{align}\label{daf98sh}
 & \alpha_{ p}\tonde*{ l^{ \mathcal{R} A }, D' } 
 =
 \alpha_{ p}\tonde*{ l^{ A }, D' } ,
 & \psi_{\alpha_{p}}\tonde*{ \mathcal{R} A } = \psi_{\alpha_{p}}\tonde*{ A } 
 ,
 \end{align}
 for every $ A \in \mathbb{R}^{m\times n} $ and $ \mathcal{R} \in \mathrm{O}\tonde*{m} $.

 Moreover, recalling \cref{def:psi_alphap}, for every $ A \in \mathbb{R}^{m\times n} $,
 \begin{enumerate}
 \item\label{item:aniso_fun_meno_media} if $ \mathcal{G} =\set*{ \operatorname{id}_{\mathbb{R}^{n}}}$ and $ D $ is as in \cref{res:tessellation_useful},
 \begin{equation*}
 \psi_{\alpha_{p}}\tonde*{ A } 
 =
 \frac{1}{ \abs*{ D } ^2} \norm*{ \abs*{ l^{ A}}_{2}}^{ p }_{L^{p}\tonde*{D-\operatorname{bar}\tonde*{ D}}} 
 \text{ }\forall A\in\mathbb{R}^{m\times n}
 .
 \end{equation*}
 \item\label{item:iso_fun_meno_media} if $ \mathcal{G} = \mathrm{SO}\tonde*{n} $ and $ m = n $,
 \begin{equation*}
 \psi_{ \alpha_{ p , D, \Gamma ^{\mathcal{G}} ,n }}\tonde*{ A } 
 =
 \psi_{ \alpha_{ p , D, \Gamma ^{\mathcal{G}} ,1 }}\tonde*{ \lambda } ,
 \end{equation*}
 where $ \lambda \in \mathbb{R}^{n} $ is a vector obtained collecting (in an arbitrary order) the $ n $ eigenvalues
 (counted with multiplicity) of the symmetric and positive definite
 matrix $\sqrt{ A A ^{T}}$.
 \end{enumerate}
\end{proposition}
\begin{remark}
 Thanks to \cref{res:formula_psi}, the function $ \mathbb{R}^{n} \ni\nu\to \psi_{ \alpha_{ p , D, \Gamma ^{\mathcal{G}} ,1 }}\tonde*{\nu} $ in \cref{res:fun_meno_media}
 is constant on the unit sphere $\mathbb{S}^{ n -1}$
 (therefore it does not depend on permutations of the components of $\nu$).
 Moreover, $ \psi_{ \alpha_{ p , D, \Gamma ^{\mathcal{G}} ,1 }} $
 coincides with the function $\tilde{\psi}^{ D }_{ p }$ in \cite[Sec. 4]{FarGuaSch20-MR4109099}.
 In particular (recalling also \cref{res:formula_psi}) it coincides with the function $ \mathbb{R}^{n} \ni\nu\to \abs*{\nu}^{ p }_{2} $
 up to a multiplicative constant,
 which can be explicitly computed
 for some $ D \subset \mathbb{R}^{n} $ and some values of $ n $ and $ p $ \cite[cfr. Rem. 2.1]{FusMosSbo18-MR3816417}.
 For instance, if $ D $ is equal to the cube $ Q $, then for every $\nu\in \mathbb{R}^{n} $
 $\tilde{\psi}^{ Q }_{ p }(\nu)=\gamma( n , p ) \abs*{\nu}^{ p }_{2} $, where
 \begin{equation*}
 \gamma( n , p )= \max\limits_{\nu\in\mathbb{S}^{ n -1}} \int_{ Q } \abs*{ x \cdot \nu}^ p \thinspace\mathrm{d} x .
 \end{equation*}
 We also recall that $\gamma( n ,1)=\frac{1}{4}$ and $\gamma( n ,2)=\frac{1}{12}$ for every $ n \geq1$.

 More generally, if $ D $ is as in \cref{res:tessellation_useful},
 thanks to \cref{res:formula_psi} we get that,
 \begin{equation*}
 \tilde{\psi}^{ D }_{ p }(\mu)
 =
 \abs*{\mu}^{ p }_{2} \frac{1}{ \abs*{ D } ^2} \max\limits_{\nu\in\mathbb{S}^{ n -1}} \int_{ D } \abs*{\tonde*{ x - \operatorname{bar}\tonde*{ D }} \cdot \nu}^ p \thinspace\mathrm{d} x ,
 \end{equation*}
 for every $\mu\in \mathbb{R}^{n} $.
\end{remark}

\begin{proof}[Proof of \cref{res:fun_meno_media}]
\eqref{daf98sh} follows by the fact that $ \abs*{ \mathcal{R} z }_{2} = \abs*{ z }_{2} $ for every $ z \in \mathbb{R}^{m} $.

We now observe that for every $ D' \in \mathcal{G}^{ D}\tonde*{\mathbb{R}^{n}} $,
\begin{equation}\label{1d7h97}
 \begin{split}
 & \alpha_{ p}\tonde*{ l^{ A}, D' } 
 =
 \strokedint_{ D'} \abs*{ l^{ A}\tonde*{ x } - \strokedint_{ D' } l^{ A}\tonde*{ y } \thinspace\mathrm{d} y }^{ p }_{2} \thinspace\mathrm{d} x 
 \\
 &=
 \strokedint_{ D'} \abs*{ A \tonde*{ x - \strokedint_{ D' } y \thinspace\mathrm{d} y }}^{ p }_{2} \thinspace\mathrm{d} x 
 =
 \strokedint_{ D'} \abs*{ A \tonde*{ x - \operatorname{bar}\tonde*{ D' }}}^{ p }_{2} \thinspace\mathrm{d} x .
 \end{split}
\end{equation}

We prove \eqref{9832h9hh}.
Let $ A \in \mathcal{N}_{\psi_{\alpha_{p}}} $. Thanks to \cref{res:well-def_psi} and \eqref{1d7h97}, we get
\begin{equation*}
 0= \psi_{\alpha_{p}}\tonde*{ A } \geq G^{\alpha_{p}}_{1/2}\tonde*{ l^{ A} , Q } \geq2^{ p - n } \strokedint_{ D'} \abs*{ A \tonde*{ x - \operatorname{bar}\tonde*{ D' }}}^{ p }_{2} \thinspace\mathrm{d} x ,
\end{equation*}
where $ D' $ is an arbitrary set in $ \mathcal{G}^{ D }_{1/2}\tonde*{ Q } $.
This implies that the affine function $ A \tonde*{\cdot- \operatorname{bar}\tonde*{ D' }}$ is identically equal to $0$
in $ D' $ (which has positive measure). Then $ A =0$ and \eqref{9832h9hh} follows.
Then, recalling \cref{res:obs_strong_p_core}, we deduce that $ \alpha_{ p} $
is a strong $p$-core functional.

We prove \eqref{item:aniso_fun_meno_media}.
By \cref{res:formula_psi} and \eqref{1d7h97} we obtain,
\begin{align*}
 \psi_{\alpha_{p}}\tonde*{ A } 
 &=
 \frac{1}{ \abs*{ D }} \strokedint_{ D} \abs*{ A \tonde*{ x - \operatorname{bar}\tonde*{ D }}}^{ p }_{2} \thinspace\mathrm{d} x 
 =
 \frac{1}{ \abs*{ D } ^2} \norm*{ \abs*{ l^{ A}}_{2}}^{ p }_{L^{p}\tonde*{D-\operatorname{bar}\tonde*{ D}}} .
\end{align*}
We now prove \eqref{item:iso_fun_meno_media}.

Let $ \alpha_{ p} = \alpha_{ p , D, \Gamma ^{\mathcal{G}} ,n } $ and $ A \in \mathbb{R}^{n\times n} $.
We now consider a (left) polar decomposition of $ A $:
\begin{equation*}
 A =\sqrt{ A A ^{T}} O ,
\end{equation*}
where $ O \in \mathrm{O}\tonde*{n} $.
Now, let $ U \in \mathrm{O}\tonde*{n} $ such that $ U O \in \mathrm{SO}\tonde*{n} $
and
\begin{equation*}
 \sqrt{ A A ^{T}}= U ^{T} \Lambda U ,
\end{equation*}
where $ \Lambda $ is a diagonal matrix.
Then, by \eqref{daf98sh} and \cref{res:formula_psi},
\begin{equation}\label{1h98eah}
 \psi_{\alpha_{p}}\tonde*{ A } 
 =
 \psi_{\alpha_{p}}\tonde*{ U ^{T} \Lambda U O } 
 =
 \psi_{\alpha_{p}}\tonde*{ \Lambda } .
\end{equation}
Let $ \lambda \in \mathbb{R}^{n} $ be the vector obtained from the diagonal of $ \Lambda $,
so that,
\begin{equation}\label{diag123}
 \abs*{ \Lambda x }_{2} =\abs*{ \lambda \cdot x },
\end{equation}
for every $ x \in \mathbb{R}^{n} $.
Then, recalling \cref{res:prop_psi_alphap}, by \eqref{diag123} and \eqref{1d7h97},
 \begin{equation}\label{s808j320hr}
 \begin{split}
 \psi_{\alpha_{p}}\tonde*{ \Lambda } 
 &=
 \lim\limits_{\varepsilon\to 0^{+}} \varepsilon ^{ n - p } \sup\limits_{\mathcal{G}_{\varepsilon}} \strokedint_{ D'} \abs*{ \Lambda \tonde*{ x - \strokedint_{ D' } y \thinspace\mathrm{d} y }}^{ p }_{2} \thinspace\mathrm{d} x \\
 &=
 \lim\limits_{\varepsilon\to 0^{+}} \varepsilon ^{ n - p } \sup\limits_{\mathcal{G}_{\varepsilon}} \strokedint_{ D'} \abs*{\tonde*{ \lambda \cdot x - \strokedint_{ D' } \lambda \cdot y \thinspace\mathrm{d} y }}^ p \thinspace\mathrm{d} x 
 = \psi_{ \alpha_{ p , D, \Gamma ^{\mathcal{G}} ,1 }}\tonde*{ \lambda } ,
 \end{split}
 \end{equation}
where for every $ \varepsilon >0$, each $ \mathcal{G}_{\varepsilon} $ is a family made of pairwise disjoint elements of $ \mathcal{G}^{ D}_{\varepsilon}\tonde*{ Q} $ 
and the supremum is taken with respect to all of these families $ \mathcal{G}_{\varepsilon} $.
Combining \eqref{s808j320hr} and \eqref{1h98eah} we conclude the proof.
\end{proof}

The following example is similar to \cref{ex:fun_meno_media} (actually when $p=2$ it coincides with \cref{ex:fun_meno_media}).
\begin{example}[][ex:fun_meno_c_inf]

 Let $ p \in \left[1,\infty\right) $.
 We define,
\begin{equation*}
 \alpha_{ p}\tonde*{ u , D' } 
 \coloneqq
 \inf\limits_{ c \in \mathbb{R}^{m}} \strokedint_{ D'} \abs*{ u\tonde*{ x } - c }^{ p }_{2} \thinspace\mathrm{d} x .
 \end{equation*}
 Then, $ \alpha_{ p} : L^{p}_{\mathrm{loc}}\tonde*{\mathbb{R}^{n};\mathbb{R}^{m}} \times \mathcal{E}^{ D}_{ \Gamma }\tonde*{\mathbb{R}^{n}} \to \left[0,+\infty\right) $ is a $p$-core functional for any $ \Gamma $ subgroup of $ \mathrm{Aff}\tonde*{\mathbb{R}^{n}} $.
 Moreover, by Fatou's lemma, $\restr{ \alpha_{ p}\tonde*{\cdot, D' }}{ L^{p}\tonde*{\mathbb{R}^{n};\mathbb{R}^{m}}}: L^{p}\tonde*{\mathbb{R}^{n};\mathbb{R}^{m}} \to \left[0,+\infty\right) $
 is lower semicontinuous for every $ D' \in \mathcal{E}^{ D}_{ \Gamma }\tonde*{\mathbb{R}^{n}} $.
\end{example}

The next example is inspired by \cite{PonSpe17-MR3614653}.

\begin{example}[][ex:integ_doppio]
 Let $ p \in \left[1,\infty\right) $.
 We define,
 \begin{equation*}
 \alpha_{ p}\tonde*{ u , D' } 
 \coloneqq
 \strokedint_{ D'} \strokedint_{ D'} \abs*{ u\tonde*{ x } - u\tonde*{ y }}^{ p }_{2} \thinspace\mathrm{d} x \mathrm{d} y .
 \end{equation*}
 Then, $ \alpha_{ p} : L^{p}_{\mathrm{loc}}\tonde*{\mathbb{R}^{n};\mathbb{R}^{m}} \times \mathcal{E}^{ D}_{ \Gamma }\tonde*{\mathbb{R}^{n}} \to \left[0,+\infty\right) $ is a $p$-core functional for any $ \Gamma $ subgroup of $ \mathrm{Aff}\tonde*{\mathbb{R}^{n}} $.
 Moreover, by Fatou's lemma, $\restr{ \alpha_{ p}\tonde*{\cdot, D' }}{ L^{p}\tonde*{\mathbb{R}^{n};\mathbb{R}^{m}}}: L^{p}\tonde*{\mathbb{R}^{n};\mathbb{R}^{m}} \to \left[0,+\infty\right) $
 is lower semicontinuous for every $ D' \in \mathcal{E}^{ D}_{ \Gamma }\tonde*{\mathbb{R}^{n}} $.
\end{example}

\begin{proposition}[][res:int_doppio] 
 Let $ p \in \left[1,\infty\right) $ and
 $ \mathcal{G} $ be a subgroup of $ \mathrm{SO}\tonde*{n} $.
 Let $ \alpha_{ p} = \alpha_{ p , D, \Gamma ^{\mathcal{G}} ,m } $ 
 be the $p$-core functional in \cref{ex:integ_doppio}
 subordinated to $ \Gamma ^{\mathcal{G}} $ as in \cref{def:Gamma_G}.

 Then, $ \alpha_{ p} $ is a strong $p$-core functional,
 \begin{equation}\label{21j03h}
 \mathcal{N}_{\psi_{\alpha_{p}}} =\set*{0},
 \end{equation}
 \begin{align}\label{fwh97}
 & \alpha_{ p}\tonde*{ l^{ \mathcal{R} A }, D' } 
 =
 \alpha_{ p}\tonde*{ l^{ A }, D' } ,
 & \psi_{\alpha_{p}}\tonde*{ \mathcal{R} A } = \psi_{\alpha_{p}}\tonde*{ A } 
 ,
 \end{align}
 for every $ A \in \mathbb{R}^{m\times n} $ and $ \mathcal{R} \in \mathrm{O}\tonde*{m} $.

 Moreover, recalling \cref{def:psi_alphap}, for every $ A \in \mathbb{R}^{m\times n} $,
 \begin{enumerate}
 \item\label{item:aniso_int_doppio} if $ \mathcal{G} =\set*{ \operatorname{id}_{\mathbb{R}^{n}}}$ and $ D $ is as in \cref{res:tessellation_useful},
 \begin{equation*}
 \psi_{\alpha_{p}}\tonde*{ A } 
 =
 \frac{1}{ \abs*{ D }} \strokedint_{ D} \strokedint_{ D} \abs*{ A \tonde*{ x - y }}^{ p }_{2} \thinspace\mathrm{d} x \mathrm{d} y 
 .
 \end{equation*}
 \item\label{item:iso_int_doppio} if $ \mathcal{G} = \mathrm{SO}\tonde*{n} $ and $ m = n $,
 \begin{equation*}
 \psi_{ \alpha_{ p , D, \Gamma ^{\mathcal{G}} ,n }}\tonde*{ A } 
 =
 \psi_{ \alpha_{ p , D, \Gamma ^{\mathcal{G}} ,1 }}\tonde*{ \lambda } ,
 \end{equation*}
 where $ \lambda \in \mathbb{R}^{n} $ is a vector obtained collecting (in an arbitrary order) the $ n $ eigenvalues
 (counted with multiplicity) of the symmetric and positive definite
 matrix $\sqrt{ A A ^{T}}$.
 \end{enumerate}
\end{proposition}
\begin{remark}
 Thanks to \cref{res:formula_psi}, the function $ \mathbb{R}^{n} \ni\nu\to \psi_{ \alpha_{ p , D, \Gamma ^{\mathcal{G}} ,1 }}\tonde*{\nu} $ in \cref{res:int_doppio}
 is constant on the unit sphere $\mathbb{S}^{ n -1}$
 (therefore it does not depend on permutations of the components of $\nu$).

 More generally, if $ D $ is as in \cref{res:tessellation_useful},
 thanks to \cref{res:formula_psi} we get that,
 \begin{equation*}
 \psi_{ \alpha_{ p , D, \Gamma ^{\mathcal{G}} ,1 }}\tonde*{\mu} 
 =
 \abs*{\mu}^{ p }_{2} \frac{1}{ \abs*{ D } ^3} \max\limits_{\nu\in\mathbb{S}^{ n -1}} 
 \int_{ D } \int_{ D } \abs*{\tonde*{ x - y } \cdot \nu}^ p \thinspace\mathrm{d} x \mathrm{d} y
 \text{ }\forall A\in\mathbb{R}^{m\times n},
 \end{equation*}
 for every $\mu\in \mathbb{R}^{n} $.
\end{remark}
\begin{proof}[Proof of \cref{res:int_doppio}]
\eqref{fwh97} follows by the fact that $ \abs*{ \mathcal{R} z }_{2} = \abs*{ z }_{2} $ for every $ z \in \mathbb{R}^{m} $.

We now observe that for every $ D' \in \mathcal{G}^{ D}\tonde*{\mathbb{R}^{n}} $,
\begin{equation}\label{9dsajh9y}
 \begin{split}
 & \alpha_{ p}\tonde*{ l^{ A}, D' } 
 =
 \strokedint_{ D'} \strokedint_{ D'} \abs*{ A \tonde*{ x - y }}^{ p }_{2} \thinspace\mathrm{d} x \mathrm{d} y\text{ }\forall A\in\mathbb{R}^{m\times n}.
 \end{split}
\end{equation}
\eqref{item:aniso_fun_meno_media} follows from \cref{res:formula_psi} and \eqref{9dsajh9y}.
The proof of \eqref{21j03h}, \eqref{item:iso_int_doppio} and the fact that
$ \alpha_{ p} $ is a strong $p$-core functional follows the same argument 
used in the proof of \cref{res:fun_meno_media}.
\end{proof}

The following result shows how \cref{ex:fun_meno_media},
\cref{ex:fun_meno_c_inf}
and \cref{ex:integ_doppio} can be used in the study of $W^{1,p}\tonde*{\Omega;\mathbb{R}^{m}}$. 
\begin{corollary}[][res:W1p_char]
 Let $ \Omega \subseteq \mathbb{R}^{n} $ be an open set, $ m \geq1$ a natural number, $ p \in \left[1,\infty\right) $,
 $ \mathcal{G} $ a subgroup of $ \mathrm{SO}\tonde*{n} $ and $ D \subseteq \mathbb{R}^{n} $ be a bounded connected open set with Lipschitz boundary.
 Let $ \alpha_{ p} = \alpha_{ p , D, \Gamma ^{\mathcal{G}} ,m } $ 
 one of the $p$-core functional in \cref{ex:fun_meno_media},
 \cref{ex:fun_meno_c_inf}
 or \cref{ex:integ_doppio}
 subordinated to $ \Gamma ^{\mathcal{G}} $ as in \cref{def:Gamma_G}.

 Then,
 \begin{enumerate}
 \item\label{zxc1} for every $ u \in W^{1,p}_{\mathrm{loc}}\tonde*{\Omega;\mathbb{R}^{m}} $,
 \begin{equation*}
 \lim\limits_{\varepsilon\to 0^{+}} G^{\alpha_{p}}_{ \varepsilon }\tonde*{ u,\Omega } = \int_{ \Omega } \psi_{\alpha_{p}}\tonde*{ \nabla u } ,
 \end{equation*}
 where $ \psi_{\alpha_{p}} $ is as in \cref{def:psi_alphap}
 and $ G^{\alpha_{p}}_{ \varepsilon } $ is as in \cref{def:tipo_BMO}.

 Furthermore, if $ \mathcal{G} =\set*{ \operatorname{id}_{\mathbb{R}^{n}}}$ or $ \mathcal{G} = \mathrm{SO}\tonde*{n} $ and $ D $ is as in \cref{res:tessellation_useful} 
 the function $ \psi_{\alpha_{p}} $ can be written more explicitly
 according to \cref{res:fun_meno_media} or \cref{res:int_doppio}.
 \item\label{zxc2} if $ p \in \left(1,\infty\right) $, for every $ u \in L^{p}\tonde*{\Omega;\mathbb{R}^{n}} $,
 \begin{equation*} 
 \liminf\limits_{\varepsilon\to 0^{+}} G^{\alpha_{p}}_{ \varepsilon }\tonde*{ u,\Omega } <+\infty\iff u \in W^{1,p}\tonde*{\Omega;\mathbb{R}^{m}} .
 \end{equation*}
 \item\label{zxc3} if $ \Omega \subseteq \mathbb{R}^{n} $ is a connected open set,
 $ p \in \left(1,\infty\right) $ and $ u \in L^{p}_{\mathrm{loc}}\tonde*{\Omega;\mathbb{R}^{n}} $ satisfies $ \liminf\limits_{\varepsilon\to 0^{+}} G^{\alpha_{p}}_{ \varepsilon }\tonde*{ u,\Omega } =0$,
 then
 \begin{equation*}
 u\tonde*{ x } = h \text{ }\forall x \in \Omega , 
 \end{equation*}
 where $ h \in \mathbb{R}^{m} $.
 \end{enumerate} 
\end{corollary}
\begin{proof}
 The proof of \eqref{zxc1} and \eqref{zxc2} follows combining \cref{res:fun_meno_media}, \cref{res:int_doppio},
 \cref{thm:Main_BMO} and \cref{res:cor_Main_BMO2}.
 To prove \eqref{zxc3} we apply \cref{res:constancy}.
\end{proof}

\subsection{Characterization of \texorpdfstring{$E^{1,p}(\Omega;\mathbb{R}^n)$}{E1p(Omega;Rn)}}
\phantom{}

In this section we will consider examples of $p$-core functionals used to characterize the space $E^{1,p}(\Omega;\mathbb{R}^n)$, of $L^{p}(\Omega;\mathbb{R}^n)$ functions whose distributional symmetric gradient is $p$-integrable, when $p\in(1,\infty)$ and to approximate the following $L^p$ norm of the distributional symmetric gradient of functions in $E^{1,p}(\Omega;\mathbb{R}^n)$ for $p\in[1,\infty)$:
\begin{equation*}
    \int_\Omega |\mathcal{E} u|^p_{\alpha_p}\mathrm{d} x,
\end{equation*}
where $|.|_{\alpha_p}\coloneqq\psi_{\alpha_p}^{1/p}(.)$
is a (possibly anisotropic) norm in $\mathbb{R}^{m\times n}$ (\cref{res:WL_antisym}).

In this section $ D \subseteq \mathbb{R}^{n} $ will always denote a bounded connected open set with Lipschitz boundary.

\begin{example}[][ex:inf_antisym]
 Let $ p \in \left[1,\infty\right) $ and $ \mathcal{G} =\set*{ \operatorname{id}_{\mathbb{R}^{n}}}$.
 We define,
 \begin{equation}\label{j0dsa}
 \alpha_{ p}\tonde*{ u , D' } 
 \coloneqq
 \inf\limits_{ A \in \mathbb{R}^{n\times n}_{\mathrm{skew}}} 
 \strokedint_{ D'} \abs*{ u\tonde*{ x } - A \tonde*{ x - z }- u_{ D'}}^{ p }_{2} \thinspace\mathrm{d} x ,
 \end{equation}
 where $ z = \operatorname{bar}\tonde*{ D' } $ is the barycenter of $ D' $.
 Then, $ \alpha_{ p} = \alpha_{ p , D, \Gamma ^{\mathcal{G}} ,m } $ 
 is a $p$-core functional as in \cref{def:core_function}
 subordinated to $ \Gamma ^{\mathcal{G}} $ as in \cref{def:Gamma_G}.
 Moreover, if $ p \in \left(1,\infty\right) $ $ \alpha_{ p} $ is a strong $p$-core functional as in \cref{def:strong_core_function},
 thanks to the \emph{Korn inequality} (cfr. \cite[Eq. (2.2)]{DurMus04-MR2108898}).
\end{example}
\begin{remark}[][res:min_antisym]
 Under the assumptions of \cref{ex:inf_antisym}, we observe that the infimum 
 in \eqref{j0dsa} is actually a minimum.
 Indeed, let $ u \in L^{p}_{\mathrm{loc}}\tonde*{\mathbb{R}^{n};\mathbb{R}^{n}} $ and $ D' \in \mathcal{G}^{ D}\tonde*{\mathbb{R}^{n}} $.
 We define,
 \begin{equation}
 f\tonde*{ A } \coloneqq \strokedint_{ D'} \abs*{ u\tonde*{ x } - A \tonde*{ x - \operatorname{bar}\tonde*{ D' }}- u_{ D'}}^{ p }_{2} \thinspace\mathrm{d} x \text{ for every } A \in \mathbb{R}^{n\times n} .
 \end{equation}
 We observe that $ f $ is coercive, since for every $ \eta \in \left(0,1\right) $,
 \begin{equation*}
 f\tonde*{ A } \geq\frac{1}{(1+ \eta )^ p } \strokedint_{ D'} \abs*{ A \tonde*{ x - \operatorname{bar}\tonde*{ D' }}}^{ p }_{2} \thinspace\mathrm{d} x 
 -
 \frac{1}{ \eta ^ p } \strokedint_{ D'} \abs*{ u\tonde*{ x } - u_{ D'}}^{ p }_{2} \thinspace\mathrm{d} x ,
 \end{equation*}
 \begin{equation*}
 \lim\limits_{ \abs*{ A }_{2} \to +\infty} \strokedint_{ D'} \abs*{ A \tonde*{ x - \operatorname{bar}\tonde*{ D' }}}^{ p }_{2} =+\infty.
 \end{equation*}
 Moreover $ f $ is convex and $ \mathbb{R}^{n\times n}_{\mathrm{skew}} $ is a closed subset of $ \mathbb{R}^{n} $.
 This, implies that the minimum of $\restr{ f }{ \mathbb{R}^{n\times n}_{\mathrm{skew}}}$ is attained.
\end{remark}

\begin{proposition}[][res:inf_antisym_lsc]
 Let $ p \in \left[1,\infty\right) $ and
 $ \mathcal{G} =\set*{ \operatorname{id}_{\mathbb{R}^{n}}}$.
 Let $ \alpha_{ p} = \alpha_{ p , D, \Gamma ^{\mathcal{G}} ,m } $ 
 be the $p$-core functional in \cref{ex:inf_antisym}
 subordinated to $ \Gamma ^{\mathcal{G}} $ as in \cref{def:Gamma_G}.
 Then,
 $\restr{ \alpha_{ p}\tonde*{\cdot, D' }}{ L^{p}\tonde*{\mathbb{R}^{n};\mathbb{R}^{n}}}: L^{p}\tonde*{\mathbb{R}^{n};\mathbb{R}^{n}} \to \left[0,+\infty\right) $
 is lower semicontinuous for every $ D' \in \mathcal{G}^{ D}\tonde*{\mathbb{R}^{n}} $.
\end{proposition}
\begin{proof}
 Let $ D' \in \mathcal{G}^{ D}\tonde*{\mathbb{R}^{n}} $, $ z = \operatorname{bar}\tonde*{ D' } $, $ u \in L^{p}\tonde*{\mathbb{R}^{n};\mathbb{R}^{n}} $ and $ \graffe*{u_{k}}_{k\in\mathbb{N}} \subset L^{p}\tonde*{\mathbb{R}^{n};\mathbb{R}^{n}} $ such that $ u_{k} \to u $ in $ L^{p}\tonde*{\mathbb{R}^{n};\mathbb{R}^{n}} $.
 Without loss of generality we assume that $ u_{k} \to u $ $ \operatorname{a.e.} $ in $ \mathbb{R}^{n} $ and,
 \begin{equation}\label{wlogg}
 \lim\limits_{ k\to +\infty} \alpha_{ p}\tonde*{ u_{k}, D' } 
 =
 \liminf\limits_{ k\to +\infty} \alpha_{ p}\tonde*{ u_{k}, D' } <+\infty.
 \end{equation}
 For every $ k \in \mathbb{N} $ there exist $ A_{k} \in \mathbb{R}^{n\times n}_{\mathrm{skew}} $ such that,
 \begin{equation}\label{definf}
 \strokedint_{ D'} \abs*{ u_{k}\tonde*{ x } - A_{k} ( x - z )- \strokedint_{ D'} u_{k}}^{ p }_{2} \thinspace\mathrm{d} x \leq \alpha_{ p}\tonde*{ u_{k}, D' } +1/ k .
 \end{equation}
 Since $ \graffe*{u_{k}}_{k\in\mathbb{N}} $ is bounded in $ L^{p}\tonde*{\mathbb{R}^{n};\mathbb{R}^{n}} $, combining \eqref{wlogg} and \eqref{definf}
 we obtain that
 \begin{equation*}
 \sup\limits_{ k \in \mathbb{N}} \abs*{ A_{k}}_{2} <+\infty.
 \end{equation*}
 Then, up to subsequences, $ A_{k} \to A_{\infty} $ for some $ A_{\infty} \in \mathbb{R}^{n\times n}_{\mathrm{skew}} $.
 Then, by Fatou's lemma and \eqref{definf} we get that 
 \begin{equation*}
 \begin{split}
 & \alpha_{ p}\tonde*{ u , D' } \leq \strokedint_{ D'} \abs*{ u\tonde*{ x } - A_{\infty} ( x - z )- \strokedint_{ D'} u }^{ p }_{2} \thinspace\mathrm{d} x \\
 &\leq \liminf\limits_{ k\to +\infty} \strokedint_{ D'} \abs*{ u_{k}\tonde*{ x } - A_{k} ( x - z )- \strokedint_{ D'} u_{k}}^{ p }_{2} \thinspace\mathrm{d} x 
 \leq \liminf\limits_{ k\to +\infty} \alpha_{ p}\tonde*{ u_{k}, D' } .
 \end{split}
 \end{equation*}
\end{proof}

\begin{proposition}[][res:inf_antisym]
 Let $ p \in \left[1,\infty\right) $ and
 $ \mathcal{G} =\set*{ \operatorname{id}_{\mathbb{R}^{n}}}$.
 Let $ \alpha_{ p} = \alpha_{ p , D, \Gamma ^{\mathcal{G}} ,m } $ 
 be the $p$-core functional in \cref{ex:inf_antisym}
 subordinated to $ \Gamma ^{\mathcal{G}} $ as in \cref{def:Gamma_G}.
 Recalling \cref{def:psi_alphap},
 if $ D $ is as in \cref{res:tessellation_useful},
 then for every $ B \in \mathbb{R}^{n\times n} $,
 \begin{equation}\label{h3298hedw}
 \psi_{\alpha_{p}}\tonde*{ B } 
 =\frac{1}{ \abs*{ D }} \inf\limits_{ A \in \mathbb{R}^{n\times n}_{\mathrm{skew}}} \strokedint_{ D} \abs*{\tonde*{ B - A }\tonde*{ x - \operatorname{bar}\tonde*{ D }}}^{ p }_{2} \thinspace\mathrm{d} x ,
 \end{equation}
 and
 \begin{equation}\label{nul_skew}
 \mathcal{N}_{\psi_{\alpha_{p}}} = \mathbb{R}^{n\times n}_{\mathrm{skew}} ,
 \end{equation}
 so that 
 \begin{equation*}
 \psi_{\alpha_{p}}\tonde*{ B } = \psi_{\alpha_{p}}\tonde*{\frac{ B + B ^{T}}{2}} ,
 \end{equation*}
 for every $ B \in \mathbb{R}^{n\times n} $.

 Moreover, if $ p =2$ and $ D = Q\tonde*{x;r} $ for some $ x \in \mathbb{R}^{n} $ and $ r >0$,
 then,
 \begin{equation}\label{13290hj}
 \psi_{\alpha_{2}}\tonde*{ B } =\frac{1}{ \abs*{ D } ^2} \norm*{ \abs*{ l^{ A}}_{2}}^{2}_{L^{2}\tonde*{D-\operatorname{bar}\tonde*{ D}}} ,
 \end{equation}
 where
 \begin{equation*}
 A =\frac{ B + B ^{T}}{2}.
 \end{equation*}

\end{proposition}
\begin{proof}
 We observe that for every $ D' \in \mathcal{G}^{ D}\tonde*{\mathbb{R}^{n}} $ and $ B \in \mathbb{R}^{n\times n} $,
 \begin{equation}\label{fd94h9122}
 \begin{split}
 & \alpha_{ p}\tonde*{ l^{ B}, D' } 
 =
 \inf\limits_{ A \in \mathbb{R}^{n\times n}_{\mathrm{skew}}} \strokedint_{ D'} \abs*{ l^{ B}\tonde*{ x } - A \tonde*{ x - \operatorname{bar}\tonde*{ D' }}- \strokedint_{ D' } l^{ B}\tonde*{ y } \thinspace\mathrm{d} y }^{ p }_{2} \thinspace\mathrm{d} x 
 \\
 &=
 \inf\limits_{ A \in \mathbb{R}^{n\times n}_{\mathrm{skew}}} \strokedint_{ D'} \abs*{\tonde*{ B - A }\tonde*{ x - \operatorname{bar}\tonde*{ D' }}}^{ p }_{2} \thinspace\mathrm{d} x .
 \end{split}
 \end{equation}
 Then, \eqref{h3298hedw} follows from \cref{res:formula_psi}.
 Moreover, \eqref{nul_skew} follows from \cref{res:min_antisym} and \eqref{h3298hedw}.

 Let $ B \in \mathbb{R}^{m\times n} $ and define,
 \begin{equation*}
 f_{B}\tonde*{ A } \coloneqq \strokedint_{ D - \operatorname{bar}\tonde*{ D }} \abs*{\tonde*{ B - A } x }^{ p }_{2} \thinspace\mathrm{d} x \text{ }\forall A \in \mathbb{R}^{m\times n} .
 \end{equation*}
 Then, by \eqref{h3298hedw} and \cref{res:min_antisym},
 \begin{equation*}
 \psi_{\alpha_{p}}\tonde*{ B } 
 =\frac{1}{ \abs*{ D }} \min\limits_{ A \in \mathbb{R}^{n\times n}_{\mathrm{skew}}} f_{B}\tonde*{ A } 
 =\frac{1}{ \abs*{ D }} \min\limits_{ A \in \mathbb{R}^{n\times n}} f_{B}\tonde*{ A - A ^{T}} .
 \end{equation*}

 If $ p >1$ we can then find the minimum point of the strictly convex function $\restr{ f_{B}}{ \mathbb{R}^{n\times n}_{\mathrm{skew}}}$
 by differentiating $ \mathbb{R}^{n\times n} \ni A \mapsto f_{B}\tonde*{ A - A ^{T}} $ and finding its stationary point.

 If particular, if $ p =2$ we obtain the following conditions,
 \begin{equation}\label{vspl1}
 \begin{split}
 &\sum_{ j =1}^{ n }\tonde*{ A_{ k j } - A_{ j k } - B_{ k j }} \int_{ D - \operatorname{bar}\tonde*{ D }} x_{ j } x_{ l } \thinspace\mathrm{d} x 
 \\
 &=\sum_{ j =1}^{ n }\tonde*{ A_{ l j } - A_{ j l } - B_{ l j }} \int_{ D - \operatorname{bar}\tonde*{ D }} x_{ j } x_{ k } \thinspace\mathrm{d} x ,
 \end{split}
 \end{equation}
 for every $ k , l \in\set*{1,\dots, n }$.

 Without loss of generality we may assume $ \operatorname{bar}\tonde*{ D} =0$.
 Then, $ D $ is equal to the cube $\tonde{-\frac{ r }{2},\frac{ r }{2}}^ n $, so that,
 \begin{equation}\label{osdk1}
 \int_{ D } x_{ i } x_{ j } \thinspace\mathrm{d} x =0\text{ for every } i , j \in\set*{1,\dots, n },
 \end{equation}
 \begin{equation}\label{osdk2}
 \int_{ D } x_{ i } ^2 \thinspace\mathrm{d} x = \int_{ D } x_{ j } ^2 \thinspace\mathrm{d} x \text{ for every } i , j \in\set*{1,\dots, n }.
 \end{equation}
 Thanks to \eqref{osdk1} then \eqref{vspl1} becomes,
 \begin{equation}\label{vspl2}
 \tonde*{ A_{ k l } - A_{ l k } - B_{ k l }} \int_{ D } x_{ l } ^2 \thinspace\mathrm{d} x =
 \tonde*{ A_{ l k } - A_{ k l } - B_{ l k }} \int_{ D } x_{ k } ^2 \thinspace\mathrm{d} x .
 \end{equation}
 Finally, thanks to \eqref{osdk2},
 \eqref{vspl2} becomes,
 \begin{equation}\label{vspl3}
 A_{ k l } - A_{ l k } =\frac{ B_{ k l } - B_{ l k }}{2} \text{ for every } k , l \in\set*{1,\dots, n }.
 \end{equation}
 \eqref{vspl3} implies that the minimum of $\restr{ f_{B}}{ \mathbb{R}^{n\times n}_{\mathrm{skew}}}$
 is attained when $ A =\tonde*{ B - B ^{T}}/2$. Therefore,
 \begin{equation*}
 \psi_{\alpha_{p}}\tonde*{ B } 
 =\frac{1}{ \abs*{ D }} \strokedint_{ D - \operatorname{bar}\tonde*{ D }} \abs*{\frac{ B + B ^{T}}{2}\tonde*{ x }}^{2}_{2} \thinspace\mathrm{d} x ,
 \end{equation*}
 so \eqref{13290hj} is proved.
\end{proof}

The following example is similar to \cref{ex:inf_antisym}
\begin{example}[][ex:inf_antisym_meno_c]
    Let $ p \in \left[1,\infty\right) $ and $ \mathcal{G} =\set*{ \operatorname{id}_{\mathbb{R}^{n}}}$.
 We define,
\begin{equation*}
 \alpha_{ p}\tonde*{ u , D' } 
 \coloneqq
 \inf\limits_{ A \in \mathbb{R}^{n\times n}_{\mathrm{skew}} , c \in \mathbb{R}^{m}} 
 \strokedint_{ D'} \abs*{ u\tonde*{ x } - A \tonde*{ x - z }- c }^{ p }_{2} \thinspace\mathrm{d} x ,
 \end{equation*}
 where $ z = \operatorname{bar}\tonde*{ D' } $ is the barycenter of $ D' $.
 Then, $ \alpha_{ p} = \alpha_{ p , D, \Gamma ^{\mathcal{G}} ,m } $ 
 is a $p$-core functional as in \cref{def:core_function}
 subordinated to $ \Gamma ^{\mathcal{G}} $ as in \cref{def:Gamma_G}.
 Moreover, if $ p \in \left(1,\infty\right) $ $ \alpha_{ p} $ is a strong $p$-core functional as in \cref{def:strong_core_function},
 thanks to the \emph{Korn inequality} (cfr. \cite[Eq. (2.2)]{DurMus04-MR2108898}).

\end{example}

The following result shows how \cref{ex:inf_antisym} can be used in the study of $E^{1,p}\tonde*{\Omega;\mathbb{R}^{n}}$.
\begin{corollary}[][res:WL_antisym]
 Let $ \Omega \subseteq \mathbb{R}^{n} $ be an open set, $ p \in \left[1,\infty\right) $,
 $ \mathcal{G} =\set*{ \operatorname{id}_{\mathbb{R}^{n}}}$ and $ D \subseteq \mathbb{R}^{n} $ be a bounded connected open set with Lipschitz boundary.
 Let $ \alpha_{ p} = \alpha_{ p , D, \Gamma ^{\mathcal{G}} ,m } $ 
 be one of the $p$-core functionals in \cref{ex:inf_antisym} or \cref{ex:inf_antisym_meno_c}
 subordinated to $ \Gamma ^{\mathcal{G}} $ as in \cref{def:Gamma_G}.

 Then,
 \begin{enumerate}
 \item\label{abc1} for every $ u \in E^{1,p}_{\mathrm{loc}}\tonde*{\Omega;\mathbb{R}^{n}} $,
 \begin{equation*}
 \lim\limits_{\varepsilon\to 0^{+}} G^{\alpha_{p}}_{ \varepsilon }\tonde*{ u,\Omega } = \int_{ \Omega } \psi_{\alpha_{p}}\tonde*{ \mathcal{E}u } ,
 \end{equation*}
 where $ \psi_{\alpha_{p}} $ is as in \cref{def:psi_alphap}
 and $ G^{\alpha_{p}}_{ \varepsilon } $ is as in \cref{def:tipo_BMO}.
 Furthermore, if $ D $ is as in \cref{res:tessellation_useful} 
 the function $ \psi_{\alpha_{p}} $ can be written more explicitly
 according to \cref{res:inf_antisym}.
 \item\label{abc2} if $ p \in \left(1,\infty\right) $, for every $ u \in L^{p}\tonde*{\Omega;\mathbb{R}^{n}} $,
 \begin{equation*} 
 \liminf\limits_{\varepsilon\to 0^{+}} G^{\alpha_{p}}_{ \varepsilon }\tonde*{ u,\Omega } <+\infty\iff u \in E^{1,p}\tonde*{\Omega;\mathbb{R}^{n}} .
 \end{equation*}
 \item\label{abc3} if $ \Omega \subseteq \mathbb{R}^{n} $ is a connected open set,
 $ p \in \left(1,\infty\right) $ and $ u \in L^{p}_{\mathrm{loc}}\tonde*{\Omega;\mathbb{R}^{n}} $ satisfies $ \liminf\limits_{\varepsilon\to 0^{+}} G^{\alpha_{p}}_{ \varepsilon }\tonde*{ u,\Omega } =0$,
 then
 \begin{equation}\label{9j02}
 u\tonde*{ x } = A x + h \text{ }\forall x \in \Omega , 
 \end{equation}
 where $ A \in \mathbb{R}^{n\times n}_{\mathrm{skew}} $ and $ h \in \mathbb{R}^{n} $.
 \end{enumerate} 
\end{corollary}
\begin{proof}
 The proof of \eqref{abc1} and \eqref{abc2} follows combining \cref{res:inf_antisym}, \cref{res:inf_antisym_lsc},
 \cref{thm:Main_BMO_WL} and \cref{res:cor_Main_BMO2_WL}.

 To prove \eqref{abc3} we apply \cref{res:constancy} and \cref{res:inf_antisym}
 to deduce that
 for every $ D' \in \mathcal{G}^{ D}\tonde*{\Omega} $ there exist $ A \in \mathbb{R}^{n\times n}_{\mathrm{skew}} $ and $ h \in \mathbb{R}^{n} $ (possibly depending on $ D' $)
 s.t. 
 \begin{equation*}
 u\tonde*{ x } = A x + h \text{ }\forall x \in D' .
 \end{equation*}
 At this point exploiting the connectedness of $ \Omega $ we can employ the simple argument in 
 \cite[Lem. 8]{Ari17-MR3624874} to deduce that \cref{9j02} holds
 for $ A \in \mathbb{R}^{n\times n}_{\mathrm{skew}} $ and $ h \in \mathbb{R}^{n} $.
\end{proof}

\subsection{Other examples of \texorpdfstring{$p$-core}{p-core} functionals}
 \begin{example}[][ex:trivial]
 Let $ m \geq1$ be a natural number and $ p \in \left[1,\infty\right) $ and $ \Gamma $ a subgroup of $ \mathrm{Aff}\tonde*{\mathbb{R}^{n}} $.
 Let $ D \subseteq \mathbb{R}^{n} $ be a bounded open set.
 We define,
 \begin{equation*}
 \alpha_{ p}\tonde*{ u , D' } \coloneqq0\text{ }
 \forall u \in L^{p}\tonde*{\mathbb{R}^{n};\mathbb{R}^{m}} 
 \forall D' \in \mathcal{E}^{ D}_{ \Gamma }\tonde*{\mathbb{R}^{n}} .
 \end{equation*}
 $ \alpha_{ p} \equiv0$ is a (trivial) strong $p$-core functional (cfr. \cref{res:obs_strong_p_core}).
\end{example}

If a $p$-core functional $ \alpha_{ p} $ is identically equal to zero (\cref{ex:trivial}),
then $ \psi_{\alpha_{p}} \equiv0$, but the vice-versa does not hold (see \cref{ex:counterex_psi_null}).

\begin{example}[][ex:counterex_psi_null]
 The following example is somehow pathological: linear maps
 are not able to entirely describe the behavior of the $p$-core functional $ \alpha_{ p} $ when 
 we are interested in computing the pointwise limit of $ G^{\alpha_{p}}_{ \varepsilon }\tonde*{ \cdot, \Omega } $ as
 $ \varepsilon \to 0^{+} $.

 We consider a $p$-core functional
 subordinated to $ p \in \left[1,\infty\right) $,
 the reference set $ D \coloneqq Q =(-1/2,1/2)^ n $, $ m \geq1$
 and the subgroup $ \Gamma ^{\mathcal{G}} $ with $ \mathcal{G} \coloneqq\set*{ \operatorname{id}_{\mathbb{R}^{n}}}$.
 The $p$-core functional is defined by,
 \begin{equation*}
 \alpha_{ p}\tonde*{ u , D' } 
 \coloneqq
 \inf\limits_{ A \in \mathbb{R}^{m\times n}} 
 \strokedint_{ D'} \abs*{ u\tonde*{ x } - A \tonde*{ x - \operatorname{bar}\tonde*{ D' }}- u_{ D'}}^{ p }_{2} \thinspace\mathrm{d} x .
 \end{equation*}
 Then, $ \psi_{\alpha_{p}} \equiv0$, since
 \begin{equation*}
 \alpha_{ p}\tonde*{ l^{ B}, D' } = \inf\limits_{ A \in \mathbb{R}^{m\times n}} \strokedint_{ D'} \abs*{\tonde*{ B - A }\tonde*{ x - \operatorname{bar}\tonde*{ D' }}}^{ p }_{2} \thinspace\mathrm{d} x =0,
 \end{equation*}
 for all $ B \in \mathbb{R}^{m\times n} $ and $ D' \in \mathcal{G}^{ D}\tonde*{\mathbb{R}^{n}} $.

 Being $ \psi_{\alpha_{p}} \equiv0$ and $ \alpha_{ p} \not\equiv0$, having in mind \cref{res:obs_strong_p_core},
 we deduce that $ \alpha_{ p} $ is not a strong $p$-core functional as in \cref{def:strong_core_function}.

 Let $ \Omega = D $ be the ambient space, $ \delta \in( n - p , n )\setminus\set*{0}$
 and $ u\tonde*{ x } \coloneqq \abs*{ x }^{- \delta / p }_{2} e_1$ for every $ x \in \Omega \setminus\set*{0}$
 and $ u\tonde*{ x } \coloneqq0$ for every $ x \in \mathbb{R}^{n} \setminus \Omega $.
 We observe that $ u \in L^{p}_{\mathrm{loc}}\tonde*{\mathbb{R}^{n};\mathbb{R}^{m}} $ but $\restr{ \nabla u }{ \Omega }\not\in L^{p}\tonde*{\Omega;\mathbb{R}^{m\times n}} $.
 Let,
 \begin{equation*}
 f\tonde*{ A } \coloneqq \strokedint_{ D} \abs*{ u\tonde*{ x } - A \tonde*{ x - \operatorname{bar}\tonde*{ D }}- u_{ D}}^{ p }_{2} \thinspace\mathrm{d} x \text{ }\forall A \in \mathbb{R}^{m\times n} .
 \end{equation*}
 We observe that $ f $ is convex and strictly positive.
 Indeed, if $ f $ vanished at some point $ A \in \mathbb{R}^{m\times n} $, then 
 $ u $ would be equal $ \operatorname{a.e.} $ in $ D $ to the  map $ l^{ A}\tonde*{\cdot- \operatorname{bar}\tonde*{ D }} - u_{ D} $.
 We also observe, using the argument in \cref{res:min_antisym}, that $ f $ is coercive.
 Therefore, $ \min\limits_{ A \in \mathbb{R}^{m\times n}} f\tonde*{ A } >0$ and for every $ \varepsilon >0$,
 \begin{equation*}\label{sa9h}
 \begin{split}
 & G^{ \alpha_{ p}}_{ \varepsilon }\tonde*{ u,\Omega } 
 \geq
 \varepsilon ^{ n - p } \alpha_{ p}\tonde*{ u , \varepsilon D } 
 =
 \varepsilon ^{ n - p } \alpha_{ p}\tonde*{ u\tonde*{ \varepsilon \cdot}, D } 
 \\
 &
 =
 \varepsilon ^{ n - p - \delta } \min\limits_{ A \in \mathbb{R}^{m\times n}} f\tonde*{ \varepsilon ^{ \delta / p } A } 
 =
 \varepsilon ^{ n - p - \delta } \min\limits_{ A \in \mathbb{R}^{m\times n}} f\tonde*{ A } ,
 \end{split}
 \end{equation*}
 thanks to \eqref{alpha_p:scaling} and the fact that $ u\tonde*{ \varepsilon x } = \varepsilon ^{- \delta / p } u\tonde*{ x } $ for every $ x \in D $.
 Therefore, 
 \begin{equation*}
 G^{ \alpha_{ p}}_{ +}\tonde*{ u,\Omega } \geq G^{ \alpha_{ p}}_{ -}\tonde*{ u,\Omega } 
 \geq
 \tonde*{ \min\limits_{ A \in \mathbb{R}^{m\times n}} f\tonde*{ A }}\tonde*{ \lim\limits_{\varepsilon\to 0^{+}} \varepsilon ^{ n - p - \delta }}=+\infty,
 \end{equation*}
 which implies that $ G^{ \alpha_{ p}}\tonde*{ u,\Omega } =+\infty$.
 Then, being $ \psi_{\alpha_{p}} \equiv0$,
 \eqref{eq:tesi_Main_BMO} and \eqref{eq:tesi_Main_BMO2} cannot hold.

 Finally, we highlight that $ \mathcal{N}_{\psi_{\alpha_{p}}} = \mathbb{R}^{m\times n} $, so that $ \mathcal{P} =\set*{0}$ trivially satisfies \eqref{eq:dir_sum_dec}
 and $ \pi_{\mathcal{P}} =0$.
 Therefore, $ D_{ \pi_{\mathcal{P}}}\tonde*{\restr{ u }{ \Omega }} =0\in L^{p}\tonde*{\Omega;\mathbb{R}^{m\times n}} $,
 but $ G^{ \alpha_{ p}}_{ -}\tonde*{ u,\Omega } = G^{ \alpha_{ p}}\tonde*{ u,\Omega } =+\infty$.
 This then shows that the implication \eqref{abcjh219h} in \cref{res:cor_Main_BMO2}
 cannot be reversed in general.

 This also proves in a different way that $ \alpha_{ p} $ is not a strong $p$-core functional, since this would be
 in contradiction with \cref{thm:Main_BMO_WL}.
\end{example}
\begin{example}
Let $ p \in \left[1,\infty\right) $, $m\geq 1$, $D\subset\mathbb{R}^n$ be a bounded open connected set with Lipschitz boundary and $ S \subset \mathbb{R}^{m} $ be a nonempty set.
 We define,
 \begin{equation*}
 \alpha_{ p}\tonde*{ u , D' } 
 \coloneqq
 \sup\limits_{\nu\in S } \strokedint_{ D'} \abs*{
 \tonde*{ u - u_{ D'}} \cdot \nu
 }^ p \thinspace\mathrm{d} x ,
 \end{equation*}
 Then, $ \alpha_{ p} : L^{p}_{\mathrm{loc}}\tonde*{\mathbb{R}^{n};\mathbb{R}^{m}} \times \mathcal{E}^{ D}_{ \Gamma }\tonde*{\mathbb{R}^{n}} \to \left[0,+\infty\right) $ is a $p$-core functional for any $ \Gamma $ subgroup of $ \mathrm{Aff}\tonde*{\mathbb{R}^{n}} $
. 
 Moreover, $\restr{ \alpha_{ p}\tonde*{\cdot, D' }}{ L^{p}\tonde*{\mathbb{R}^{n};\mathbb{R}^{m}}}: L^{p}\tonde*{\mathbb{R}^{n};\mathbb{R}^{m}} \to \left[0,+\infty\right) $
 is lower semicontinuous for every $ D' \in \mathcal{E}^{ D}_{ \Gamma }\tonde*{\mathbb{R}^{n}} $ being supremum of lower semicontinuous functionals.
 
In this case, it is easy to observe that $ \mathcal{N}_{\psi_{\alpha_{p}}} =\set*{ A \in \mathbb{R}^{m\times n} : A \tonde*{ \mathbb{R}^{n}}^{\perp}\supset S }$. Thus, if $S$ contains $m$ linearly independent vectors, then $\mathcal{N}_{\psi_{\alpha_p}}=\{0\}$ and can be used to characterize $W^{1,p}$. Otherwise, the null set is non trivial and can be used to approximate larger Sobolev-type spaces.
\end{example}

\bibliographystyle{amsalpha}
\bibliography{biblio}
 \end{document}